\documentclass[11pt, twoside, leqno]{article}
\usepackage{enumerate}
\usepackage{graphics}
\usepackage{graphicx}
\usepackage{amssymb}
\usepackage{amsmath}
\usepackage{amsthm}
\usepackage{color}
\usepackage{xcolor}
\usepackage{mathrsfs}
\usepackage{titletoc}

\usepackage{indentfirst}
\usepackage{txfonts}
\usepackage{verbatim}
\usepackage{color}
\allowdisplaybreaks

\def\red{\color{red}}

\def\rr{{\mathbb R}}
\def\rn{{\mathbb{R}^n}}

\def\zz{{\mathbb Z}}
\def\cc{{\mathbb C}}

\def\nn{{\mathbb N}}

\def\ca{{\mathcal A}}

\def\cm{{\mathcal M}}

\def\dfrac{\displaystyle\frac}

\def\cs{{\mathcal S}}

\def\fz{\infty }
\def\az{\alpha}
\def\bz{\beta}

\def\lz{\lambda}

\def\tz{\theta}

\def\lf{\left}
\def\r{\right}

\def\ls{\lesssim}

\def\noz{\nonumber}
\def\wz{\widetilde}

\def\com{\complement}

\def\dist{\mathop\mathrm{\,dist\,}}

\def\supp{\mathop\mathrm{\,supp\,}}

\def\XXint#1#2#3{{\setbox0=\hbox{$#1{#2#3}{\int}$ }
\vcenter{\hbox{$#2#3$ }}\kern-.6\wd0}}

\DeclareMathOperator{\esssup}{ess\,sup}
\DeclareMathOperator{\essinf}{ess\,inf}

\def\ga{\gamma}

\def\f{\frac}

\def\({\left(}
\def \){ \right)}

\def\lz{{\lambda}}

\def\BB{{\mathbb B}}

\newcommand{\wt}{\widetilde}

\newcommand{\p}{\partial}

\def\al{\alpha}

\def\supp{\operatorname{supp}}

\def\BB{\mathbb{B}}

\def\ld{\lambda}

\newcommand{\bp}{ \begin{proof} }

\newcommand{\ep}{ \end{proof} }

\def\XXint#1#2#3{{\setbox0=\hbox{$#1{#2#3}{\int}$ }
\vcenter{\hbox{$#2#3$ }}\kern-.6\wd0}}

\newtheorem{theorem}{Theorem}[section]
\newtheorem{lemma}[theorem]{Lemma}

\newtheorem{assumption}[theorem]{Assumption}
\newtheorem{proposition}[theorem]{Proposition}

\newtheorem{Ellipticity Condition}[theorem]{Ellipticity Condition}
\newtheorem{Strong Ellipticity Condition}[theorem]{Strong Ellipticity Condition}
\theoremstyle{definition}
\newtheorem{remark}[theorem]{Remark}
\newtheorem{definition}[theorem]{Definition}
\renewcommand{\appendix}{\par
\setcounter{section}{0}%
\setcounter{subsection}{0}%
\setcounter{subsubsection}{0}%
\gdef\thesection{\@Alph\c@section}%
\gdef\thesubsection{\@Alph\c@section.\@arabic\c@subsection}%
\gdef\theHsection{\@Alph\c@section.}%
\gdef\theHsubsection{\@Alph\c@section.\@arabic\c@subsection}%
\csname appendixmore\endcsname
}

\numberwithin{equation}{section}

\begin{document}
\title{\bf\Large Maximal Function and
Riesz Transform Characterizations of
Hardy Spaces Associated with
Homogeneous Higher Order  Elliptic
Operators and  Ball Quasi-Banach
Function Spaces
\footnotetext{\hspace{-0.35cm} 2020 {\it
Mathematics Subject Classification}.
Primary 42B30;
Secondary 42B25, 47B06, 35J30, 42B35.
\endgraf {\it Key words and phrases.}
ball quasi-Banach function space, high order elliptic operator, Hardy
space, maximal function, Riesz transform.
\endgraf
This project is partially supported by the National Key Research and
Development Program of China
(Grant No. 2020YFA0712900), the National
Natural Science Foundation of China (Grant Nos.
11971058, 12071197, 12122102,
11871100, 11871254 and 12071431)
and the Fundamental Research Funds for the
Central Universities (Grant No. lzujbky-2021-ey18).}}
\author{Xiaosheng Lin, Dachun
Yang\footnote{Corresponding
author, E-mail: \texttt{dcyang@bnu.edu.cn}/{\red July 7, 2022}/Final version.},
\ Sibei Yang and Wen Yuan}
\date{}
\maketitle

\vspace{-0.7cm}

\begin{center}
\begin{minipage}{13cm}
{\small {\bf Abstract}\quad
Let $L$ be a homogeneous divergence form higher
order elliptic operator with complex
bounded measurable coefficients on
$\mathbb{R}^n$ and $X$ a ball
quasi-Banach function space on
$\mathbb{R}^n$ satisfying
some mild assumptions. Denote by
$H_{X,\, L}(\mathbb{R}^n)$
the Hardy space, associated with
both $L$ and $X$,
which is defined via the Lusin area
function related to
the semigroup generated by $L$.
In this article,
the authors establish both
the maximal function and the
Riesz transform characterizations
of $H_{X,\, L}(\mathbb{R}^n)$.
The results obtained in this article
have a wide range of generality
and can be applied to the weighted
Hardy space, the variable Hardy space,
the mixed-norm Hardy space, the
Orlicz--Hardy space, the Orlicz-slice
Hardy space, and the Morrey--Hardy
space, associated with $L$.
In particular, even when $L$ is a
second order divergence
form elliptic operator, both the
maximal function and
the Riesz transform characterizations
of the mixed-norm Hardy space, the
Orlicz-slice Hardy space, and the
Morrey--Hardy space, associated
with $L$, obtained in this
article, are totally new.}
\end{minipage}
\end{center}

\vspace{0.3cm}

\tableofcontents

\vspace{0.4cm}

\section{Introduction}
The real-variable theory of the
classical Hardy space $H^p(\rn)$
with $p\in(0,1]$, creatively  initiated
by Stein and Weiss  \cite{SW60} and
then further seminally
developed by Fefferman and
Stein  \cite{FS72}, plays a key role
in both harmonic analysis and partial
differential equations (see, for
instance, \cite{CLMS93,FS72,S94}
and the references therein).
It is well known that, when $p\in(0,1]$,
the Hardy space $H^p(\rn)$ is a good substitute
of the Lebesgue space $L^p(\rn)$ in
the study of the boundedness of
some classical operators.
For instance, when $p\in(0,1]$,
the Riesz transform is bounded on
$H^p(\rn)$ but not on $L^p(\rn)$.
However, there exist many settings
in which the real-variable theory of
the Hardy space can not be applicable;
for instance, the Riesz transform
$\nabla L^{-1/2}$ may not be bounded
from $H^1(\mathbb R^n)$ to
$L^1(\mathbb R^n)$ when
$L:=-{\rm div}(A\nabla)$ is a second
order divergence elliptic operator with
complex bounded measurable
coefficients (see, for instance,
\cite{HM09,hmm11}).

In recent years, there has been a lot
of studies which pay attention to the
real-variable theory of Hardy
spaces associated with operators. Here,
we give a brief overview of this
research field. Auscher et al.
\cite{ADM05} first introduced  the
Hardy space $H^1_{L}(\rn)$ associated
with an operator $L$ whose heat kernel
satisfies a  pointwise Gaussian upper
bound estimates and obtained its
molecular characterization.
Moreover, Duong and Yan
\cite{DY051,DY05} established the
dual theory of the Hardy space
$H^1_L(\rn).$ Yan \cite{Y08} further
generalized these results to the
Hardy space $H^p_L(\rn)$ with certain
$p\in(0,1].$ Furthermore, Hardy spaces
associated operators satisfying the
Davies--Gaffney estimates which are
weaker than the Gaussian upper bound
estimates were studied in
\cite{A08JGE,CMY17,CMY16,DL13,HMMY11,HM09,hmm11}.
Meanwhile, the real-variable theory of
various variants of Hardy spaces
associated with operators
were developed; see, for instance,
\cite{B14,BCKYY131,SY10} for weighted
Hardy spaces
associated with operators,
\cite{BCKYY13,BL11,JY10,JY11,YY14} for
(Musielak--)Orlicz--Hardy spaces
associated with operators, and
\cite{YZ16,YZZ18} for variable Hardy
spaces associated with operators. For
more studies on Hardy spaces
and other function spaces associated
with operators, we refer the reader to
\cite{BDL18,BDL20,GKKP19,GKKP17,gn17,
LW17,A02,SY18,YY13,sy15} and the references therein.

In particular, let $L$ be a homogeneous
divergence form higher order elliptic
operator with complex
bounded measurable coefficients on
$\mathbb{R}^n$ (see, for
instance, \cite{CMY16} or Subsection
\ref{juhuatai} below for its detailed
definition). For the Hardy space
$H^p_L(\rn)$ with $p\in(0,1],$
Cao and the second author of this
article \cite{CY12} established several
characterizations of
$H^p_L(\rn)$ by means of the molecule,
the square function, and the Riesz
transform. Furthermore, Deng et al.
\cite{DDY12} also established the
corresponding characterizations of the
Hardy space $H^1_L(\rn).$ Moreover, Cao
et al. \cite{CMY16} obtained the
various maximal function characterizations of $H^p_L(\rn).$

Let $X$ be a ball quasi-Banach function
space on $\rn$ (see, for instance,
\cite{SHYY17} or Subsection \ref{feng1}
below for its detailed definition) and
$L$ a homogeneous divergence form
higher order elliptic operator with
complex bounded measurable coefficients
on $\mathbb{R}^n$. Sawano et al.
\cite{SHYY17} introduced the Hardy space
$H_{X,\,L}(\rn)$ associated with both
$X$ and $L$, via the Lusin area
function related to the semigroup
generated by $L$, and established its
molecular characterization. \emph{It is then quite natural
to ask whether or not there exist the maximal
function and the Riesz
transform characterizations of $H_{X,\,
L}(\rn)$}.

Recall that
ball quasi-Banach function spaces  were
introduced in \cite{SHYY17} to include
more important function spaces than
quasi-Banach function spaces originally
introduced in \cite[Definitions 1.1
and 1.3]{BS88}. Indeed, the former
includes (weighted) Lebesgue spaces,
variable Lebesgue spaces, mixed-norm
Lebesgue spaces, Orlicz spaces,
Orlicz-slice spaces,
and Morrey spaces which, except
Lebesgue spaces, are usually not
quasi-Banach function spaces
(see, for instance, \cite{SHYY17,WYY20,ZWYY20}).

The main \emph{purpose} of
this article is to answer the aforementioned question,
namely, to establish both the
maximal function and the Riesz transform
characterizations of the Hardy space
$H_{X,\,L}(\rn)$ associated with both
$X$ and $L$. More precisely, we obtain
the radial and the non-tangential
maximal function characterizations
of the space $H_{X,\,L}(\rn)$ under
some mild assumptions on both $X$ and
$L$. Moreover, we prove that the Riesz
transform $\nabla^m L^{-1/2}$
associated with $L$ is bounded from
$H_{X,\,L}(\rn)$ to $H_X(\rn)$ and
further establish the Riesz transform
characterization of $H_{X,\,L}(\rn)$.
Here, $H_X(\rn)$ denotes the Hardy
space associated with $X$ introduced in
\cite{SHYY17}. Furthermore, some
applications of the main results
obtained as above to some concrete
function spaces are given. It is worth
pointing out that, when $X$ is just the
classical Lebesgue space, both the
maximal function and the Riesz transform
characterizations of  $H_{X,\,L}(\rn)$
were established in \cite{CMY16,CY12};
however, when $L$ is a general
homogeneous divergence form higher
order elliptic operator with complex
bounded measurable coefficients
on $\rn$ (except a second order divergence
form elliptic operator) and $X$ is one
of the weighted Lebesgue space, the
variable Lebesgue space, the mixed-norm
Lebesgue space, the Orlicz space, the
Orlicz-slice space, or the Morrey space,
both the maximal function and the Riesz
transform characterizations of
$H_{X,\,L}(\rn)$ are totally new.

Compared with the second order
divergence form elliptic operator, the
argument arose in the  homogeneous
divergence  form higher order elliptic
operator $L$ is more complicated.
Because of this and the deficiency of
the explicit expression of the norm for
the ball quasi-Banach
function space $X$, the methods used in
\cite{CMY16} are no longer applicable.
To overcome these difficulties,
in this article, by the Sobolev
embedding theorem, we deal with the
higher derivative caused by the higher
order elliptic operator. Meanwhile,  by
taking full advantage of the divergence
form elliptic operator, the
Besicovitch covering lemma, and the
Caccioppoli inequality, we first
establish a  good-$\lambda$ inequality
concerning both the non-tangential maximal
function and the truncated Lusin area
function associated with the
heat semigroup of $L$ in Lemma
\ref{good} below. Using this
good-$\lambda$ inequality and
borrowing some ideas from the proof of
the  extrapolation theorem in
the scale of Banach function spaces
(see, for instance, \cite{cmp11}),
we then establish the maximal function
characterization of the Hardy space
$H_{X,\, L}(\rn)$ associated with both
$X$ and $L.$ Moreover, to obtain the
Riesz transform characterization of
$H_{X,\,L}(\rn)$, following an
approach similar to that used in \cite{hmm11,CY12},
we need to introduce the homogeneous
Hardy--Sobolev space associated with
the ball quasi-Banach function space $X$
and establish its atomic
characterization. When $X$ is just the
Lebesgue space, the atomic
characterization of the homogeneous
Hardy--Sobolev space associated
with $X$ can be obtained easily by the
connection between the classical homogeneous
Hardy--Sobolev space and the
homogeneous Triebel--Lizorkin space.
However, this is impossible when $X$ is
a general ball quasi-Banach function
space. Motivated by \cite{LY07,GL},
we directly establish the atomic
decomposition of the homogeneous
Hardy--Sobolev space associated with
the ball quasi-Banach function space
$X$ via both the atomic decomposition of the
$X$-tent space
$T_X(\mathbb{R}^{n+1}_+)$ and the
Calder\'on reproducing formula. Then,
using the atomic characterization
of the homogeneous Hardy--Sobolev space
associated with $X$, we finally obtain
the Riesz transform characterization
of $H_{X,\,L}(\rn)$

The remainder of this article is
organized as follows.

In Section \ref{pre}, we state some
known results on the homogeneous
divergence form higher order elliptic
operator $L$, the ball quasi-Banach
function space $X$, and the
Hardy space $H_{X,\,L}(\rn)$ associated
with both $X$ and $L$.

In Section \ref{max}, we establish the
radial and the non-tangential maximal
function characterizations
of the Hardy space $H_{X,\, L}(\rn)$
(see Theorem \ref{th2} below). By
making full use of the special
structure of the divergence form
elliptic operator, the Caccioppoli
inequality, the Besicovitch covering
lemma, and the Sobolev embedding
theorem, we first establish a
good-$\lambda$ inequality concerning
the non-tangential maximal function and
the truncated Lusin area function
associated with the heat semigroup
generated by $L$,
which plays an important role in the
proof of Theorem \ref{th2} (see Lemma
\ref{good} below).
Moreover, to overcome the essential
difficulty caused by the deficiency of
the explicit  expression of the norm of
the ball quasi-Banach function space
$X$, we cleverly use the extrapolation
theorem in the scale of ball Banach
function spaces (see, for instance,
\cite{cmp11} or Lemma \ref{waicha} below).

Section \ref{Riesz} is devoted to
establishing the Riesz characterization of
the Hardy space $H_{X,\,L}(\rn)$
(see Theorems \ref{r} and \ref{riesz}
below). In Subsection \ref{s1}, with
the help of the molecular
characterization of both the Hardy
spaces $H_X(\rn)$ and $H_{X,\,L}(\rn)$,
we prove that the Riesz transform
$\nabla^m L^{-1/2}$ is bounded
from $H_{X,\,L}(\rn)$ to $H_X(\rn)$
under some mild assumptions
on the ball quasi-Banach space $X$ (see
Theorem \ref{r} below). Furthermore, in
Subsection \ref{s2}, motivated
by \cite{LY07,GL}, we first introduce the
homogeneous Hardy--Sobolev space
$\dot{H}_{m,\,X}(\rn)$ associated
with the ball quasi-Banach function
space $X$ (see Definition
\ref{zhanshen} below for the details).
Then,  using both
the atomic decomposition of
the $X$-tent space
$T_X(\mathbb{R}^{n+1}_+)$ and the
Calder\'on reproducing formula, we
establish the atomic decomposition of
the homogeneous Hardy--Sobolev space
$\dot{H}_{m,\,X}(\rn)$, which is
essential in the proof of Theorem
\ref{riesz}. Applying the atomic
decomposition of the homogeneous
Hardy--Sobolev space $\dot{H}_{m,\,X}(\rn)$
and following the approach used in
\cite{hmm11,CY12}, we then obtain the
Riesz transform characterization of $H_{X,\,L}(\rn)$.

In Section \ref{appl}, we give some
applications of both the maximal function
and the Riesz transform
characterizations of the space
$H_{X,\,L}(\rn)$ obtained in Sections
\ref{max} and \ref{Riesz}, respectively, to the weighted Hardy
space, the variable Hardy space, the
mixed-norm Hardy space, the
Orlicz--Hardy space,  the Orlicz-slice
Hardy space, and the Morrey--Hardy
space, associated with the operator $L$.
It is worth pointing out that, when $L$
is a general homogeneous divergence
form higher order elliptic
operator with complex bounded measurable
coefficients on $\rn$ (except a second
order divergence form elliptic
operator), both the maximal function and the
Riesz transform characterizations of
the weighted Hardy space, the variable
Hardy space, the mixed-norm Hardy
space, the Orlicz--Hardy space,
the Orlicz-slice Hardy space, and the
Morrey--Hardy space associated with the
operator $L$, obtained in this article,
are new. In particular, even when $L$
is just a second order divergence
form elliptic operator, both the
maximal function and the Riesz
transform characterizations of the
mixed-norm Hardy space, the Orlicz-slice
Hardy space, and the Morrey--Hardy
space associated with $L$, established
in this article, are totally new.

Due to the generality and the
practicability, more applications of these 
main results of this
article are predictable.

Finally, we make some conventions on
notation.
Let $\nn:=\{1,2,\ldots\}$,
$\zz_+:=\nn\cup\{0\}$, and $\zz_+^n:=(\zz_+)^n$.
We always denote by $C$ a
\emph{positive constant} which is
independent of the main parameters,
but it may vary from line to line. We
also use $C_{(\alpha,\beta,\ldots)}$ to
denote a positive constant depending on
the indicated parameters $\alpha,\,\beta,\,\ldots.$
The \emph{symbol} $f\lesssim g$ means
that $f\le Cg$. If $f\lesssim g$
and $g\lesssim f$, we then write $f\sim
g$. If $f\le Cg$ and $g=h$ or $g\le h$,
we then write $f\ls g\sim h$ or $f\ls
g\ls h$, \emph{rather than} $f\ls g=h$
or $f\ls g\le h$. We use $\mathbf{0}$
to denote the \emph{origin} of $\rn$.
For any measurable subset $E$ of $\rn$,
we denote by $\mathbf{1}_E$ its
characteristic function. For any
$\alpha\in(0,\infty)$ and any ball
$B:=B(x_B,r_B)$ in $\rn$, with
$x_B\in\rn$ and $r_B\in(0,\infty)$, let
$\alpha B:=B(x_B,\alpha r_B)$.
Denote by ${\mathcal Q}$ the \emph{set}
of all the cubes having their edges
parallel to the coordinate axes. For
any $j\in\nn$ and any ball
$B\subset\rn$, let
$U_j(B):=(2^{j+1}B)\setminus(2^jB)$
and $U_0(B):=2B.$ Denote the
\emph{differential operator}
$\frac{\p^{|\al|}}{\p 	
x_1^{\al_1}\cdots\p x_n^{\al_n}}$
simply by $\p^\al,$ where
$\al:=(\al_1,\ldots,\al_n)$
and $|\al|:=\al_1+\cdots\al_n$.
Denote by $\cs(\rn)$ the \emph{space of
all Schwartz functions}, equipped with
the well-known topology determined by a
countable family of norms,  and
by $\cs'(\rn)$ its \emph{dual space},
equipped with the weak-$\ast$ topology
(namely, the \emph{space of all
tempered distributions}).
For any  $f\in\mathcal{S}'(\rn)$ and
$i\in\{1,\ldots,n\}$, denote
$\frac{\p f}{\p x_i}$ by $\p _i f$, and
let
$$\nabla f:=(\p_1 f,\ldots,\p_n f).$$
Let
$\mathbb{R}^{n+1}_+:=\rn\times(0,\infty).$
For any sets $E,F\subset\rn$, let
$$\dist(E,F):=\inf_{x\in E,\,y\in F}|x-y|.$$
Finally, for any given $q\in[1,\infty]$,
we denote by $q'$ its \emph{conjugate
exponent}, namely, $1/q+1/q'=1$.

\section{Preliminaries}\label{pre}

In this section, we state some known
concepts and facts on the homogeneous
higher order elliptic operator $L$,
the ball quasi-Banach function space
$X$, and the Hardy space $H_{X,\,L}(\rn)$
associated with both $X$ and $L$.

\subsection{Homogeneous Higher Order Elliptic Operators}\label{juhuatai}

In this subsection, we recall the
definition and some properties of the
homogeneous divergence form higher
order elliptic operator $L$.

Let $m\in \mathbb{N}$ and
$\dot{W}^{m,2}(\rn)$ be the
$m$-\emph{order homogeneous Sobolev
space} equipped with the \emph{norm}
\begin{align*}
\|f\|_{\dot{W}^{m,2}(\rn)}:=\left[\sum_{
|\al|=m}\|\partial^\al
f\|_{L^2(\rn)}^2\right]^{1/2}<\infty,
\end{align*}
where, for any given $\al\in\zz_+^n$
with $|\al|=m$, $\partial^\al f$
denotes the $m$-order derivative
in the sense of distributions. For any
multi-indices $\az,\bz\in\zz_+^n$
satisfying $|\az|=m=|\bz|$,
let $a_{\az,\,\bz}$ be a complex
bounded measurable function on $\rn$.
For any given $f, g\in\dot{W}^{m,2}(\rn),$
define the \emph{sesquilinear form}
$a_0,$ mapping $\dot{W}^{m,2}(\rn)\times
\dot{W}^{m,2}(\rn)$ into $\mathbb{C},$ by setting
\begin{align*}
a_0(f,g):=\sum_{|\al|=|\beta|=m}
\int_\rn a_{\al,\beta}(x)\partial^
\beta f(x) \overline{\partial ^\al g(x)}\,dx.
\end{align*}
The following both ellipticity
conditions on $\{a_{\al,\,\beta}\}
_{|\al|=|\beta|=m}$ are necessary.

\begin{Ellipticity Condition}\label{ec}
Let $m\in\nn.$ There  exist constants
$0<\lambda_0\leq \Lambda_0<\infty$
such that, for any $f,g\in \dot{W}^{m,2}(\rn)$,
\begin{align*}
a_0(f,g)\leq \Lambda_0 \|\nabla ^m
f\|_{L^2(\rn)}\|\nabla ^m g\|_{L^2(\rn)}
\end{align*}
and
\begin{align*}
\Re(a_0(f,g))\geq \lambda_0 \|\nabla ^m
f\|_{L^2(\rn)}^2.
\end{align*}
Here and thereafter, for any $z\in\cc$, $\Re z$
denotes the \emph{real part} of $z$ and
\begin{align*}
\|\nabla ^m f\|_{L^2(\rn)}:=\left[\sum_{|\al|=m}
\int_\rn\left|\partial ^\al f(x)
\right|^2\,dx\right]^{1/2}.
\end{align*}
\end{Ellipticity Condition}

\begin{Strong Ellipticity Condition}\label{sec}
Let $m\in\nn.$ There exists a positive
constant $\lambda_1$ such that,
for any $\xi:=\{\xi_\al\}_{|\al|=m}$
with $\xi_\al\in \mathbb{C},$ and
for almost every $x\in\rn,$
\begin{align*}
\Re\left\{\sum_{|\al|=|\beta|=m}
a_{\al,\,\beta}(x)\xi_\beta\overline{
\xi_\al}\right\}\geq \lambda_1
|\xi|^2=\lambda_1\left\{\sum_{|\al|=m}
|\xi_\al|^2\right\}.
\end{align*}
\end{Strong Ellipticity Condition}

It is easy to prove that Strong Ellipticity
Condition \ref{sec}  implies Ellipticity
Condition \ref{ec}. However, the equivalence
between  Ellipticity Condition \ref{ec} and
Strong Ellipticity Condition \ref{sec} is
only a specific feature of second order
divergence form elliptic operators (see, for
instance, \cite[Remark 2.1]{CMY16}), which is
no longer true for divergence form elliptic
operators of order greater than two.

Assume that, for any $\az,\bz\in\zz_+^n$
satisfying $|\az|=m=|\bz|$, $a_{\az,\,\bz}$
is a complex bounded measurable
function on $\rn$. Assume further that the
sesquilinear form $a_0$ satisfies Ellipticity
Condition \ref{ec}. Then it is well known that
there exists a \emph{densely defined
operator $L$} in $L^2(\rn)$
associated with $a_0$ (see, for instance,
\cite[p.\,830]{CMY16}), which is formally written as
\begin{align}\label{high}
L:=\sum_{|\al|=m=|\beta|}(-1)^m\partial
^{\al}\lf(a_{\al,\,\beta}\partial^\beta\r).
\end{align}
Usually, we call $L$ a \emph{homogeneous
divergence form $2m$-order elliptic
operator} on $\rn$.

Denote by $(p_{-}(L),p_{+}(L))$ the
\emph{interior of the maximal interval
of exponents} $p\in[1,\infty]$,
for which the family $\{e^{-tL}\}_{t\in
(0,\infty)}$ of operators is bounded
on $L^p(\rn).$ Meanwhile, denote by
$(q_{-}(L),q_{+}(L))$ the \emph{interior
of the maximal interval of exponents}
$q\in[1,\infty]$ such that
the family $\{\sqrt{t}\nabla^m e^{-tL}\}
_{t\in(0,\infty)}$ of operators is
bounded on $L^q(\rn).$
For the exponents $p_{-}(L),\,p_{+}(L),
\,q_{-}(L),$ and $q_{+}(L)$, we have
the following conclusions
(see, for instance, \cite[p.\,67]{A02}
and \cite[Proposition 2.5]{CMY16}) .

\begin{proposition}\label{basic}
Let $n,m\in\nn$ and $L$ be a homogeneous
divergence form $2m$-order elliptic
operator on $\rn$
in \eqref{high} satisfying Ellipticity
Condition \ref{ec}. Then the
following conclusions hold true:
\begin{itemize}
\item [$(\mathrm{i})$] It holds true that
\begin{align*}
\begin{cases}	
(p_-(L),p_+(L))=(1,\infty)&\text{if}\ n\leq2m,\\
\displaystyle
\left[
\frac{2n}{n+2m},\frac{2n}{n-2m}\right]
\subset(p_-(L),p_+(L))&\text{if}\ n>2m.
\end{cases}
\end{align*}
\item [$(\mathrm{ii})$] $q_-(L)=p_-(L)$
and $q_+(L)\in(2,\infty).$
\item [$(\mathrm{iii})$]For any $k\in\zz_+$
and $p_-(L)<p\leq q\leq p_+(L),$ the family
$\{(tL)^ke^{-tL}\}_{t\in(0,\infty)}$ of
operators satisfies the following
\emph{$m-L^p-L^q$ off-diagonal estimates}:
There exist positive constants $C_1$ and
$C_2$ such that, for any closed sets $E$
and $F$ of $\rn,$ any $t\in(0,\infty)$,
and any $f\in L^2(\rn)\cap L^p(\rn)$
supported in $E,$
\end{itemize}
\begin{align*}
\left\|(tL)^ke^{-tL}(f)\right\|_{L^q(F)}
\leq C_1 t^{\frac{n}{2m}(\frac{1}{q}-\f 1 p)}
\exp {\left\{-C_2\frac{[\dist(E,F)]^{\frac{2m}{2m-1}
}}{t^{\frac{1}{2m-1}}}\right\}}\|f\|_{L^p(E)}.
\end{align*}
\end{proposition}

Furthermore, we also have the following
\emph{Caccioppoli inequality} which was
obtained in \cite[Proposition 3.2]{CMY16}.

\begin{proposition}\label{cacci}
Let $m\in\nn$ and $L$  be a homogeneous
divergence form $2m$-order elliptic operator
 in \eqref{high} satisfying Ellipticity
Condition \ref{ec}. For any $f\in L^2(\rn),$
$x\in\rn,$ and $t\in(0,\infty),$
let $u(x,\, t):=e^{-t^{2m}L}(f)(x).$ Then
there exists a positive constants $C$, independent
of $f$, such that, for any $x_0\in\rn,$
$r\in(0,\infty),$ and $t_0\in(3r,\infty),$
\begin{align*}
\int_{t_0-r}^{t_0+r}\int_{B(x_0,r)}|\nabla
^m u(x,t)|^2\, dx \, dt\leq \frac{C}{r^{2m}}
\int_{t_0-2r}^{t_0+2r}\int_{B(x_0,2r)}|
u(x,t)|^2\, dx \, dt.
\end{align*}
\end{proposition}	

\subsection{Ball Quasi-Banach Function Spaces}\label{feng1}

In this subsection, we recall some
preliminaries on ball quasi-Banach
function spaces introduced
in \cite{SHYY17}. Denote by the
\emph{symbol} ${\mathscr M}(\rn)$ the
set of all measurable functions
on $\rn$. For any given $x\in\rn$ and
$r\in(0,\infty)$, let $B(x,r):=\{y\in\rn:\ |x-y|<r\}$ and
\begin{align}\label{4}
\mathbb{B}:=\{B(x,r):\ x\in\rn\ \text{and}\ r\in(0,\infty)\}.
\end{align}

\begin{definition}\label{Debqfs}
A quasi-Banach space $X\subset{
\mathscr M}(\rn)$, equipped with
a quasi-norm
$\|\cdot\|_X$ which makes sense for
all functions in ${\mathscr M}(\rn)$,
is called a \emph{ball quasi-Banach
function space} if it satisfies that
\begin{itemize}
\item[(i)] for any $f\in {\mathscr M}(\rn)$,
$\|f\|_X=0$ implies that $f=0$ almost everywhere;
\item[(ii)] for any $f,g\in {\mathscr M}(\rn)$,
$|g|\le |f|$ almost everywhere implies
that $\|g\|_X\le\|f\|_X$;
\item[(iii)] for any $\{f_m\}_{m\in\nn}
\subset {\mathscr M}(\rn)$
and $f\in {\mathscr M}(\rn)$, $0\le f_m\uparrow f$
almost everywhere as $m\to\fz$ implies that
$\|f_m\|_X\uparrow\|f\|_X$ as $m\to\fz$;
\item[(iv)] $B\in\BB$ implies that $\mathbf
{1}_B\in X$, where $\BB$ is the same as in \eqref{4}.
\end{itemize}
Moreover, a ball quasi-Banach function space
$X$ is called a \emph{ball Banach function space}
if the norm of $X$ satisfies the triangle
inequality: for any $f,g\in X$,
\begin{align*}
\|f+g\|_X\le \|f\|_X+\|g\|_X,
\end{align*}
and that, for any $B\in \BB$, there exists a
positive constant $C_{(B)}$, depending on $B$, such that,
for any $f\in X$,
\begin{equation*}\label{eq2.3}
\int_B|f(x)|\,dx\le C_{(B)}\|f\|_X.
\end{equation*}
\end{definition}

\begin{remark}
\begin{itemize}
\item[{\rm(i)}] Let $X$ be a ball
quasi-Banach function space on $\rn$.
By \cite[Remark 2.6(i)]{yhyy21},
we conclude that, for any $f\in {\mathscr M}
(\rn)$, $\|f\|_X=0$ if and only if $f=0$
almost everywhere.
\item[{\rm(ii)}] As was mentioned in
\cite[Remark 2.6(ii)]{yhyy21},
we obtain an equivalent formulation of
Definition \ref{Debqfs}
via replacing any ball $B$ by any bounded
 measurable set $E$ therein.
\item[{\rm(iii)}] In Definition \ref{Debqfs},
if we replace any ball $B$ by any measurable
set E with $|E|<\infty$,
we obtain the definition of (quasi-)Banach
function spaces originally introduced in
\cite[Definitions 1.1 and 1.3]{BS88}.
Thus, a (quasi-)Banach function space
is always a ball (quasi-)Banach function space.
\item[{\rm(iv)}] By \cite[Theorem 2]{dfmn21},
we conclude that both (ii) and (iii)
of Definition \ref{Debqfs} imply that
any ball quasi-Banach function space is complete.
\end{itemize}
\end{remark}

It is known that Lebesgue spaces, Lorentz
spaces, variable Lebesgue spaces, and Orlicz
spaces are (quasi-)Banach function spaces
(see, for instance, \cite{SHYY17,WYY20,ZWYY20}).
However, weighted Lebesgue spaces,
mixed-norm Lebesgue spaces, Orlicz-slice
spaces, and Morrey spaces are
not necessary to be quasi-Banach function
spaces (see, for instance \cite{SHYY17,WYY20,ZWYY20,ZYYW}
for more details and examples). Therefore,
in this sense, compared with the concept
of ball (quasi-)Banach  function spaces,
the concept of (quasi-)Banach
function spaces is more restrictive. Based on this,
the ball (quasi-)Banach function spaces were
originally introduced in \cite{SHYY17}
to extend (quasi-)Banach function spaces
further so that weighted Lebesgue spaces,
mixed-norm Lebesgue spaces, Orlicz-slice
spaces, and Morrey spaces are
also included in this more generalized framework.

The following concept of the associate space
of a ball Banach function space can be found in
\cite[Section 2.1]{SHYY17}.
\begin{definition}\label{def-X'}
For any ball Banach function space $X$,
the \emph{associate space} (also called the
\emph{K\"othe dual}) $X'$ is defined by setting
\begin{equation*}
X':=\lf\{f\in{\mathscr M}(\rn):\ \|f\|_{X'}
:=\sup_{\{g\in X:\ \|g\|_X=1\}}\|fg\|
_{L^1(\rn)}<\infty\r\},
\end{equation*}
where $\|\cdot\|_{X'}$ is called the
\emph{associate norm} of $\|\cdot\|_X$.
\end{definition}

\begin{remark}\label{bbf}
By \cite[Proposition 2.3]{SHYY17}, we
find that, if $X$ is a ball Banach function
space, then its associate space $X'$
is also a ball Banach function space.
\end{remark}

\begin{definition}\label{Debf}
Let $X$ be a ball quasi-Banach function
space and $p\in(0,\infty)$.
The \emph{$p$-convexification} $X^{p}$ of
$X$ is defined by setting
$X^{p}:=\{f\in\mathscr M(\rn):\ |f|^{p}
\in X\}$ equipped with
the \emph{quasi-norm} $\|f\|_{X^{p}}:=
\|\,|f|^{p}\,\|_{X}^{1/p}$.
\end{definition}

\begin{definition}\label{defi:2-1-4}
Let $X$ be a ball quasi-Banach function
space. A function $f\in X$ is said to have an
\emph{absolutely continuous quasi-norm} in
$X$ if $\|f\mathbf{1}_{E_j}\|_X\downarrow 0$
whenever $\{E_j\}_{j=1}^\infty$ is a sequence
of measurable sets that satisfy $E_j \supset E_{j+1}$
for any $j \in \nn$ and $\bigcap_{j=1}^\infty
E_j = \emptyset$. Moreover, $X$ is said to
have an \emph{absolutely continuous
quasi-norm} if, for any $f\in X$, $f$
have an absolutely continuous quasi-norm in $X$.
\end{definition}

It is worth pointing out that the Lebesgue
space $L^\fz(\rn)$ and the Morrey space
$\cm^p_q(\rn)$ with $1\le q<p<\fz$ do
\emph{not} have an absolutely continuous
norm (see, for instance, \cite[p.\,10]{SHYY17}).

Denote by $L^{1}_{\rm loc}(\rn)$ the set of
\emph{all locally integrable functions on}
${\mathbb R}^n$.
Recall that the \emph{Hardy--Littlewood
maximal operator} $\cm$ is  defined by setting,
for any $f\in L^1_{\mathrm{loc}}(\rn)$ and $x\in\rn$,
\begin{equation}\label{jidahanshu}
\cm  (f)(x):=\sup_{r\in(0,\fz)}\frac{1
}{|B(x,r)|}\int_{B(x,r)}|f(y)|\,dy.
\end{equation}
For any $\theta\in(0,\infty)$, the \emph{powered
Hardy--Littlewood maximal operator}
$\cm^{(\theta)}$ is defined by setting,
for any $f\in L^{1}_{\rm loc}(\rn)$ and $x\in\rn$,
\begin{equation*}
\cm^{(\theta)}(f)(x):=\lf[\cm\lf(|f|^\theta
\r)(x)\r]^{\frac{1}{\theta}}.
\end{equation*}

To study the Hardy  space $H_{X,\,L}(\rn)$
(see Definition \ref{defi:6-1-2} below), we
need the following assumption on $X$.

\begin{assumption}\label{vector}
Let $X$ be a ball quasi-Banach function space
on $\rn$. Assume that there exist constants
$\theta,s\in(0,1]$ and $C\in(0,\infty)$ such
that, for any $\{f_j\}_{j=1}^\infty\subset
\mathscr{M}(\rn),$
\begin{align*}
\left\|\left\{\sum_{j=1}^\infty\left[\cm
^{(\theta)}(f_j)\right]^s\right\}^{1/s}
\right\|_X\leq C\left\|\left\{\sum_{j=1}^
\infty|f_j|^s\right\}^{1/s}\right\|_X.
\end{align*}
\end{assumption}
\begin{assumption}\label{vector2}
Let $s\in(0,1].$ Assume that  $X^{1/s}$ is a
ball Banach function space on $\rn$  and there
exists a positive constant  $q\in(1,\infty]$
such that, for any $f\in[(X^{1/s})']^{1/(q/s)'}$,
\begin{equation*}
\lf\|\cm(f)\r\|_{[(X^{1/s})']^{1/(q/s)'}}
\le C\|f\|_{[(X^{1/s})']^{1/(q/s)'}},
\end{equation*}
where $C$ is a positive constant independent of $f$.	
\end{assumption}

By a proof similar to that of \cite[Theorem
2.11]{SHYY17}, we have the following conclusion; we omit the details.
\begin{lemma}\label{thm:2-2-3}
Assume that $X$ is a ball quasi-Banach
function space satisfying both Assumptions \ref{vector} and
\ref{vector2} for some $\theta,s\in
(0,1]$ and  $q\in(1,\infty].$  Let $\tau
\in(n[1/\theta-1/q],\infty),$
$\{Q_j\}_{j=1}^\infty \subset {\mathcal Q}$,
and both $\{m_j\}_{j=1}^\infty \subset L^q(\rn)$
and $\{\lambda_j\}_{j=1}^\infty \subset
[0,\infty)$ satisfy that, for any $j\in\nn$
and $k\in\zz_+$,
$$
\lf\|m_j\mathbf{1}_{U_k(Q_j)}\r\|_{L^q(\rn)}\le2^{-\tau k}\frac{|Q_j|^{1/q}}{\|\mathbf{1}_{Q_j}\|_X}
$$
and
\begin{equation*}
\left\|\left\{\sum_{j=1}^\infty\left(
\frac{\lambda_{j}}{\|\mathbf{1}_{Q_{j}}
\|_{X}}\right)^{s}\mathbf{1}_{Q_{j}}
\right\}^{\frac{1}{s}}\right\|_{X}<\infty.
\end{equation*}
Then $\sum_{j=1}^\infty \lambda_jm_j$
converges in ${\mathcal S}'(\rn)$,
$\sum_{j=1}^\infty \lambda_j|m_j|\in X,$ and
there exists a positive constant $C$,
independent of $f$, such that
$$\left\|\sum_{j=1}^\infty \lambda_j
|m_j|\right\|_{X}\le C
\left\|\left\{\sum_{j=1}^\infty\left(
\frac{\lambda_{j}}{\|
\mathbf{1}_{Q_{j}}\|_{X}}\right)^{s}
\mathbf{1}_{Q_{j}}
\right\}^{\frac{1}{s}}\right\|_{X}.
$$
\end{lemma}
\subsection{Hardy Spaces $H_{X,\, L}(\rn)$}\label{feng2}

In this subsection, we recall the definition
of the Hardy  space $H_{X,\, L}(\rn)$ which
was introduced in \cite[Section 6.1]{SHYY17}.
For any  $\al\in(0,\infty)$ and $x\in\rn$, let
\begin{align}\label{zui}
\Gamma_\al(x):=\left\{(y,t)\in\mathbb{R}^{n+1}_+:
\ \ |x-y|<\al t\right\},
\end{align}
where $\mathbb{R}^{n+1}_+:=\rn\times(0,\fz)$.
When $\al=1$, we denote $\Gamma_\al(x)$ simply
by $\Gamma(x)$. Let $m\in\nn$ and $L$ be a
homogeneous
divergence form $2m$-order elliptic operator
in \eqref{high}. For any $f\in L^2(\rn)$,
the \emph{Lusin area function} $S_L(f)$,
associated with $L$, is defined by setting,
for any $x\in\rn,$
\begin{align*}
S_L(f)(x):=\left[\int_{\Gamma(x)}\left|
t^{2m}Le^{-t^{2m}L}(f)(y)\right|^2\,\frac{dy
\,dt}{t^{n+1}}\right]^{1/2}.
\end{align*}

\begin{definition}\label{defi:6-1-2}
Let $m\in\nn$, $X$ be a ball quasi-Banach function
space on $\rn$, and $L$ a homogeneous
divergence form $2m$-order
elliptic operator  in \eqref{high}.
Then the \emph{Hardy space} $H_{X,\,L}(\rn)$,
associated with both $X$ and $L$, is defined as
the completion of the set
$$
H_{X,\,L}(\rn)\cap L^2(\rn):=\lf\{f \in
L^2(\rn):\ \|S_L(f)\|_{X}<\infty\r\}
$$
with respect to the \emph{quasi-norm} $\|
f\|_{H_{X,\,L}(\rn)}:= \|S_L(f)\|_{X}.$
\end{definition}

\begin{definition}\label{buguoruci}
Let $m\in\nn$, $X$ be a ball quasi-Banach
function space on $\rn$ satisfying Assumption
\ref{vector} for some
$\theta,s\in(0,1]$, and  $L$ a homogeneous
divergence form $2m$-order elliptic operator
 in \eqref{high}
satisfying Ellipticity Condition \ref{ec}.
Assume that $M\in\nn$, $\epsilon\in(0,\fz)$,
and $q\in (p_{-}(L),p_{+}(L)).$ Denote by
$\mathcal{R}(L^M)$ the range of $L^M$.
\begin{itemize}
\item[\rm(i)] A function $\az\in L^q(\rn)$
is called an \emph{$(X,\,q,\,M,\,\epsilon)_L$-molecule}
associated with the ball $B:=B(x_B,r_B)
\subset\rn$, with some $x_B\in\rn$ and
$r_B\in(0,\fz)$, if $\az\in\mathcal{R}(L^M)$
and, for any $k\in\{0,\,\ldots,\,M\}$ and $j\in\zz_+$,
it holds true that
$$\lf\|\lf(r_B^{-2m} L^{-1}\r)^k\az\r\|_{
	L^q(U_j(B))}\le
2^{-j\epsilon}|2^jB|^{1/q}\|\mathbf{1}_{B}
\|_{X}^{-1}.$$
Moreover, if $\az$ is an $(X,\,q,\,M,\,
\epsilon)_L$-molecule for all $q\in(p_{-}(L),p_{+}(L))$,
then $\az$ is called an \emph{$(X,\,M,\,
\epsilon)_L$-molecule}.

\item[\rm(ii)] For any $f\in L^2(\rn)$,
$f=\sum_{j=1}^\fz \lz_j\az_j$ is called a
\emph{molecular $(X,\,q,\,M,\,\epsilon)$-representation}
of $f$ if, for any $j\in\nn$, $\az_j$ is
an $(X,\,q,\,M,\,\epsilon)_L$-molecule associated
with the ball $B_j\subset\rn$, the summation
converges in $L^2(\rn),$ and $\{\lz_j\}_{j\in\nn}
\subset[0,\fz)$ satisfies
\begin{align}\label{eq:6-44}
\Lambda\lf(\{\lz_j\az_j\}_{j\in\nn}\r):=
\lf\|\lf\{\sum_{j=1}^\fz
\lf(\frac{\lz_j}{\|\mathbf{1}_{B_j}\|_{X}}\r)^s
\mathbf{1}_{B_j}\r\}^{1/s}\r\|_X<\fz.
\end{align}
Let
\begin{align*}
\wz{H}^{M,\,q,\,\epsilon}_{X,\,L}(\rn):=\lf\{f\in L^2(\rn):\ f \
\text{has a molecular}\ (X,\,q,\,M,\,\epsilon)
\text{-representation}\r\}
\end{align*}
equipped with the \emph{quasi-norm} $\|
\cdot\|_{H^{M,\,q,\,\epsilon}_{X,\,L}(\rn)}$
given
by setting, for any $f\in\wz{H}^{M,\,q,\,
\epsilon}_{X,\,L}(\rn)$,
\begin{align*}
\|f\|_{H^{M,\,q,\,\epsilon}_{X,\,L}(\rn)}
&:=\inf\Bigg\{\Lambda(\{\lz_j\az_j\}_{j\in\nn}):
\ f=\sum_{j=1}^\fz \lz_j\az_j \\
&\quad\quad\quad\quad \ \text{is
a molecular}
\ (X,\,q,\,M,\,\epsilon)\text{-representation}
\Bigg\},
\end{align*}
where  $\Lambda(\{\lz_j\az_j\}_{j\in\nn})$ is
the same as in \eqref{eq:6-44} and the
infimum is taken over all the molecular
$(X,\,q,\,M,\,\epsilon)$-representations of $f$ as above.

The \emph{molecular Hardy  space}
$H^{M,\,q,\,\epsilon}_{X,\,L}(\rn)$ is then
defined as the completion of $\wz{H}^{M,\,q,
\,\epsilon}_{X,\,L}(\rn)$ with respect to the
quasi-norm $\|\cdot\|_{H^{M,\,q,\,\epsilon}_{X,\,L}(\rn)}$.
\end{itemize}
\end{definition}
The following molecular decomposition of
$H_{X,\,L}(\rn)$ can be found in
\cite[Prosition 6.11]{SHYY17}.

\begin{proposition}\label{prop:6-2-2}
Let $m\in\nn$ and $L$ be a homogeneous divergence
form $2m$-order elliptic operator  in
\eqref{high} satisfying Ellipticity Condition
\ref{ec}. Assume that $X$ is a
ball quasi-Banach function space satisfying both
Assumptions \ref{vector} and \ref{vector2}
for some $\theta,s\in(0,1]$ and  $q\in[2,p_{+}(L)).$
Let $\epsilon\in(n/\tz,\fz)$ and $M
\in(\frac{n}{2m\theta},\infty)\cap\nn$.
Then, for any $f\in H_{X,\,L}(\rn)\cap L^2(\rn)$,
there exists a sequence  $\{\lz_j\}_{j\in\nn}\subset[0,\fz)$
and a sequence $\{\az_j\}_{j\in\nn}$ of $(X,\,M,
\,\epsilon)_L$-molecules
associated, respectively, with the balls $\{B_j
\}_{j\in\nn}$ such that $f=
\sum_{j=1}^\fz \lz_j\az_j$ in $L^2(\rn)$.
Moreover, there exists a positive constant $C$ such that,
for any $f\in H_{X,\,L}(\rn)\cap L^2(\rn)$,
$$\left\|\left\{\sum_{j=1}^\infty
\left(\frac{\lambda_j}{\|\mathbf{1}_{B_j}\|_X}
\right)^s\mathbf{1}_{B_j}\right\}^{\frac{1}{s}}
\right\|_X\le C\|f\|_{H_{X,\,L}(\rn)}.$$
\end{proposition}

\begin{remark}
We point out that Proposition \ref{prop:6-2-2}
was obtained in \cite[Proposition 6.11]{SHYY17}
under the additional assumption that $X$ has an
absolutely continuous quasi-norm.
However, by checking the proof of \cite[Proposition
6.11]{SHYY17} very carefully, we find that,
under the assumption that $f\in H_{X,\,L}(\rn)\cap
L^2(\rn)$, the additional
condition that $X$ has an absolutely continuous
quasi-norm in \cite[Proposition
6.11]{SHYY17}
is  superfluous.
\end{remark}

\section{Maximal Function Characterizations of $H_{X,\, L}(\rn)$}\label{max}

In this section, by making full use of the special
structure of the divergence form elliptic operator and
the extrapolation theorem in the scale of ball
quasi-Banach function spaces, we establish the radial and
the non-tangential maximal function
characterizations of $H_{X,\, L}(\rn)$.

Let $m\in\nn$ and $L$ be a homogeneous
divergence form $2m$-order elliptic
operator  in \eqref{high}.
We first recall the definitions of several maximal
functions associated with $L$. For any given
$\al\in(0,\infty)$ and  for any $f\in L^2(\rn)$,
the \emph{radial maximal function} $\mathcal{R
}^{(\al)}_h(f)$ and the \emph{non-tangential maximal function}
$\mathcal{N}^{(\al)}_h(f)$, associated with the
heat semigroup generated by $L$,
are defined, respectively, by setting,  for
any $x\in\rn,$
\begin{align*}
\mathcal{R}^{(\al)}_h(f)(x):=\sup_{t\in(0,
\infty)}\left[\frac{1}{(\al t)^n}\int
_{B(x,\al t)}
\left|e^{-t^{2m}L}(f)(z)\right|^2\,dz\right]^{1/2}
\end{align*}
and
\begin{align*}
\mathcal{N}^{(\al)}_h(f)(x):=\sup_{(y,t)\in
\Gamma_\al(x)}\left[\frac{1}{(\al t)^n}\int_{B(y,\al t)}
\left|e^{-t^{2m}L}(f)(z)\right|^2\,dz\right]^{1/2},
\end{align*}
where $\Gamma_\al(x)$ is the same as in
\eqref{zui}. For any given $\al\in(0,\infty)$
and for any $f\in L^2(\rn),$ we
also define the \emph{Lusin area function}
$S_h^{(\al)}(f)$,
associated with the heat semigroup generated
by $L$, by setting, for any  $x\in\rn,$
\begin{align*}
S_h^{(\al)}(f)(x):=\left[\int_{\Gamma_\al(x)}
\left|(t\nabla)^m e^{-t^{2m}L}(f)(y)\right|^2\,
\frac{dy\,dt}{t^{n+1}}\right]^{1/2}.
\end{align*}
In what follows, when $\al:=1,$ we remove
the superscript $\al$ for simplicity.
In a similar way, the $S_h$-\emph{adapted},
the $\mathcal{R}_h$-\emph{adapted}, and the
$\mathcal{N}_h$-\emph{adapted Hardy spaces},
$H_{X,\,S_h}(\rn)$, $H_{X,\,\mathcal{R}_h}(\rn)$,
and $H_{X,\,\mathcal{N}_h}(\rn)$, are defined
in the way same as $H_{X,\,L}(\rn)$ with $S_L(f)$
replaced, respectively, by $S_h(f)$,
$\mathcal{R}_h(f),$  and $\mathcal{N}_h(f).$

The main result of this section is stated as follows.

\begin{theorem}\label{th2}
Let  $m\in\nn$ and $L$ be a homogeneous divergence form $2m$-order elliptic operator  in \eqref{high}
satisfying Strong Ellipticity Condition \ref{sec}.
Assume that $X$ is a ball quasi-Banach
function space satisfying both Assumptions
\ref{vector} and \ref{vector2} for some $\theta,s\in(0,1]$ and
$q\in(p_-(L),p_+(L))$. Then the
spaces $H_{X,\,L}(\rn)$, $H_{X,\,S_h}(\rn),$
$H_{X,\,\mathcal{R}_h}(\rn),$
and $H_{X,\,\mathcal{N}_h}(\rn)$ coincide
with equivalent quasi-norms.
\end{theorem}

\begin{remark}
When $X:=L^p(\rn)$ with $p\in(0,p_+(L))$,
Theorem \ref{th2} is just \cite[Theorem 1.4]{CMY16}.
\end{remark}

To prove Theorem \ref{th2}, we need  several
lemmas. Let  $m\in\nn$, $L$ be a homogeneous divergence
form $2m$-order elliptic operator  in
\eqref{high}, and $\epsilon,R,\al \in(0,\infty)$ with
$\epsilon<R.$ For any $f\in L^2(\rn)$,
the \emph{truncated Lusin area function}
$S_h^{\epsilon,\, R,\, \al}(f)$, associated
with the heat semigroup of $L$, is defined
by setting, for any $x\in\rn,$
\begin{align*}
S_h^{\epsilon,\, R,\, \al}(f)(x):=\left\{
\int_{\Gamma_\al^{\epsilon,\, R}(x)}\left|(t\nabla)^m
e^{-t^{2m}L}(f)(y)\right|^2\,\frac{dy \,
dt}{t^{n+1}}\right\}^{1/2},
\end{align*}
where $\Gamma_\al^{\epsilon,\, R}(x):=\{(y,t)
\in\rn\times (\epsilon,R):\  |y-x|<\al t\}.$
Motivated by \cite[Lemma 3.4]{YY13}, we have
the following conclusion.

\begin{lemma}\label{dtkz}
Let $m\in\nn$, $L$ be a homogeneous divergence
form $2m$-order elliptic operator  in \eqref{high}
satisfying Ellipticity Condition \ref{ec},
$\al\in(0,1),$ and $\epsilon, R\in(0,\infty)$
with $\epsilon<R.$ Then there exists a positive
constant $C$, depending only
on both $\al$ and $n,$ such that, for any
$f\in L^2(\rn)$ and $x\in\rn,$
\begin{align*}
S_h^{\epsilon,\, R,\, \al}(f)(x)\leq C[1+
\ln(R/\epsilon)]^{1/2}\mathcal{N}_h(f)(x).
\end{align*}
\end{lemma}

To show Lemma \ref{dtkz}, we need the Besicovitch
covering lemma.
For any $(x,t),(z,\tau)\in\mathbb{R}^{n+1}_+,$ let
\begin{align*}
\rho((x,t),(z,\tau)):=\max\{|x-z|,\,|t-\tau|\}.
\end{align*}
It is obvious that $\rho$ is a metric on $\mathbb{R}^{n+1}_+.$ In what follows, we denote
by $B_\rho((z,\tau),r)$ the ball of $\mathbb{R}^{n+1}_+$ with \emph{center} $(z,\tau)$
and \emph{radius} $r$; namely
\begin{align*}
B_\rho((z,\tau),r):=\lf\{(x,t)\in\mathbb{R}^{n+1}_+:\  \rho((x,t),(z,\tau))<r\r\}.
\end{align*}
The following Besicovitch covering lemma is a  part of \cite[Theorem 1.14]{H01}.

\begin{lemma}\label{covering}
Let $\mathcal{F}$ be any collection of closed balls  with uniformly bounded diameter
in $(\mathbb{R}^{n+1}_+,\rho)$ and $A$ the set of centers of balls in $\mathcal{F}.$
Then there exist $N_n$ countable collections
$\{\mathcal{G}_i\}_{i=1}^{N_n}$ of disjoint balls in $\mathcal{F}$ such that
\begin{align*}
A\subset \bigcup_{i=1}^{N_n}\bigcup_{B\in \mathcal{G}_i}B,
\end{align*}
where the positive integer $N_n$ depends only on $n.$
\end{lemma}

Now, we prove Lemma \ref{dtkz} by using Lemma \ref{covering}.

\begin{proof}[Proof of Lemma \ref{dtkz}]
Let all the symbols be the same as in the present lemma. Fix an $f\in L^2(\rn)$ and an $x\in\rn$, let
\begin{align*}
\mathcal{F}:=\left\{B_\rho((z,\tau),\gamma \tau)\right\}_{(z,\tau)\in \Gamma_\al^{\epsilon,\, R}(x)},
\end{align*}
where $\gamma\in(0,\min\{1/3,(1-\al)/2\})$. Then, by Lemma \ref{covering}, we find that there
exists a collection $\{B_\rho((z_j,\tau_j),\gamma \tau_j)\}_{j\in\Lambda}$ of $\mathcal{F}$ such that
\begin{align}\label{fenglan0}
\Gamma_\al^{\epsilon, R}(x)\subset  \bigcup_{j\in\Lambda}
B_\rho((z_j,\tau_j),\gamma \tau_j)
\ \ \text{and}\ \ \sum_{j\in\Lambda}\mathbf{1}_{B_\rho((z_j,\tau_j),\gamma \tau_j)}\leq N_n,
\end{align}
where $N_n$ is the same as in Lemma \ref{covering}.
In the remainder of this proof, for the simplicity of the presentation, 
for any $j\in\Lambda$, we denote $B_\rho((z_j,\tau_j),\gamma \tau_j)$
simply by $E_j$.

Next, we show that, for any $j\in\Lambda,$
\begin{align}\label{fenglan}
E_j\subset \Gamma^{\epsilon/2,\, 2R}(x).
\end{align}
Indeed, for any $(y,t)\in  E_j,$ we have
\begin{align}\label{gala}
|z_j-y|<\gamma\tau_j\ \ \text{and}\ \ (1-\gamma)\tau_j<t<(1+\gamma)\tau_j.
\end{align}
Since $(z_j,\tau_j)\in\Gamma_\al^{\epsilon,\, R}(x)$, it follows that
\begin{align*}
|x-z_j|<\al \tau_j\ \ \text{and}\ \ \epsilon<\tau_j<R.
\end{align*}
Using this, \eqref{gala}, and $\gamma\in(0,(1-\al)/2),$ we find that, for any $(y,t)\in  E_j,$
\begin{align*}
|x-y|\leq|x-z_j|+|z_j-y|<(\al+\ga)\tau_j<(1-\ga)\tau_j<t
\end{align*}
and
\begin{align*}
\epsilon/2<(1-\ga)\epsilon<(1-\ga)\tau_j<t<(1+\ga)\tau_j<(1+\ga) R<2R,
\end{align*}
which implies that $(y,t)\in \Gamma^{\epsilon/2,\, 2R}(x)$ and hence
$E_j\subset \Gamma^{\epsilon/2,\, 2R}(x).$ That is, \eqref{fenglan} holds true.

On the other hand, we have, for any $j\in\Lambda,$
\begin{align*}
\int_{E_j}\frac{dy\,dt}{t^{n+1}}=\int_{(1-\gamma)\tau_j}^{(1+\gamma)\tau_j}\int_{B(z_j,
\gamma\tau_j)}t^{-n-1}\, dy \, dt\sim 1.
\end{align*}
Using this, \eqref{fenglan}, and \eqref{fenglan0}, we conclude that
\begin{align*}
\mathrm{card}\,(\Lambda)
&\sim\sum_{j\in\Lambda}\int_{E_j}\frac{dy\,dt}{t^{n+1}}
\lesssim
\int_{\Gamma^{\epsilon/2, 2R}(x)}\sum_{j\in\Lambda}\mathbf{1}_{E_j}(y,t)\,\frac{dy\,dt}{t^{n+1}}\\
&\lesssim
\int_{\Gamma^{\epsilon/2, 2R}(x)}\frac{dy\,dt}{t^{n+1}}
\sim1+\ln(R/\epsilon).
\end{align*}
Here and thereafter, $\mathrm{card}\,(\Lambda)$ denotes the \emph{cardinality} of the set $\Lambda$.
From this, \eqref{fenglan0}, the definition of $\mathcal{N}_h(f),$ and Proposition \ref{cacci} with
$x_0:=z_j,\,r:=\gamma \tau_j,$ and $t_0:=\tau_j$, it follows that
\begin{align*}
\left[S_h^{\epsilon,\, R,\, \al}(f)(x)
\right]^2
&=
\int_{\Gamma_\al^{\epsilon,\, R}(x)}\left|(t\nabla)^m e^{-t^{2m}L}(f)(y)\right|^2\,
\frac{dy \, dt}{t^{n+1}}\\
&\lesssim
\sum_{j\in\Lambda}{\tau_j}^{2m-n-1}
\int_{(1-\gamma)\tau_j}^{(1+\gamma)\tau_j}
\int_{B(z_j,\gamma\tau_j)}\left|\nabla^m e^{-t^{2m}L}(f)(y)\right|^2\,dy\,dt\\
&\lesssim
\sum_{j\in\Lambda}{\tau_j}^{-n-1}
\int_{(1-2\gamma)\tau_j}^{(1+2\gamma)\tau_j}\int_{B(z_j,2\gamma\tau_j)}
\left|e^{-t^{2m}L}(f)(y)\right|^2
\,dy\,dt\\
&\lesssim
\sum_{j\in\Lambda}\tau_j^{-1}
\int_{(1-2\gamma)\tau_j}^{(1+2\gamma)\tau_j}t^{-n}\int_{B(z_j,t)}
\left|e^{-t^{2m}L}(f)(y)\right|^2
\,dy\,dt\\
&\lesssim \mathrm{card}\,(\Lambda)[\mathcal{N}_h(f)(x)]^2
\lesssim [1+\ln(R/\epsilon)][\mathcal{N}_h(f)(x)]^2.
\end{align*}
This finishes the proof of Lemma \ref{dtkz}.
\end{proof}

Now, we recall the concept of the Muckenhoupt weight class $A_p(\rn)$ (see, for instance, \cite{g14}).

\begin{definition}\label{weight}
An \emph{$A_p(\rn)$ weight} $\omega$, with $p\in[1,\infty)$, is a nonnegative locally
integrable function on $\rn$ satisfying that, when $p\in(1,\infty),$
$$[\omega]_{A_p(\rn)}:=\sup_{Q \subset\rn}
\left[\frac{1}{|Q|}\int_{Q}\omega(x)\,dx\right]
\left\{\frac{1}{|Q|}\int_Q[\omega(x)]^{\frac{1}{1-p}}\,dx\right\}^{p-1}<\infty$$
and
$$[\omega]_{A_1(\rn)}:=\sup_{Q\subset\rn}\frac{1}{|Q|}\int_Q\omega(x)\,dx\left
[\|\omega^{-1}\|_{L^\infty(Q)}\right]<\infty,$$
where the suprema are taken over all cubes $Q\subset\rn$. Moreover, let
$$A_{\infty}(\rn):=\bigcup_{p\in[1,\infty)}A_p(\rn).$$
\end{definition}

Then we have the following good-$\lambda$ inequality concerning the non-tangential maximal function and
the truncated Lusin area function associated with the heat semigroup of $L$, which plays a key role
in the proof of Theorem \ref{th2}.

\begin{lemma}\label{good}
Let  $m\in\nn$ and $L$ be a homogeneous divergence form $2m$-order elliptic operator  in \eqref{high}
satisfying  Strong Ellipticity Condition \ref{sec}. Assume that $\omega\in A_1(\rn)$ and $\epsilon,
R\in(0,\infty)$ with $\epsilon<R$. Then there exist positive constants $\epsilon_0$ and $C$,
depending only on both $[\omega]_{A_1(\rn)}$ and $L,$ such that, for any $\gamma\in(0,1]$, $\lambda\in(0,\infty)$,
and $f\in L^2(\rn)$,
\begin{align}\label{2022}
&\omega\left(\left\{x\in\rn:\  S_h^{\epsilon,\, R,\, \frac{1}{20}}(f)(x)>2\lambda,\,
\mathcal{N}_h(f)(x)\leq \gamma\lambda\right\}\right)\\\noz
&\quad\leq C\gamma^{\epsilon_0}\omega\left(\left\{x\in\rn:\  S_h^{\epsilon,\, R,\,
\frac{1}{2}}(f)(x)>\lambda
\right\}\right).
\end{align}
\end{lemma}
\begin{proof}
Let all the symbols be the same as in the present lemma. Let $\gamma\in(0,1]$, $\lambda\in(0,\infty)$,
$f\in L^2(\rn)$, and
\begin{align*}
O:=\left\{x\in\rn:\  S_h^{\epsilon,\, R,\,\frac{1}{2}}(f)(x)>\lambda\right\}.
\end{align*}
It is obvious that $O$ is an open set of $\rn.$ By the Whitney decomposition of $O,$ we find that there
exists a family $\{Q_j\}_{j\in\nn}$ of dyadic cubes, respectively, with the edge-lengths
$\{l_j\}_{j\in\nn},$ satisfying that
\begin{itemize}
\item[(i)] $O=\bigcup_{j\in\nn}Q_j$ and $\{Q_j\}_{j\in\nn}$ are disjoint;
\item[(ii)] $2Q_j\subset O$ and $4Q_j\cap O^\com\neq \emptyset$ for any $j\in\nn.$
\end{itemize}
By the fact that $\{Q_j\}_{j\in\nn}$ are disjoint and the estimate that $S_h^{\epsilon,\, R,\,
\frac{1}{20}}(f)\leq S_h^{\epsilon,\, R,\,\frac{1}{2}}(f)$, we conclude that,
to prove \eqref{2022}, it suffices to show that,
for any $j\in\nn,$
\begin{align}\label{2022a}
\omega\left(\left\{x\in Q_j\cap F:\  S_h^{\epsilon,\, R,\, \frac{1}{20}}(f)(x)>2\lambda\right\}\right)
\lesssim\gamma^{\epsilon_0}\omega(Q_j),
\end{align}
where
$$F:=\lf\{x\in\rn:\  \mathcal{N}_h(f)(x)\leq\gamma\lambda\r\}$$
and
$$\omega(Q_j):=\int_{Q_j}\omega(x)\,dx.$$
Using both the assumption that $\omega\in A_1(\rn)$ and \cite[Proposition 7.2.8]{g14}, we find that,
to prove \eqref{2022a}, it suffices to show that, for any $j\in\nn,$
\begin{align}\label{2022b}
\left|\left\{x\in Q_j\cap F:\  S_h^{\epsilon,\, R,\, \frac{1}{20}}(f)(x)>2\lambda\right\}\right|
\lesssim \gamma^{2}|Q_j|.
\end{align}

Let $j\in\nn$. We first prove that, if $\epsilon\in[10\sqrt{n}l_j,R)$, then, for any $x\in Q_j$,
\begin{align}\label{leidi}
S_h^{\epsilon,\,R,\,\frac{1}{20}}(f)(x)\leq \lambda.
\end{align}
Indeed, let $x_j\in4 Q_j\cap O^\com$
and $x\in Q_j$. Then, for any $(y,t)\in\Gamma_{1/ 20}^{\epsilon,\, R}(x)$,
we have
\begin{align*}
|x_j-y|\leq |x_j-x|+|x-y|
< \frac{5}{2}\sqrt{n}l_j+\frac{t}{20}
\leq \frac{1}{4}\epsilon+\frac{t}{20}
<\frac{t}{2},
\end{align*}
and hence
$(y,t)\in \Gamma_{1/ 2}^{\epsilon,\, R}(x_j).$
Thus, $\Gamma_{1/ 20}^{\epsilon,\, R}(x)\subset\Gamma_{1/ 2}^{\epsilon,\, R}(x_j).$
This, together with the fact that
$x_j\in O^\com,$ implies that
\begin{align*}
S_h^{\epsilon,\,R,\,\frac{1}{20}}(f)(x)
\leq S_h^{\epsilon,\,R,\,\frac{1}{2}}(f)(x_j)\leq \lambda.
\end{align*}
This proves \eqref{leidi}.

By \eqref{leidi}, we conclude that  \eqref{2022b} holds true if $\epsilon\in[10\sqrt{n}l_j,R)$.
Thus, in the remainder of this proof,
we always assume that $\epsilon\in(0,10\sqrt{n}l_j).$
We claim that, to prove \eqref{2022b} in this case, it suffices to show that
\begin{align}\label{2022c}
	\left|\left\{x\in Q_j\cap F:\  S_h^{\epsilon,\, 10\sqrt{n}l_j,\, \frac{1}{20}}(f)(x)>\lambda\right\}\right|
	\lesssim \gamma^{2}|Q_j|.
\end{align}

Assume that \eqref{2022c} holds true for the moment. If $R\in(0,10\sqrt{n}l_j]$, by the estimate that
$$S_h^{\epsilon,\,10\sqrt{n}l_j,\,\frac{1}{20}}(f)(x)\geq S_h^{\epsilon,\,R,\,\frac{1}{20}}(f)(x)$$
for any $x\in\rn,$ we find that  \eqref{2022b} in this case holds true.
If $R\in(10\sqrt{n}l_j,\infty)$, applying an argument similar to that used in the estimation of \eqref{leidi}, we conclude that, for any $x\in Q_j,$
\begin{align*}
S_h^{10\sqrt{n}l_j,\,R,\,\frac{1}{20}}(f)(x)\leq \lambda;
\end{align*}
using this and the estimate that
\begin{align*}
S_h^{\epsilon,\,R,\,\frac{1}{20}}(f)(x)
\leq S_h^{\epsilon,\,10\sqrt{n}l_j,\,\frac{1}{20}}(f)(x)
+S_h^{10\sqrt{n}l_j,\,R,\,\frac{1}{20}}(f)(x)
\end{align*}
for any $x\in\rn,$
we find that \eqref{2022b} in this case also holds true. Thus, \eqref{2022b} always holds true.

Now, we prove \eqref{2022c}.
By the Chebyshev inequality, we conclude that \eqref{2022c} can be deduced from
\begin{align}\label{2022d}
\int_{Q_j\cap F}\left[S_h^{\epsilon,\, 10\sqrt{n}l_j,\, \frac{1}{20}}(f)(x)\right]^2\,dx
\lesssim (\gamma \lambda)^2|Q_j|.
\end{align}
Next, we prove \eqref{2022d}. If $\epsilon\in[5\sqrt{n}l_j,10\sqrt{n}l_j),$ from Lemma \ref{dtkz}, it follows that
\begin{align*}
\int_{Q_j\cap F}\left[S_h^{\epsilon,\, 10\sqrt{n}l_j,\, \frac{1}{20}}(f)(x)\right]^2\,dx
\lesssim
\int_{Q_j\cap F}\left[\mathcal{N}_h(f)(x)\right]^2\,dx
\lesssim
(\gamma\lambda)^2|Q_j|,
\end{align*}
which proves that \eqref{2022d} in this case holds true. In the remainder of this proof, we assume that $\epsilon\in(0,5\sqrt{n}l
_j).$ For any $y\in\rn,$ let
\begin{align*}
\psi(y):=\dist(y,Q_j\cap F),
\end{align*}
\begin{align*}
G:=\left\{(y,t)\in\rn\times(\epsilon,10\sqrt{n}l_j):\  \psi(y)<\frac{t}{20}\right\}
\end{align*}
and
\begin{align*}
G_1:=\left\{(y,t)\in\rn\times(\epsilon/2,20\sqrt{n}l_j):\  \psi(y)<\frac{t}{10}\right\}.
\end{align*}
Denote by
$C_{\mathrm{c}}^\infty(G_1)$ the set of all the infinitely differentiable functions on $G_1$ with compact support. Then there exists a function $\eta\in C_{\mathrm{c}}^\infty(G_1)$
such that $\eta\equiv1$ on $G$, $0\leq\eta\leq 1,$ and, for any $k\in\{1,\ldots,m\}$ and $(x,t)\in G_1,$
\begin{align}\label{eta}
\sum_{|\al|=k}|\p^\al \eta(x,t)|^2\lesssim \frac{1}{t^{2k}}
\ \ \text{and}\ \
\p_t\eta(x,t)\lesssim \frac{1}{t}.
\end{align}
Let $u(y,t):=e^{-t^{2m}L}(f)(y)$ for any $(y,t)\in\mathbb{R}^{n+1}_+.$
In what follows, for any $\gamma:=(\gamma_1,\ldots,\gamma_n)\in \zz_+^n$ and any weakly differentiable functions $h$ and $g$ on $\rr^{n+1}_+$, let
$\p^\gamma(hg)(y,t):=\p^\gamma(h(y,t)g(y,t))$ for any $(y,t)\in\rr^{n+1}_+$,
where $\p^\gamma:=(\frac{\p}{\p y_1})^{\gamma_1}\cdots (\frac{\p}{\p y_n})^{\gamma_n}$.
By  the Tonelli theorem, Strong
Ellipticity Condition \ref{sec},  the integral by parts, and the Leibniz rule, we obtain
\begin{align*}
&\int_{Q_j\cap F}\left[S_h^{\epsilon,\, 10\sqrt{n}l_j,\, \frac{1}{20}}(f)(x)\right]^2\,dx\\
&\quad=\int_{Q_j\cap F}\int_{\Gamma_{1/20}^{\epsilon,\,10\sqrt{n}l_j}(x)}\left|(t\nabla)^mu(y,t)\right|^2\,\frac{dy\,dt}{t^{n+1}}\,dx\\
&\quad\lesssim
\int_G t^{2m-1}\left|\nabla ^mu(y,t)\right|^2\,dy\,dt\\\noz
&\quad\lesssim
\Re\left\{\int_{G}t^{2m-1}\sum_{|\al|=|\beta|=m}a_{\al,\,\beta}(y)\p^{\beta}u(y,t)
\overline{\p^\al u(y,t)}\,dy\,dt\right\}\\\noz
&\quad\lesssim
\Re\left\{\int_{G_1}t^{2m-1}\sum_{|\al|=|\beta|=m}a_{\al,\,\beta}(y)\p^{\beta}u(y,t)
\overline{\p^\al u(y,t)}\eta(y,t)\,dy\,dt\right\}\\\noz
&\quad\sim
\left|\Re\left\{\int_{G_1}t^{2m-1}\sum_{|\al|=|\beta|=m}(-1)^m\p^\al \left(\eta a_{\al,\,
\beta}\p^\beta u\right)(y,t)\overline{u(y,t)}\,dy\,dt\right\}\right|\\\noz
&\quad\lesssim
\sum_{k=0}^m\Bigg|\Re\Bigg\{\int_{G_1}t^{2m-1}\sum_{|\al|=|\beta|=m}\sum_{|\widetilde{\al}
|=k,\,\widetilde{\al}\leq\al}C_{(\al,\widetilde{\al})}(-1)^m\p^{\widetilde{\al}}\eta(y,t)\\\noz
&\quad\quad\times\p^{\al-\widetilde{\al}}\left(a_{\al,\,\beta}\p^\beta
u\right)(y,t)\overline{u(y,t)}\,dy\,dt\Bigg\}\Bigg|\\\noz
&\quad=:\sum_{k=0}^m \mathrm{J}_k,
\end{align*}
where
the positive constant $C_{(\al,\,\widetilde{\al})}$ depends only on both $\al$ and $\widetilde{\al}$.
This further implies that, to prove \eqref{2022d}, it suffices to show that, for any $k\in\{0,\ldots,m\},$
\begin{align}\label{2022f}
\mathrm{J}_k\lesssim (\gamma\lambda)^2|Q_j|.
\end{align}

We first estimate $\mathrm{J}_0$. Since
$$\frac{\p u(y,t)}{\p t}=-2mt^{2m-1}L(u)(y,t)$$
for any $(y,t)\in\mathbb{R}^{n+1}_+,$ it follows that
\begin{align*}
\frac{\p |u(y,t)|^2}{\p t}
&=-2mt^{2m-1}L(u)(y,t)\overline{u(y,t)}-2mt^{2m-1}u(y,t)\overline{L(u)(y,t)}\\
&=-4mt^{2m-1}\Re{L(u)(y,t)\overline{u(y,t)}},
\end{align*}
which, together with both the integral by parts and \eqref{eta}, further implies that
\begin{align}\label{j0}
\mathrm{J}_0
&\sim\left|\Re\left\{\int_{G_1}t^{2m-1}\eta(y,t) L(u)(y,t) \overline{u(y,t)}\,dy\,dt\right\}\right|\\ \noz
&\sim\left|\int_{G_1}\frac{\p |u(y,t)|^2}{\p  t}\eta(y,t)\,dy\,dt\right|
\sim
\left|\int_{G_1}|u(y,t)|^2\p_t\eta(y,t)\,dy\,dt\right|\\\noz
&\lesssim\int_{G_1\setminus G}\frac{|u(y,t)|^2}{t}\,dy\,dt.
\end{align}
Let $\delta\in(0,\frac{1}{41})$  and
\begin{align*}
\mathcal{F}:=\left\{B_\rho((y,t),\delta t)\right\}_{(y,t)\in G_1\setminus G}.
\end{align*}
By Lemma \ref{covering}, we find that there exists a collection $\{B_\rho((y_i,t_i),\delta t_i)
\}_{i\in\Lambda}$ of $\mathcal{F}$ such that
\begin{align}\label{juren}
G_1\setminus G\subset
\bigcup_{i\in\Lambda} B_\rho((y_i,t_i),\delta t_i)
\ \ \text{and}\ \
\sum_{i\in\Lambda}\mathbf{1}_{B_\rho((y_i,t_i),\delta t_i)}\leq N_n,
\end{align}
where $N_n$ is the same as in Lemma \ref{covering}. For simplicity, let $E_i:=B_\rho((y_i,t_i),\delta t_i)$
for any $i\in\Lambda.$ From both the fact that
\begin{align*}
G_1\setminus G&\subset
\left\{(y,t)\in\rn\times(\epsilon/2,20\sqrt{n}l_j):\  \frac{t}{20}\leq\psi(y)<\f t {10}\right\}\\
&\quad\cup\left\{(y,t)\in\rn\times(\epsilon/2,20\sqrt{n}l_j):\  \psi(y)<\f t {10},\,
\frac{\epsilon}{2}\leq t<\epsilon\right\}\\
&\quad\cup\left\{(y,t)\in\rn\times(\epsilon/2,20\sqrt{n}l_j):\  \psi(y)<\f t {10},\,
10\sqrt{n}l_j\leq t<20\sqrt{n}l_j\right\}
\end{align*}
and $\delta\in(0,\frac{1}{41})$, we deduce that, for any $i\in\Lambda,$
\begin{align}\label{1951}
E_i\subset G_2,
\end{align}
where
\begin{align*}
G_2:&=
\left\{(y,t)\in\rn\times(\epsilon/5,30\sqrt{n}l_j):\  \frac{t}{40}\leq\psi(y)<\f t 2\right\}\\
&\quad\cup\left\{(y,t)\in\rn\times(\epsilon/5,30\sqrt{n}l_j):\  \psi(y)<\f t 2,\,
\frac{\epsilon}{5}\leq t<2\epsilon\right\}\\
&\quad\cup\left\{(y,t)\in\rn\times(\epsilon/5,30\sqrt{n}l_j):\  \psi(y)<\f t 2,\,
5\sqrt{n}l_j\leq t<30\sqrt{n}l_j\right\}.
\end{align*}
It is easy to prove that, for any $i\in\Lambda$,
\begin{align}\label{202213}
\int_{E_i}|u(y,t)|^2\,dy\,dt
\leq\int_{\widetilde{E}_i}|u(y,t)|^2\,dy\,dt
\lesssim (\gamma\lambda)^2t_i^{n+1},
\end{align}
where $\widetilde{E}_i:=B_\rho((y_i,t_i),2\delta t_i).$
Indeed, by  both $(y_i,t_i)\in G_1$ and $\delta\in(0,\frac{1}{41})$, we conclude that there exists an
$x_{i,j}\in Q_j\cap F$ such that
\begin{align*}
|x_{i,j}-y_i|<\frac{t_i}{10}<(1-2\delta)t_i.
\end{align*}
From this and the definition of $\mathcal{N}_h(f)$, it follows that
\begin{align*}
\int_{\widetilde{E}_i}|u(y,t)|^2\,dy\,dt
&=\int_{(1-2\delta)t_i}^{(1+2\delta)t_i}\int_{B(y_i,2\delta t_i)}
|u(y,t)|^2\,dy\,dt\\
&\lesssim t_i^n\int_{(1-2\delta)t_i}^{(1+2\delta)t_i}\frac{1}{t^n}\int_{B(y_i,t)}
|u(y,t)|^2\,dy\,dt\\
&\lesssim  [\mathcal{N}_h(f)(x_{i,j})]^2t_i^n\int_{(1-2\delta)t_i}^{(1+2\delta)t_i}\,dt
\lesssim (\gamma\lambda)^2t_i^{n+1},
\end{align*}
which implies that \eqref{202213} holds true. Then, using \eqref{juren}, \eqref{202213},
and \eqref{1951}, we find that
\begin{align}\label{20221957}
\int_{G_1\setminus G}
\f{|u(y,t)|^2}{t}\,dy\,dt
&\leq\sum_{i\in\Lambda}\int_{E_i}
\f{|u(y,t)|^2}{t}\,dy\,dt\\\noz
&\lesssim t_i^{-1}\sum_{i\in\Lambda}\int_{E_i}|u(y,t)|^2\,dy\,dt
\lesssim\sum_{i\in\Lambda}(\gamma\lambda)^2t_i^{n}\\\noz
&\sim(\gamma\lambda)^2\sum_{i\in\Lambda}\int_{E_i}t^{-1}\,dy\,dt
\lesssim(\gamma\lambda)^2\int_{G_2}t^{-1}\,dy\,dt\\\noz
&\lesssim(\gamma\lambda)^2\int_{H_1}\left\{\int_{2\psi(y)}^{40\psi(y)}\,\frac{dt}{t}+
\int_{\epsilon/5}^{2\epsilon}\,\frac{dt}{t}+
\int_{5\sqrt{n}l_j}^{30\sqrt{n}l_j}\,\frac{dt}{t}\right\}\,dy\\\noz
&\sim(\gamma\lambda)^2|H_1|,
\end{align}
where
\begin{align*}
H_1:=\left\{y\in\rn:\  \text{there exists a}\ t\in(\epsilon/5,30\sqrt{n}l_j)\ \text{such that}\ (y,t)\in G_2 \right\}.
\end{align*}
By the definitions of both $H_1$ and $G_2,$ we conclude that, for any $y\in H_1,$  there exists a
$t\in(\epsilon/5,30\sqrt{n}l_j)$ and an $x\in Q_j$ such that
\begin{align*}
|x-y|<\f t 2<15\sqrt{n}l_j,
\end{align*}
which further implies that $y\in 32Q_j.$ Therefore, $H_1\subset 32Q_j.$ From this, \eqref{20221957},
and \eqref{j0}, we deduce that
\begin{align}\label{2022fj0}
\mathrm{J}_0\lesssim (\gamma\lambda)^2|Q_j|.
\end{align}

Let $k\in\{1,\ldots,m\}$. Next, we deal with $\mathrm{J}_k$. Via the integral by parts,
\eqref{juren}, and the H\"older inequality, we find that
\begin{align}\label{2022fjk}
\mathrm{J}_k
&\lesssim \sum_{|\al|=|\beta|=m}\sum_{|\widetilde{\al}|=k,\,\widetilde{\al}\leq\al}
\left|\int_{G_1\setminus G}t^{2m-1}
\p^{\widetilde{\al}}\eta (y,t)\overline{u(y,t)}\p^{\al-\widetilde{\al}}\left(a_{\al,\,\beta}\p^\beta u\right)(y,t)
\,dy\,dt\right|\\\noz
&\sim \sum_{|\al|=|\beta|=m}\sum_{|\widetilde{\al}|=k,\,\widetilde{\al}\leq\al}
\left|\int_{G_1\setminus G}t^{2m-1}a_{\al,\beta}(y)
\p^\beta u(y,t)
\overline{\p^{\al-\widetilde{\al}}\left(\left(\p^{\widetilde{\al}}\eta\right)u\right)(y,t)}\,dy\,dt\right|\\\noz
&\lesssim
\sum_{|\al|=|\beta|=m}\sum_{|\widetilde{\al}|=k,\,\widetilde{\al}\leq\al}
\sum_{i\in\Lambda}
\int_{E_i}\left|t^{2m-1}a_{\al,\beta}(y) \p^\beta u(y,t)
\p^{\al-\widetilde{\al}}\left(\left(\p^{\widetilde{\al}}\eta\right)u\right)(y,t)\right|\,dy\,dt\\\noz
&\lesssim
\sum_{|\al|=|\beta|=m}\sum_{|\widetilde{\al}|=k,\,\widetilde{\al}\leq\al}
\sum_{i\in\Lambda}
\left\{\int_{E_i}t^{2m-1}\left|\p^\beta
u(y,t)\right|^2\,dy\,dt\right\}^{1/2}\\\noz
&\quad\times
\left\{\int_{E_i}t^{2m-1}\left|\p^{\al-\widetilde{\al}}\left(\left(\p^{\widetilde{\al}}\eta\right)u\right)(y,t)\right|^2\,dy\,dt\right\}^{1/2}\\\noz
&=:\sum_{|\al|=|\beta|=m}\sum_{|\widetilde{\al}|=k,\,\widetilde{\al}\leq\al}\sum_{i\in\Lambda}
\mathrm{J}_{k,\,i,\,\al,\,\widetilde{\al},\,\beta}.
\end{align}
Using both Proposition \ref{cacci}
and \eqref{202213}, we conclude that, for any $i\in\Lambda$,
\begin{align}\label{yourenma}
\int_{E_i}t^{2m-1}\left|\p^\beta
u(y,t)\right|^2\,dy\,dt
&\lesssim{t_i}^{2m-1}\int_{E_i}\left|\nabla^m
u(y,t)\right|^2\,dy\,dt\\\noz
&\lesssim{t_i}^{-1}\int_{\widetilde{E}_i}|u(y,t)|^2\,dy\,dt
\lesssim (\gamma\lambda)^2t_i^{n}.
\end{align}
On the other hand, from \cite[Theorem 5.2]{AF03} (with some slight modifications), it follows that,
for any ball $B$, any $g\in W^{m,2}(B)$ (see, for instance, \cite[Chapter 3]{AF03} for the precise definition),
and $\ell\in\{0,1,\ldots,m\}$,
\begin{align*}
\left\|\nabla ^\ell g\right\|_{L^2(B)}
\leq  C_{(n,\,m)}\left\|\nabla^m g\right\|_{L^2(B)}^{\f \ell m}
\left\|g\right\|_{L^2(B)}^{1-\f \ell m},
\end{align*}
where the positive constant $C_{(n,\,m)}$ depends only on both $n$ and $m.$ By this, \eqref{eta}, the H\"older inequality,
Proposition \ref{cacci}, and \eqref{202213}, we find that
\begin{align*}
&\int_{E_i}t^{2m-1}\left|\p^{\al-\widetilde{\al}}\left(\left(\p^{\widetilde{\al}}\eta\right)u\right)(y,t)\right|^2\,dy\,dt\\
&\quad\lesssim
\int_{E_i}\sum_{|\widetilde{\gamma}|=h,\,
\widetilde{\gamma}\leq \al-\widetilde{\al}}t^{2m-2|\al-\widetilde{\gamma}|-1}\left|\p^{\widetilde{\gamma}} u(y,t)\right|^2\,dy\,dt
\\\noz
&\quad\lesssim \sum_{|\widetilde{\gamma}|=h,\,\widetilde{\gamma}\leq \al-\widetilde{\al}}t_i^{2m-2|\al-\widetilde{\gamma}|-1}\int_{(1-\delta )t_i}^{(1+\delta)t_i}\left\{\int_{B(y_i,\,\delta t_i)}
\left|\nabla ^mu(y,t)\right|^2
\,dy\right\}^{\frac{|\widetilde{\gamma}|}{m}}\\\noz
&\quad\quad\times
\left\{\int_{B(y_i,\,\delta t_i)}\left| u(y,t)\right|^2\,dy\right\}^{1-\frac{|\widetilde{\gamma}|}{m}}dt\\\noz
&\quad\lesssim
\sum_{|\widetilde{\gamma}|=h,\,\widetilde{\gamma}\leq \al-\widetilde{\al}}t_i^{2m-2|\al-\widetilde{\gamma}|-1}\left\{\int_{E_i}
\left|\nabla^m u(y,t)\right|^2\,dy\,dt\right\}
^{\frac{|\widetilde{\gamma}|}{m}}
\left\{\int_{E_i}\left| u(y,t)\right|^2\,dy\,dt\right\}^{1-\frac{|\widetilde{\gamma}|}{m}}\\\noz
&\quad\lesssim \sum_{|\widetilde{\gamma}|=h,\,\widetilde{\gamma}\leq \al-\widetilde{\al}}t_i^{2m-2|\al-\widetilde{\gamma}|-1}\left\{\frac{1}{t_i^{2m}}\int_{\widetilde{E}_i}
\left|u(y,t)\right|^2\,dy\,dt\right\}
^{\frac{|\widetilde{\gamma}|}{m}}
\left\{\int_{E_i}\left| u(y,t)\right|^2\,dy\,dt\right\}^{1-\frac{|\widetilde{\gamma}|}{m}}\\\noz
&\quad\lesssim (\gamma\lambda)^2t_i^n.
\end{align*}
This, combined with \eqref{yourenma}, \eqref{2022fjk}, and the estimation of \eqref{20221957}, implies that
\begin{align*}
\mathrm{J}_k\lesssim(\gamma\lambda)^2\sum_{i\in\Lambda}t_i^n
\lesssim (\gamma\lambda)^2|H_1|
\lesssim(\gamma\lambda)^2|Q_j|,
\end{align*}
which, together with \eqref{2022fj0}, further implies that \eqref{2022f} holds true. This finishes
the proof of Lemma \ref{good}.
\end{proof}

For any $M\in\zz_+,$ $f\in L^2(\rn)$, and $x\in\rn,$  let
\begin{align*}
\mathcal{R}_{h,\, M}(f)(x):=\sup_{t\in(0,\infty)}\left[\frac{1}{t^n}\int_{B(x,t)}
\left|\left(t^{2m}L\right)^Me^{-t^{2m}L}(f)(y)\right|^2\,dy\right]^{1/2}.
\end{align*}
By an argument similar to that used in the estimations of \cite[(6.48) and (6.49)]{HM09},
we have the following conclusion; we omit the details here.

\begin{lemma}\label{6.47}
Let  $m\in\nn$, $L$ be a homogeneous divergence form $2m$-order elliptic operator  in \eqref{high} satisfying
Ellipticity Condition \ref{ec}, $q\in(p_{-}(L),p_{+}(L))$, and $M\in\zz_+.$  Then there exists
a positive constant $C$ such that, for any $f\in L^2(\rn)\cap L^q(\rn),$
\begin{align*}
\left\|\mathcal{R}_{h,\, M}(f)\right\|_{L^q(\rn)}\leq C\|f\|_{L^q(\rn)}.
\end{align*}
\end{lemma}

The following extrapolation theorem is  a slight variant of a special case of \cite[Theorem 4.6]{cmp11}
via replacing Banach function spaces by ball Banach function spaces; we omit the details here.

\begin{lemma}\label{waicha}
Assume that $X$ is a ball quasi-Banach function space on $\rn$ and $p_0\in(0,\infty)$.
Let $\mathcal{F}$ be the set of all pairs $(F,G)$ of nonnegative measurable functions such that,
for any given $\omega\in A_1(\rn)$,
$$
\int_{\rn}[F(x)]^{p_0}\omega(x)\,dx\leq C_{(p_0,[\omega]_{A_1(\rn)})}
\int_{\rn}[G(x)]^{p_0}\omega(x)\,dx,
$$
where $C_{(p_0,[\omega]_{A_1(\rn)})}$ is a positive constant independent of $(F,G)$, but depends on both $p_0$
and $[\omega]_{A_1(\rn)}$. Assume further that
$X^{1/p_0}$ is a ball Banach function space and $\cm$ is bounded on $(X^{1/p_0})'$,
where $\cm$ is the same as in \eqref{jidahanshu}.
Then there exists a positive constant $C_0$ such that, for any $(F,G)\in\mathcal{F}$,
$$\|F\|_{X}\leq C_0\|G\|_{X}.$$
\end{lemma}

For any  $\al\in(0,\infty)$, any measurable function $F$ on $ \mathbb{R}^{n+1}_+$, and any $x\in\rn,$  let
\begin{align}\label{zhangpeng}
\mathcal{A}^{(\al)}(F)(x):=\left\{\int_{\Gamma_\al(x)}|F(y,t)|^2\,\frac{dy\,dt}{t^{n+1}}\right\}^{1/2},
\end{align}
where $\Gamma_\al(x)$ is the same as in \eqref{zui}.
We have the following lemma, which is a part of \cite[Proposition 4.9]{CMP20}.

\begin{lemma}\label{pro4.9}
Let $\al,\beta\in(0,\infty)$, $p\in(0,2],$ and $\omega\in A_1(\rn).$ Then there exist
positive constants $C_1$ and $C_2$, depending only on both $p$ and $[\omega]_{A_1(\rn)},$ such that, for any measurable function $F$ on $\mathbb{R}^{n+1}_+$,
\begin{align*}
C_1\left\|\mathcal{A}^{(\beta)}(F)\right\|_{L^p_\omega(\rn)}
\leq\left\|\mathcal{A}^{(\al)}(F)\right\|_{L^p_\omega(\rn)}
\leq C_2\left\|\mathcal{A}^{(\beta)}(F)\right\|_{L^p_\omega(\rn)}.
\end{align*}
\end{lemma}

Now, we show Theorem \ref{th2} by using Lemmas \ref{good}, \ref{6.47}, \ref{waicha}, and \ref{pro4.9}.

\begin{proof}[Proof of Theorem \ref{th2}]
Let all the symbols be the same as in the present theorem. The proof of the present theorem is divided into
the following four steps.

{\bf Step 1}. In this step, we show that
\begin{align}\label{yys}
\left[L^2(\rn)\cap H_{X,\, S_h}(\rn)\right]
\subset\left[L^2(\rn)\cap H_{X,\, L}(\rn)\right].
\end{align}
Let $p_0:=s.$ By Lemma \ref{waicha}, to prove \eqref{yys}, it suffices to show that, for any
$\omega\in A_1(\rn)$, there exists a positive constant $C_{(p_0,\,[\omega]_{A_1(\rn)})}$,
depending only on both $p_0$ and $[\omega]_{A_1(\rn)},$ such that, for any $f\in L^2(\rn)\cap H_{X,\,S_h}(\rn),$
\begin{align}\label{weightyys}
\int_{\rn}\left[S_{L}(f)(x)\right]^{p_0}\omega(x)\,dx\leq
C_{(p_0,\,[\omega]_{A_1(\rn)})}\int_{\rn}\left[S_{h}(f)(x)\right]^{p_0}\omega(x)\,dx.
\end{align}

Recall that, for any $f\in L^2(\rn)$ and $x\in\rn$,
\begin{align*}
S_{h}^{(2)}(f)(x):=\left[\int_{\Gamma_2(x)}\left|(t\nabla)^me^{-t^{2m}L}(f)(y)\right|^2
\frac{dy\,dt}{t^{n+1}}\right]^{1/2}	
\end{align*}
and
\begin{align*}
S_{L}^{(2)}(f)(x):=\left[\int_{\Gamma_2(x)}\left|t^{2m}Le^{-t^{2m}L}(f)(y)\right|^2
\frac{dy\,dt}{t^{n+1}}\right]^{1/2}.	
\end{align*}
By \cite[(3.30)]{CMY16}, we find that, for any $f\in L^2(\rn)$ and almost every $x\in\rn,$
\begin{align*}
S_L(f)(x)\lesssim\left[S_{h}^{(2)}(f)(x)\right]^{1/2}\left[S_{L}^{(2)}(f)(x)\right]^{1/2},
\end{align*}
and hence, for any $\epsilon\in(0,\infty),$
\begin{align*}
S_L(f)(x)\lesssim
\frac{1}{\epsilon}S_{h}^{(2)}(f)(x)
+\epsilon S_{L}^{(2)}(f)(x).
\end{align*}
From this and Lemma \ref{pro4.9}, it follows that, for any
$\epsilon\in(0,\infty)$ and $f\in L^2(\rn)\cap H_{X,\, S_h}(\rn)$,
\begin{align}\label{e3.25}
\left\|S_L(f)\right\|_{L^{p_0}_\omega(\rn)}
&\lesssim\frac{1}{\epsilon}
\left\|S_h^{(2)}(f)\right\|_{L^{p_0}_\omega(\rn)}
+\epsilon\left\|S_L^{(2)}(f)\right\|_{L^{p_0}_\omega(\rn)}\\ \noz
&\sim\frac{1}{\epsilon}\left\|S_h(f)\right\|_{L^{p_0}_\omega(\rn)}
+\epsilon\left\|S_L(f)\right\|_{L^{p_0}_\omega(\rn)}
\end{align}
with the implicit positive constants independent of both $f$ and $\epsilon.$ Letting $\epsilon$ be
sufficiently small in \eqref{e3.25}, we then obtain \eqref{weightyys}.

{\bf Step 2}. In this step, we prove that
\begin{align}\label{nn}
\left[L^2(\rn)\cap H_{X,\, \mathcal{N}_h}(\rn)\right]\subset \left[L^2(\rn)
\cap H_{X,\, S_h}(\rn)\right].
\end{align}
Let $p_0:=s$ with $s$  in the present theorem. We first show that, for any $\omega\in A_1(\rn),$ there exists a positive constant
$C_{(p_0,\,[\omega]_{A_1(\rn)})}$, depending only on both $p_0$ and $[\omega]_{A_1(\rn)}$, such that, for any
$f\in L^2(\rn)\cap H_{X,\, \mathcal{N}_h}(\rn),$
\begin{align}\label{guxinglei}
\int_\rn \left[S_h(f)(x)\right]^{p_0}\omega(x)\,dx
\leq C_{(p_0,\,[\omega]_{A_1(\rn)})} \int_\rn \left[\mathcal{N}_h(f)(x)\right]^{p_0}\omega(x)\,dx.
\end{align}
Without loss of generality, we may assume that
\begin{align}\label{cs2}
\int_\rn [\mathcal{N}_h(f)(x)]^{p_0}\omega(x)\,dx<\infty,
\end{align}
otherwise \eqref{guxinglei} holds true automatically. By both the Fubini theorem and Lemma \ref{good}, we conclude that,
for any $\gamma\in(0,1]$ and $\epsilon, R\in(0,\infty)$ with $\epsilon< R,$
\begin{align}\label{bs}
&\int_\rn\left[S_h^{\epsilon,\,  R,\, \frac{1}{20}}(f)(x)\right]^{p_0}\omega(x)\,dx\\\noz
&\quad\sim
\int_0^\infty t^{p_0-1}\omega\left(\left\{x\in\rn:\  S_h^{\epsilon,\, R,\,
\frac{1}{20}}(f)(x)>t\right\}\right)\,dt\\\noz
&\quad\lesssim
\int_0^\infty t^{p_0-1}\omega\left(\left\{x\in\rn:\  S_h^{\epsilon,\, R,\,
\frac{1}{20}}(f)(x)>t,\,\mathcal{N}_h(f)(x)\leq\gamma t\right\}\right)\,dt\\\noz
&\quad\quad+
\int_0^\infty t^{p_0-1}\omega\left(\left\{x\in\rn:\  \mathcal{N}_h(f)(x)> \gamma t\right\}\right)\,dt\\\noz
&\quad\lesssim
\gamma^{\epsilon_0}\int_0^\infty t^{p_0-1}\omega\left(\left\{x\in\rn:\
S_h^{\epsilon,\, R,\, \frac{1}{2}}(f)(x)>\frac{t}{2}
\right\}\right)\,dt\\\noz
&\quad\quad+
\int_0^\infty t^{p_0-1}\omega\left(\left\{x\in\rn:\  \mathcal{N}_h(f)(x)>
\gamma t\right\}\right)\,dt\\\noz
&\quad\sim\gamma^{\epsilon_0}\int_\rn\left[S_h^{\epsilon,\,  R,\, \frac{1}{2}}(f)(x)
\right]^{p_0}\omega(x)\,dx+
\frac{1}{\gamma}\int_\rn\left[\mathcal{N}_h(f)(x)\right]^{p_0}\omega(x)\,dx.
\end{align}
On the other hand, using both Lemma \ref{dtkz} and \eqref{cs2}, we find that
\begin{align}\label{cs}
\int_\rn\left[S_h^{\epsilon,\,  R,\, \frac{1}{2}}(f)(x)\right]^{p_0}\omega(x)\,dx
\lesssim\int_\rn\left[\mathcal{N}_h(f)(x)\right]^{p_0}\omega(x)\,dx<\infty.
\end{align}
Moreover, from Lemma \ref{pro4.9}, we deduce that
\begin{align*}
\int_\rn\left[S_h^{\epsilon,\,  R,\, \frac{1}{20}}(f)(x)\right]^{p_0}\omega(x)\,dx
&\sim\int_\rn\left[S_h^{\epsilon,\,  R,\, \frac{1}{2}}(f)(x)\right]^{p_0}\omega(x)\,dx\\
&\sim\int_\rn\left[S_h^{\epsilon,\,  R,\, 1}(f)(x)\right]^{p_0}\omega(x)\,dx
\end{align*}
with the positive equivalence constants independent of $f$, $\epsilon,$ and $R.$ By this, \eqref{cs},
and \eqref{bs} with $\gamma$ sufficient small, we conclude that
\begin{align}\label{e3.31}
\int_\rn\left[S_h^{\epsilon,\,  R,\, 1}(f)(x)\right]^{p_0}\omega(x)\,dx
\lesssim\int_\rn\left[\mathcal{N}_h(f)(x)\right]^{p_0}\omega(x)\,dx<\infty.
\end{align}
Letting both $\epsilon\to0$ and $R\to\infty$ in \eqref{e3.31}, we find that \eqref{guxinglei} holds true.
Then, from Lemma \ref{waicha} and \eqref{guxinglei}, it follows that, for any $f\in L^2(\rn)\cap
H_{X,\, \mathcal{N}_h}(\rn),$ $f\in H_{X,\,S_h}(\rn)$ and
\begin{align*}
\|S_h(f)\|_{X}\lesssim \|\mathcal{N}_h(f)\|_X,
\end{align*}
which implies that \eqref{nn} holds true.

\textbf{Step 3}. In this step, we show that
\begin{align}\label{nn2}
\left[L^2(\rn)\cap H_{X,\, \mathcal{R}_h}(\rn)\right]\subset
\left[L^2(\rn)\cap H_{X,\, \mathcal{N}_h}(\rn)\right].
\end{align}
Let $p_0:=s$ with $s$ in the present theorem. By Lemma \ref{waicha}, to prove \eqref{nn2},
it suffices to show that, for any $\omega\in A_1(\rn),$
there exists a positive constant $\widetilde{C}_{(p_0,\,[\omega]_{A_1(\rn)})}$, depending only on both
$p_0$ and $[\omega]_{A_1(\rn)}$, such that, for any $f\in L^2(\rn)\cap H_{X,\,\mathcal{R}_h}(\rn),$
\begin{align}\label{guxinglei2}
\int_\rn \left[\mathcal{N}_h(f)(x)\right]^{p_0}
\omega(x)\,dx
\leq\widetilde{C}_{(p_0,\,[\omega]_{A_1(\rn)})} \int_\rn \left[\mathcal{R}_h(f)(x)\right]^{p_0}\omega(x)\,dx.
\end{align}
It is obvious that, for any $f\in L^2(\rn)$ and $x\in\rn,$
\begin{align}\label{df}
\mathcal{N}_h^{(\f 1 2)}(f)(x)\leq2^{n/2}\mathcal{R}_h(f)(x).
\end{align}
On the other hand, applying an argument similar to that used in the proof of \cite[Lemma 6.2]{HM09}, we conclude that,
for any $\al,\beta\in(0,\infty)$,
\begin{align*}
\int_\rn \left[\mathcal{N}_h^{(\al)}(f)(x)
\right]^{p_0}\omega(x)\,dx
\sim\int_\rn \left[\mathcal{N}_h^{(\beta)}(f)(x)
\right]^{p_0}\omega(x)\,dx,
\end{align*}
where the positive equivalence constants depend only on both $p_0$ and $[\omega]_{A_1(\rn)}.$ From this and \eqref{df},
we deduce that \eqref{guxinglei2} holds true.

\textbf{Step 4}. In this step, we prove that
\begin{align}\label{nn3}
\left[L^2(\rn)\cap H_{X,\, L}(\rn)\right]
\subset\left[L^2(\rn)\cap H_{X,\, \mathcal{R}_h}(\rn)\right].
\end{align}
Let $f\in L^2(\rn)\cap H_{X,\, L}(\rn),$ $\epsilon\in (n/\theta,\infty)$, and $M\in \nn$
be sufficiently large. By Proposition \ref{prop:6-2-2}, we conclude that there exist
$\{\lz_i\}_{i\in\nn}\subset[0,\fz)$ and a sequence $\{m_i\}_{i\in\nn}$ of $(X,\,M,\,\epsilon)_L$-molecules associated,
respectively, with the balls $\{B_i\}_{i\in\nn}$ such that
\begin{align}\label{xiaoxiang}
f=\sum_{i=1}^\fz \lz_im_i
\end{align}
in $L^2(\rn)$, and
\begin{align}\label{fengong}
\left\|\left\{\sum_{i=1}^\infty
\left(\frac{\lambda_i}{\|\mathbf{1}_{B_i}\|_X}\right)^s\mathbf{1}_{B_i}
\right\}^{\frac{1}{s}}\right\|_X\lesssim\|f\|_{H_{X,\,L}(\rn)}.
\end{align}

Next, we show that there exists a positive constant $\beta\in(n/\theta-n/q,\,\infty)$ such that,
for any $j\in\zz_+$ and any $(X,\,M,\,\epsilon)_L$-molecule $b$ associated with the ball
$B:=B(x_B,r_B)\subset\rn$ for some $x_B\in\rn$ and $r_B\in(0,\infty),$
\begin{align}\label{hexin0}
\left\|\mathcal{R}_h(b)\right\|_{L^q(U_j(B))}\lesssim |B|^{1/q}\|\mathbf{1}_{B}\|_X^{-1},
\end{align}
where $\theta$ is the same as in Assumption \ref{vector},  $U_0(B):=2B,$ and $U_j(B):=(2^{j+1}B)\setminus(2^jB)$
for any $j\in\nn$.

When $j=0,$ from Lemma \ref{6.47} and the assumption that $b$ is a molecule, we deduce that
\begin{align*}
\left\|\mathcal{R}_h(b)\right\|_{L^q(U_0(B))}\lesssim \|b\|_{L^q(\rn)}\lesssim
|B|^{1/q} \|\mathbf{1}_B\|_X^{-1}.
\end{align*}
For any $j\in\nn$ and $x\in U_j(B),$ we have
\begin{align*}
\mathcal{R}_h(b)(x)
&\leq \left(\sup_{t\in(0,2^{j/2-2}r_B)}+\sup_{t\in[2^{j/2-2}r_B,\infty)}\right)\left[\frac{1}{t^n}\int_{B(x,t)}\left|e^{-t^{2m}L}(b)(y)
\right|^2\,dy\right]^{1/2}
\\
&=:I_j(x)+II_j(x).
\end{align*}
For any $j\in\nn,$ let $S_j(B):=(2^{j+3}B)/(2^{j-3}B),$ $R_j(B):=(2^{j+5}B)/(2^{j-5}B),$ and
$E_j(B):=[R_j(B)]^\com.$ It is obvious that, for any $j\in\nn,$ $t\in(0,2^{j/2-2}r_B)$,
and $x\in U_j(B),$ we have $B(x,t)\subset S_j(B)$ and $\dist(S_j(B),E_j(B))\sim 2^jr_B$.
By this, Proposition \ref{basic}(iii), and Definition \ref{buguoruci}(i), we conclude that
\begin{align}\label{2053}
&\left\|\sup_{t\in(0,2^{j/2-2}r_B)}\left[\frac{1}{t^n}\int_{B(x,t)}\left|e^{-t^{2m}L}
\lf(b\mathbf{1}_{E_j(B)}\r)(y)\right|^2\,dy\right]^{1/2}\right\|_{L^q(U_j(B))}\\\noz
&\quad\lesssim\sup_{t\in(0,2^{j/2-2}r_B)}\left[\frac{1}{t^n}\int_{S_j(B)}\left|e^{-t^{2m}L}
\lf(b\mathbf{1}_{E_j(B)}\r)(y)\right|^2\,dy\right]^{1/2}|2^jB|^{1/q}\\\noz
&\quad\lesssim\sup_{t\in(0,2^{j/2-2}r_B)}t^{-n/2}
\exp\left\{-c\left(\frac{2^j r_B}{t}\right)^{2m/(2m-1)}
\right\}\|b\|_{L^2(E_j(B))}|2^jB|^{1/q}\\\noz
&\quad\lesssim\sup_{t\in(0,2^{j/2-2}r_B)}t^{-n/2}\left(\frac{t}{2^jr_B}\right)^N
\|b\|_{L^2(E_j(B))}|2^jB|^{1/q}\\\noz
&\quad\lesssim 2^{-j(n/4+N/2)}r_B^{-n/2}\|b\|_{L^2(\rn)}|2^jB|^{1/q}
\lesssim2^{-j(n/4+N/2-n/q)}|B|^{1/q}\|\mathbf{1}_B\|_{X}^{-1},
\end{align}
where $N$ is a positive number satisfying $N\in(2n/\theta-n/2,\infty).$
On the other hand, from Lemma \ref{6.47} and Definition \ref{buguoruci}(i), it follows that
\begin{align*}
&\left\|\sup_{t\in(0,2^{j/2-2}r_B)}\left[\frac{1}{t^n}\int_{B(x,t)}\left|e^{-t^{2m}L}
\lf(b\mathbf{1}_{R_j(B)}\r)(y)\right|^2\,dy\right]^{1/2}\right\|_{L^q(U_j(B))}\\\noz
&\quad\lesssim\lf\|\mathcal{R}_h(b\mathbf{1}_{R_j(B)})\r\|_{L^q(\rn)}\lesssim \|b\|_{L^q(R_j(B))}
\lesssim 2^{-j(\epsilon-n/q)}|B|^{1/q}\|\mathbf{1}_B\|_X^{-1}.
\end{align*}
By this and \eqref{2053}, we find that, for any $j\in\nn,$
\begin{align}\label{hehehehe}
\|I_j\|_{L^q(U_j(B))}\lesssim [2^{-j(\epsilon-n/q)}+2^{-j(n/4+N/2-n/q)}]
|B|^{1/q}\|\mathbf{1}_{B}\|_X^{-1}.
\end{align}
Moreover, it is obvious that, for any $j\in\nn$ and $x\in U_j(B),$
\begin{align*}
II_j(x)&\lesssim 2^{-mMj}\sup_{t\in(2^{j/2-2}r_B,\infty)}\left[\frac{1}{t^n}
\int_{B(x,t)}\left|\left(t^{2m}L\right)^Me^{-t^{2m}L}\left(r_B^{-2mM}L^{-M}(b)\right)
(y)\right|^2\,dy\right]^{1/2}\\
&\lesssim 2^{-mMj}\mathcal{R}_{h,\,M}\left(r_B^{-2M}L^{-M}(b)\right)(x).
\end{align*}
From this, Lemma \ref{6.47}, and Definition \ref{buguoruci}(i), we deduce that
\begin{align*}
\|II_j\|_{L^q(U_j(B))}
&\lesssim 2^{-mMj}\left\|\mathcal{R}_{h,\,M}\left(r_B^{-2M}L^{-M}(b)\right)\right\|_{L^q(\rn)}\\
&\lesssim 2^{-mMj}\left\|r_B^{-2M}L^{-M}(b)\right\|_{L^q(\rn)}
\lesssim2^{-mMj}|B|^{1/q}\|\mathbf{1}_B\|_X^{-1}.
\end{align*}
By this and \eqref{hehehehe}, we conclude that
\begin{align*}
\left\|\mathcal{R}_h(b)\right\|_{L^q(U_j(B))}\lesssim \left[2^{-jmM}+2^{-j(\epsilon-n/q)}
+2^{-j(n/4+N/2-n/q)}\right]|B|^{1/q}\|\mathbf{1}_{B}\|_X^{-1},
\end{align*}
which implies that \eqref{hexin0} holds true with $\beta:=\min\{mM,\,\epsilon-n/q,\,
n/4+N/2-n/q\}\in(n/\theta-n/q,\infty)$.

Furthermore, from \eqref{xiaoxiang} and the fact that $\mathcal{R}_h$ is bounded on $L^2(\rn),$
it follows that, for almost every $x\in\rn,$
\begin{align*}
\lf|\mathcal{R}_h(f)(x)\r|
\leq\sum_{i=1}^\infty\lf|\ld_i\mathcal{R}_h(m_i)(x)\r|=\sum_{i=1}^\infty
\sum_{j=0}^\infty\lf|\ld_i\mathcal{R}_h(m_i)(x)\r|\mathbf{1}_{U_j(B_i)}(x).
\end{align*}
Using this, \eqref{hexin0} with $b:=m_i$ for $i\in\nn$, Lemma \ref{thm:2-2-3},
and \eqref{fengong}, we find that
\begin{align*}
\lf\|\mathcal{R}_h(f)\r\|_X
&\lesssim \left\|\sum_{i=1}^\infty\sum_{j=0}^\infty\lf|\ld_i
\mathcal{R}_h(m_i)\r|\mathbf{1}_{U_j(B_i)}\right\|_X\lesssim\left\|\left\{\sum_{i=1}^\infty
\left(\frac{\lambda_i}{\|\mathbf{1}_{B_i}\|_X}\right)^s\mathbf{1}_{B_i}\right\}^{\frac{1}{s}}
\right\|_X\lesssim\|f\|_{H_{X,\,L}(\rn)},
\end{align*}
which further implies that \eqref{nn3} holds true.

Then, by the conclusions obtained in Steps 1 though 4, we conclude that
\begin{align*}
\left[L^2(\rn)\cap H_{X,\, S_h}(\rn)\right]
&=\left[L^2(\rn)\cap H_{X,\, \mathcal{R}_h}(\rn)\right]	
=\left[L^2(\rn)\cap H_{X,\, \mathcal{N}_h}(\rn)\right]\\
&= \left[L^2(\rn)\cap H_{X,\, L}(\rn)\right]
\end{align*}
with equivalent quasi-norms. From this and a density argument, we deduce that the spaces
$H_{X,\, S_h}(\rn),\, H_{X,\, \mathcal{N}_h}(\rn),\,$
$H_{X,\, \mathcal{R}_h}(\rn),$ and $H_{X,\, L}(\rn)$
coincide with the equivalent quasi-norms. This finishes the proof of Theorem \ref{th2}.
\end{proof}

\section{Riesz Transform Characterization of $H_{X,\, L}(\rn)$}\label{Riesz}

In this section, we establish the Riesz transform characterization of $H_{X,\, L}(\rn)$.
In Subsection \ref{s1}, we prove that the Riesz transform $\nabla^mL^{-1/2}$ is
bounded from $H_{X,\,L}(\rn)$ to $H_{X}(\rn)$ by using the
molecular characterization of both $H_{X,\, L}(\rn)$ and $H_X(\rn),$
where $H_X(\rn)$ denotes the Hardy space associated with $X$ introduced in \cite{SHYY17}.
In Subsection \ref{s2}, we first introduce  the homogeneous Hardy--Sobolev space $\dot{H}_{m,\,X}(\rn)$
and then establish its atomic decomposition which plays a key role in the proof of the
Riesz transform characterization of $H_{X,\, L}(\rn)$. Then, using the atomic decomposition
of $\dot{H}_{m,\,X}(\rn)$ and borrowing some ideas from the proof of \cite[Proposition 5.17]{hmm11}, we obtain
the Riesz transform characterization of $H_{X,\,L}(\rn)$.

Let $m\in\nn$ and $L$ be a homogeneous divergence form $2m$-order elliptic operator
in \eqref{high} satisfying Ellipticity Condition \ref{ec}. Then the Riesz transform $\nabla^m L^{-1/2}$
is defined by setting, for any $f\in L^2(\rn)$ and $x\in\rn,$
\begin{align*}
\nabla^m L^{-1/2}(f)(x):=\frac{1}{2\sqrt{\pi}}\int_0^\infty\nabla^m e^{-sL}(f)(x)\frac{ds}{\sqrt{s}}.
\end{align*}

\begin{definition}
Let $X$ be a ball quasi-Banach function space on $\rn$, $m\in\nn$, and $L$ be a homogeneous divergence
form $2m$-order elliptic operator  in \eqref{high}.
The \emph{Hardy  space} $H_{X,\,L,\,\mathrm{Riesz}}(\rn)$, associated with $L$,
is defined as the completion of the set
\begin{align*}
\left\{f\in L^2(\rn):\ \|f\|_{H_{X,\,L,\,\mathrm{Riesz}}(\rn)}:
=\left\|\nabla^m L^{-1/2}(f)\right\|_{H_X(\rn)}<\infty\right\}
\end{align*}
with respect to the \emph{quasi-norm} $\|\cdot\|_{H_{X,\,L,\,\mathrm{Riesz}}(\rn)}$,
where $H_X(\rn)$ denotes the Hardy space associated with $X$ (see Definition \ref{def-HX}
below for its definition).
\end{definition}

The following theorems are the main results of this section.
\begin{theorem}\label{r}
Let $m\in\nn$ and $L$ be a homogeneous divergence form $2m$-order elliptic operator  in \eqref{high}
satisfying Ellipticity Condition \ref{ec}. Assume that $X$ is a ball quasi-Banach function
space satisfying both Assumptions \ref{vector} and \ref{vector2} for some $\theta\in(n/(n+m),1]$, $s\in(\theta,1]$,
and $q\in(p_-(L),\min\{p_{+}(L),\,q_{+}(L)\})$. Then there exists a positive constant $C$ such that,
for any $f\in H_{X,\, L}(\rn),$ $f\in H_{X,\,L,\,\mathrm{Riesz}}(\rn)$ and
\begin{align*}
\|f\|_{H_{X,\,L,\,\mathrm{Riesz}}(\rn)}\leq C\|f\|_{H_{X,\, L}(\rn)}.
\end{align*}
\end{theorem}

\begin{remark}
When $X:=L^p(\rn)$ with $p\in(n/(n+m),1],$ Theorem \ref{r} is just \cite[Theorem 6.2]{CY12}.
\end{remark}

\begin{theorem}\label{riesz}
Let $m\in\nn$, $L$ be a homogeneous divergence form $2m$-order elliptic operator  in \eqref{high}
satisfying Ellipticity Condition \ref{ec}, and the family $\{e^{-tL}\}_{t\in(0,\infty)}$
of operators satisfy $m-L^r(\rn)-L^2(\rn)$ off-diagonal estimates with some $r\in(1,2].$
Assume that $X$ is a  ball quasi-Banach function space satisfying both
Assumptions \ref{vector} and \ref{vector2} for some $s\in (0,1],$ $\theta\in(\frac{nr}{n+mr},1],$
and $q\in[2,p_+(L)).$ Then there exists a positive constant $C$ such that, for any
$h\in H_{X,\,L,\,\mathrm{Riesz}}(\rn),$ $h\in H_{X,\,L}(\rn)$ and
\begin{align}\label{shediao}
\|h\|_{H_{X,\,L}(\rn)}\leq C\|h\|_{H_{X,\,L,\,\mathrm{Riesz}}(\rn)}.
\end{align}
\end{theorem}

\begin{remark}
When $X:=L^p(\rn)$ with $p\in(nr/(n+mr),1],$ Theorem \ref{riesz} is just \cite[Propostion 6.6]{CY12}.
\end{remark}
The proofs of Theorems \ref{r} and \ref{riesz} are given, respectively, in Subsections
\ref{s1} and \ref{s2} below.

\subsection{Proof of Theorem \ref{r}}\label{s1}

In this subsection, we prove Theorem \ref{r}. To this end, we need the molecular characterization
of the Hardy space $H_{X}(\rn)$ which was obtained in \cite{SHYY17}.

Let $\Phi\in\cs(\rn)$ satisfy that $\int_\rn\Phi(x)\,dx\neq0$ and $\Phi_t(x):=t^{-n}\Phi(x/t)$
for any $x\in\rn$ and $t\in(0,\infty).$ For any $f\in\cs'(\rn)$ and $x\in\rn,$ let
$$M_{\mathrm{b}}^{**}(f,\Phi)(x):=
\sup_{(y,t)\in\mathbb{R}^{n+1}_+}\frac{|(\Phi_t*f)(x-y)|}{(1+t^{-1}|y|)^b}$$
with $b\in(0,\infty)$ sufficiently large, where, for any $t\in(0,\fz)$, $\Phi_t(\cdot):=\frac{1}{t^n}\Phi(\frac{\cdot}{t})$
and $\Phi_t*f$ denotes the convolution of $\Phi_t$ and $f$.

\begin{definition}\label{def-HX}
Let $X$ be a ball quasi-Banach function space on $\rn$. Then the \emph{Hardy space} $H_X(\rn)$
associated with $X$ is defined as
\begin{align*}
H_X(\rn):=\left\{f\in\cs'(\rn):\ \|f\|_{H_X(\rn)}:=\left\|M_{\mathrm{b}}^{**}
(f,\Phi)\right\|_X<\infty\right\}.
\end{align*}
\end{definition}

\begin{definition}
Let $X$ be a ball quasi-Banach function space satisfying Assumption \ref{vector} for some $\tz,s\in(0,1]$,
$q\in[1,\infty]$, $d\in\mathbb{N}\cap[d_X,\infty)$, and $\tau\in(0,\fz)$, where
$d_X:=\lfloor n(1/\theta-1)\rfloor.$ A measurable function $m$ on $\rn$
is called an \emph{$(X,\,q,\,d,\,\tau)$-molecule} centered at a cube $Q\in{\mathcal Q}$ if it satisfies
that, for any $j\in\zz_+$,
\begin{equation*}
\left\|\mathbf{1}_{U_j(Q)}m\right\|_{L^{q}(\rn)}\le2^{-\tau j}
\frac{|Q|^{\frac{1}{q}}}{\|\mathbf{1}_{Q}\|_{X}}
\end{equation*}
and, for any $\al\in\zz_+^n$ with $|\al|\leq d,$
\begin{align*}
\int_\rn m(x)x^\al\,dx=0.
\end{align*}
Similarly, in the above definition, if any $Q\in\mathcal{Q}$ is replaced
by any ball $B$, then one obtains the definition of an
\emph{$(X,\,q,\,d,\,\tau)$-molecule} centered at a ball $B$.
\end{definition}

The following molecular characterization of the Hardy space $H_{X}(\rn)$ is just \cite[Theorem 3.9]{SHYY17}.

\begin{lemma}\label{thm:3-4-3}
Assume that $X$ is a  ball quasi-Banach function space satisfying
both Assumptions \ref{vector} and \ref{vector2} for some $\tz,s \in (0,1]$ and $q\in(1,\infty].$
Let $\tau\in(n[1/\theta-1/q],\fz)$. Then $f\in H_X(\rn)$ if and only if there exists a
sequence $\{m_j\}_{j=1}^\fz$ of $(X,\,q,\,d_X,\,\tau)$-molecules
centered, respectively,  at the cubes $\{Q_j\}_{j=1}^\fz\subset {\mathcal Q}$ and $\{\lambda_{j}\}_{j=1}^\infty\subset[0,\fz)$
satisfying
\begin{align*}
\left\|\left\{\sum_{j=1}^\infty\left(\frac{\lambda_j}{\|\mathbf{1}_{Q_j}\|_X}\right
)^s\mathbf{1}_{Q_j}\right\}^{1/s}\right\|_X<\fz
\end{align*}
such that
\begin{equation*}
f=\sum_{j=1}^\infty \lambda_{j}m_{j}
\end{equation*}
in ${\mathcal S}'(\rn).$ Moreover,
\begin{equation*}
\|f\|_{H_X(\rn)}\sim\left\|\left\{\sum_{j=1}^\infty
\left(\frac{\lambda_{j}}{\|\mathbf{1}_{Q_{j}}\|_{X}}\right)^{s}\mathbf{1}_{Q_{j}}
\right\}^{\frac{1}{s}}\right\|_{X},
\end{equation*}
where the  positive equivalence constants are independent of $f$.
\end{lemma}

To prove Theorem \ref{r}, we need the following conclusion whose proof is a slight modification
of the proof of \cite[Lemma 6.1]{CY12};
we omit the details here.

\begin{lemma}\label{guoqing2}
Let  $m\in\nn$, $L$ be a homogeneous divergence form $2m$-order elliptic operator  in \eqref{high}
satisfying Ellipticity Condition \ref{ec}, $M\in\nn$, and $q\in(p_-(L),
\min\{p_{+}(L),\,q_{+}(L)\}).$  Then there exists a positive constant $C$ such that, for any
$t\in(0,\infty),$ any closed subsets $E,F$ of $\rn$ with $\dist(E,F)\in(0,\infty)$, and any
$f\in L^q(\rn)$ with $\supp\,(f)\subset E,$
\begin{align*}
\left\|\nabla^m L^{-1/2}\left(I-e^{-tL}\right)^{M}(f)\right\|_{L^q(F)}
\leq C \left(\frac{t}{[\dist(E,F)]^{2m}}\right)^{M}\|f\|_{L^q(E)}
\end{align*}
and
\begin{align*}
\left\|\nabla^m L^{-1/2}\left(tLe^{-tL}\right)^{M}(f)\right\|_{L^q(F)}
\leq C \left(\frac{t}{[\dist(E,F)]^{2m}}\right)^{M}\|f\|_{L^q(E)}.
\end{align*}
\end{lemma}

Furthermore, we also need the following boundedness of $\nabla^m L^{-1/2},$ which can be found
in \cite[p.\,68]{A02}.

\begin{lemma}\label{b}
Let $m\in\nn$ and $L$ be a homogeneous divergence form $2m$-order elliptic operator  in \eqref{high}
satisfying Ellipticity Condition \ref{ec}. Assume that $q\in(q_-(L),q_+(L))$.
Then the Riesz transform $\nabla^m L^{-1/2}$ is bounded on $L^q(\rn).$
\end{lemma}

Now, we prove Theorem \ref{r} by using Lemmas \ref{thm:3-4-3}, \ref{guoqing2}, and \ref{b}.

\begin{proof}[Proof of Theorem \ref{r}]
Let all the symbols be the same as in the present theorem, $f\in H_{X,\,L}(\rn)\cap L^2(\rn),$ $\epsilon\in
(\max\{n/\theta,\,n/q\},\infty)$, and  $M\in \nn\cap(n[1/(2m\theta)-1/(2mq)],\infty)$.
By Proposition \ref{prop:6-2-2}, we find that there exists a  sequence $\{\ld_j\}_{j\in\nn}\subset [0,\infty)$
and a sequence $\{\al_j\}_{j\in\nn}$ of $(X,\,M,\,\epsilon)_L$-molecules associated, respectively, with the
balls $\{B_j\}_{j\in\nn}$ such that
\begin{align}\label{1600}
f=\sum_{j=1}^\infty\ld_j\al_j,
\end{align}
in $L^2(\rn)$, and
\begin{align*}
\left\|\left\{\sum_{j=1}^\infty
\left(\frac{\lambda_j}{\|\mathbf{1}_{B_j}\|_X}\right)^s\mathbf{1}_{B_j}\right\}^{\frac{1}{s}}
\right\|_X\lesssim\|f\|_{H_{X,\,L}(\rn)}.
\end{align*}
Then, from \eqref{1600} and Lemma \ref{b}, it follows that
\begin{align*}
\nabla^m L^{-1/2}(f)=\sum_{j=1}^\infty\ld_j\nabla^m L^{-1/2}(\al_j)
\end{align*}
in $L^2(\rn)$.

Next, we show that, for any $(X,\,M,\,\epsilon)_L$-molecule $b$ associated with
the ball $B:=B(x_B,\,r_B)\subset\rn$ with some $x_B\in\rn$ and  $r_B\in(0,\infty),$ and for any $j\in\zz_+,$
\begin{align}\label{pa}
\left\|\nabla^m L^{-1/2}(b)\right\|_{L^q(U_j(B))}
\lesssim2^{-j\tau}|B|^{1/q}\|\mathbf{1}_B\|_X^{-1},
\end{align}
where both the implicit positive constant and $\tau\in(n[1/\theta-1/q],\infty)$ are independent of $m$.

When $j=0,$ by Lemma \ref{b}, Definition \ref{buguoruci}(i), and $\epsilon\in(n/q,\infty)$, we conclude that
\begin{align*}
\left\|\nabla^m L^{-1/2}(b)\right\|_{L^q(U_j(B))}
&\lesssim\|b\|_{L^q(\rn)}
=\sum_{j=0}^\infty \|b\|_{L^q(U_j(B))}\\
&\lesssim\sum_{j=0}^\infty2^{-j\epsilon} |2^jB|^{1/q}\|\mathbf{1}_B\|_X^{-1}
\lesssim |B|^{1/q}\|\mathbf{1}_B\|_X^{-1}.
\end{align*}
Furthermore, for any $j\in\nn,$ we have
\begin{align}\label{jugg}
&\left\|\nabla^m L^{-1/2}(b)\right\|_{L^q(U_j(B))}\\\noz
&\quad\leq
\left\|\nabla^m L^{-1/2}\left(I-e^{-r_B^{2m}L}\right)^{M}(b)\right\|_{L^q(U_j(B))}\\\noz
&\quad\quad+
\left\|\nabla^m L^{-1/2}\left[I-\left(I-e^{-r_B^{2m}L}\right)^{M}\right](b)\right\|_{L^q(U_j(B))}\\\noz
&\quad\lesssim
\left\|\nabla^m L^{-1/2}\left(I-e^{-r_B^{2m}L}\right)^{M}(b)\right\|_{L^q(U_j(B))}\\\noz
&\quad\quad+\sum_{k=1}^{M}\left\|\nabla^m L^{-1/2}\left(\frac{k}{M}r_B^{2m}L
e^{-(k/M)r_B^{2m}L}\right)^{M}\lf(r_B^{-2m}L^{-1}\r)^{M}(b)\right\|_{L^q(U_j(B))}\\\noz
&\quad=:\mathrm{I}+\mathrm{II}.
\end{align}
For any $j\in\nn,$ let $S_j(B):=(2^{j+2}B)\setminus(2^{j-1}B).$ It is obvious that
$\dist([S_j(B)]^\complement,U_j(B))\sim 2^jr_B.$ From this, Lemmas \ref{guoqing2} and \ref{b},
Proposition \ref{basic}(iii) with $p:=q$, and Definition \ref{buguoruci}(i), we deduce that
\begin{align}\label{1}
\mathrm{I}&\lesssim \left\|\nabla^m L^{-1/2}\left(I-e^{-r_B^{2m}L}\right)^{M}\lf(b
\mathbf{1}_{S_j(B)}\r)\right\|_{L^q(U_j(B))}\\\noz
&\quad+\left\|\nabla^m L^{-1/2}\left(I-e^{-r_B^{2m}L}\right)^{M}\lf(b
\mathbf{1}_{[S_j(B)]^\complement}\r)\right\|_{L^q(U_j(B))}\\\noz
&\lesssim \left\|\left(I-e^{-r_B^{2m}L}\right)^{M}\lf(b\mathbf{1}_{S_j(B)}\r)\right\|_{L^q(\rn)}
+2^{-2mMj}\|b\|_{L^q(\rn)}\\\noz
&\lesssim
\sum_{k=0}^M \left\|e^{-kr_B^{2m}L}\lf(b\mathbf{1}_{S_j(B)}\r)\right\|_{L^q(\rn)}
+\sum_{i=0}^\infty2^{-2mMj}\|b\|_{L^q(U_i(B))}\\\noz
&\lesssim \|b\|_{L^q(S_j(B))}+\sum_{i=0}^\infty2^{-2mMj}\|b\|_{L^q(U_i(B))}
\lesssim\left[2^{-j(\epsilon-n/q)}+2^{-2mMj}\right]|B|^{1/q}\|\mathbf{1}_B\|_X^{-1}.
\end{align}
Similarly to \eqref{1}, we also have
\begin{align*}
\mathrm{II}
&\lesssim \sum_{k=1}^{M}\left\|\nabla^m L^{-1/2}\left(\frac{k}{M}r_B^{2m}Le^{-(k/M)
r_B^{2m}L}\right)^{M}\left[\left(r_B^{-2m}L^{-1}\right)^{M}(b)\mathbf{1}_{S_j(B)}\right]
\right\|_{L^q(U_j(B))}\\
&\quad+\sum_{k=1}^{M}\left\|\nabla^mL^{-1/2}\left(\frac{k}{M}r_B^{2m}L
e^{-(k/M)r_B^{2m}L}\right)^{M}\left[\left(r_B^{-2m}L^{-1}\right)^{M}(b)\mathbf{1}_{
[S_j(B)]^\complement}\right]\right\|_{L^q(U_j(B))}\\
&\lesssim \left\|\left(r_B^{-2m}L^{-1}\right)^{M}(b)\right\|_{L^q(S_j(B))}+2^{-2mMj}
\left\|\left(r_B^{-2m}L^{-1}\right)^M(b)\right\|_{L^q(\rn)}\\
&\lesssim \left\|\left(r_B^{-2m}L^{-1}\right)^{M}(b)\right\|_{L^q(S_j(B))}+\sum_{i=0}^\infty2^{-2mMj}
\left\|\left(r_B^{-2m}L^{-1}\right)^M(b)\right\|_{L^q(U_i(B))}\\
&\lesssim\left[2^{-j(\epsilon-n/q)}+2^{-2mMj}\right]|B|^{1/q}\|\mathbf{1}_B\|_X^{-1},
\end{align*}
which, combined with both \eqref{1} and \eqref{jugg}, further implies that \eqref{pa} holds true with
$\tau:=\min\{\epsilon-n/q,2mM\}\in(n[1/\theta-1/q],\infty).$

Moreover, by \cite[p.\,1434]{CY12}, we find that, for any $(X,\,M,\,\epsilon)_L$-molecule $b$ and
any $\al\in\zz_+^n$ with $|\al|\leq m-1$,
\begin{align*}
\int_\rn b(x)x^\al\,dx=0,
\end{align*}
which, together with the assumption that $\theta\in(n/(n+m),1]$, \eqref{pa}, and Lemma
\ref{thm:3-4-3}, further implies that
\begin{align*}
\left\|\nabla^m L^{-1/2}(f)\right\|_{H_X(\rn)}
\lesssim\left\|\left\{\sum_{j=1}^\infty
\left(\frac{\lambda_j}{\|\mathbf{1}_{B_j}\|_X}\right)^s\mathbf{1}_{B_j}\right\}^{
\frac{1}{s}}\right\|_X\lesssim\|f\|_{H_{X,\,L}(\rn)}.
\end{align*}
This, combined with the fact that $H_{X,\,L}(\rn)\cap L^2(\rn)$ is dense in $H_{X,\,L}(\rn)$, then
finishes the proof of Theorem \ref{r}.
\end{proof}

\subsection{Proof of Theorem \ref{riesz}}\label{s2}
In this subsection, we prove Theorem \ref{riesz}. We first establish the atomic decomposition
of the homogeneous Hardy--Sobolev space $\dot{H}_{m,\,X}(\rn)$.

\begin{definition}
Let $X$ be a ball quasi-Banach function space on $\rn$ and $m\in\nn$. Then the \emph{homogeneous
Hardy--Sobolev space} $\dot{H}_{m,\,X}(\rn)$, associated with $X$, is defined as
$$
\dot{H}_{m,\,X}(\rn):=\left\{f\in\cs'(\rn): \|f\|_{\dot{H}_{m,\,X}(\rn)}:=
\sum_{\al\in\zz_+^n,\,|\al|=m}\left\|\p^\al f\right\|_{H_X(\rn)}<\infty\right\}.
$$
\end{definition}
Let us recall the definition of  weak derivatives of the locally 
integrable function. For any $\al\in\zz_+^n$
and any locally integrable functions $f$ and $g$, 
$g$ is called the \emph{$\al$-order weak derivative} of $f$ if, 
for any $\varphi\in C_\mathrm{c}^\infty(\rn)$ (the set of all the 
infinitely differentiable functions on $\rn$ with compact support),
\begin{align*}
	\int_\rn f(x)\p^\al \varphi(x)\,dx=(-1)^{|\al|}\int_\rn g(x)\varphi(x)\,dx.
\end{align*}
Moreover, we denote $g$ by $\p^\al f.$
\begin{definition}\label{zhanshen}
Let $X$ be a ball quasi-Banach function space, $m\in\nn,$ and $p\in(1,\infty).$ Then a function $a$
is called an \emph{$(\dot{H}_{m,\,X},\,p)$-atom} if
there exists a ball $B\subset\rn$ such that,
for any $\al\in\zz_+^n$ with $|\al|=m$,
\begin{itemize}
\item [$\mathrm{(i)}$] $\supp\,(a):=\{x\in\rn:\  a(x)\not=0\}\subset B;$
\item [$\mathrm{(ii)}$] $\| a\|_{L^p(\rn)}\leq |B|^{1/p+m/n}\|\mathbf{1}_B\|_X^{-1};$
\item [$\mathrm{(iii)}$] $\|\p^\al a\|_{L^p(\rn)}\leq |B|^{1/p}\|\mathbf{1}_B\|_X^{-1}.$
\end{itemize}
\end{definition}

Then we have the following atomic decomposition theorem for the space $\dot{H}_{m,\,X}(\rn)$.
\begin{theorem}\label{HS}
Let $m\in\nn$, $p\in(1,\infty)$, and
$X$ be  a  ball quasi-Banach function space satisfying both Assumptions
\ref{vector} and \ref{vector2} for some $\theta,s \in (0,1]$ and   $q\in(1,\infty).$ Assume that $f\in \dot{H}_{m,\,X}(\rn)$ and $\p^\al f\in  L^2(\rn)$ for
any $\al\in\zz_+^n$ with $|\al|=m.$
Then there exists a sequence $\{\lambda_j\}_{j=1}^\infty\subset[0,\fz)$ and
a sequence $\{a_j\}_{j=1}^\infty$ of $(\dot{H}_{m,\,X},\,p)$-atoms associated, respectively, with the
balls $\{B_j\}_{j=1}^\infty$ such that, for any $\al\in\zz_+^n$ with $|\al|=m,$
\begin{align}\label{jishu}
\p^\al f=\sum_{j=1}^\infty \ld_j \p^\al a_j,
\end{align}
in  $L^2(\rn)$, and
\begin{align}\label{norm}
\lf\|\lf\{\sum_{j=1}^\fz
\lf(\frac{\lz_j}{\|\mathbf{1}_{B_j}\|_{X}}\r)^s\mathbf{1}_{B_j}\r\}^{1/s}\r\|_X\lesssim
\| f\|_{\dot{H}_{m,\,X}(\rn)},
\end{align}
where the implicit positive constant is independent of $f$.
\end{theorem}

Before proving Theorem \ref{HS}, let us first recall the concept of the tent space.
For an open set $O\subset{\mathbb R}^n$, define the \emph{tent $\widehat{O}$} over $O$ by setting
$$
\widehat{O}:=\{(x,t) \in {\mathbb R}^{n+1}_+:\ B(x,t) \subset O\}.
$$

Coifman et al. \cite{CMS85} introduced the tent space $T^p(\rr^{n+1}_+)$ for any given $p\in(0,\fz)$.
Recall that a measurable function $g$ is said to belong to the \emph{tent space} $T^p(\rr^{n+1}_+)$,
with $p\in(0,\fz)$, if $\|g\|_{T^p(\rr^{n+1}_+)}:=\|\ca(g)\|_{L^p(\rn)}<\fz$, where $\mathcal{A}$
is the same as in \eqref{zhangpeng} with $\al:=1.$

For a given ball quasi-Banach function space $X$, the \emph{$X$-tent space} $T_X(\rr^{n+1}_+)$
is defined to be the set of all the  measurable functions
$g:\ {\mathbb R}^{n+1}_+ \to {\mathbb C}$ with the finite \emph{quasi-norm}
$$\|g\|_{T_X(\rr^{n+1}_+)}:=\|{\mathcal A}(g)\|_X$$
(see \cite[p.\,28]{SHYY17}).
We need the atomic decomposition
of the $X$-tent space $T_X(\mathbb{R}^{n+1}_+),$ which is a part of   \cite[Theorem 3.19]{SHYY17}.

\begin{definition}\label{defi:3-6-2}
Let $X$ be a ball quasi-Banach function space and $p \in(1,\infty)$. A measurable function
$a:\ {\mathbb R}^{n+1}_+ \to {\mathbb C}$ is called a \emph{$(T_X,\,p)$-atom} if there
exists a ball $B\subset{\mathbb R}^n$ such that
\begin{enumerate}
\item[(i)] ${\rm supp}\,(a):=\{(x,t)\in{\mathbb R}^{n+1}_+:\ a(x,t)\neq0\} \subset \widehat{B}$,

\item[(ii)] $\|a\|_{T^p(\rr^{n+1}_+)} \le \dfrac{|B|^{1/p}}{\|\mathbf{1}_B\|_X}$.
\end{enumerate}
Furthermore, if $a$ is a $(T_X,\,p)$-atom for any $p\in (1,\fz)$, then $a$ is called a \emph{$(T_X,\fz)$-atom}.
\end{definition}

\begin{lemma}\label{thm:3-6-1}
Assume that $X$ is a  ball quasi-Banach function space satisfying
both Assumptions \ref{vector} and \ref{vector2} for some $\theta,s \in (0,1]$ and
$q\in(1,\infty].$ Let $f\in T_X(\mathbb{R}^{n+1}_+)$. Then there exists a sequence
$\{\lambda_j\}_{j=1}^\infty\subset[0,\fz)$ and a sequence $\{a_j\}_{j=1}^\infty$ of
$(T_X,\,\infty)$-atoms associated, respectively,  with the balls $\{B_j\}_{j=1}^\infty$ such that,
for almost every $(x,t)\in{\mathbb R}^{n+1}_+$,
\begin{equation*}
f(x,t)=\sum_{j=1}^\infty \lambda_j a_j(x,t), \quad |f(x,t)|=\sum_{j=1}^\infty \lambda_j|a_j(x,t)|,
\end{equation*}
and
\begin{align*}
\lf\|\lf\{\sum_{j=1}^\fz
\lf(\frac{\lz_j}{\|\mathbf{1}_{B_j}\|_{X}}\r)^s\mathbf{1}_{B_j}\r\}^{1/s}\r\|_X
\lesssim\|f\|_{T_X(\rr^{n+1}_+)},
\end{align*}
where the implicit positive constant is independent of $f$.
\end{lemma}
Let $\psi\in \mathcal{S}(\rn)$ with $\int_{\rn}\psi(x)\,dx=0.$
For any $g\in T^p(\mathbb{R}^{n+1}_+)$ with compact support and for any $x\in\rn$, let
\begin{align}\label{pisuanzi}
\pi_{\psi}(g)(x):=\int_0^\infty g(\cdot,t)\ast\psi_t(x)\,\frac{dt}{t}.
\end{align}
By \cite[Theorem 6(1)]{CMS85}, we have the following lemma.
\begin{lemma}\label{touying}
Let $p\in(1,\infty)$,
$\psi\in \mathcal{S}(\rn)$ with $\int_{\rn}\psi(x)\,dx=0,$
and $\pi_\psi$
be the same as in \eqref{pisuanzi}. Then the operator $\pi_\psi$ can extend to be
a bounded liner operator from
$T^p(\mathbb{R}^{n+1}_+)$ to $L^p(\rn).$
\end{lemma}

Now, we show Theorem \ref{HS} via
using Lemmas \ref{thm:3-6-1} and \ref{touying}.
\begin{proof}[Proof of Theorem \ref{HS}]
Let all the symbols be the same as in the present theorem.
By the proof of \cite[Lemma 1.1]{FJW91}, we conclude that there
exists a radial function $\varphi\in C^\infty_{\mathrm{c}}(\rn)$ such that
$\supp\,(\varphi)
\subset B(\mathbf{0},1)$,
$\int_{\rn}\varphi(x)\,dx=0,$
and, for any $\xi\in \rn\setminus\{\mathbf{0}\}$,
\begin{align*}
\int_0^\infty t^{2m-1}|\xi|^{2m}
\left[\widehat{\varphi}(t\xi)\right]^2\,dt=1,
\end{align*}
where $\widehat{\varphi}$ denotes the \emph{Fourier transform} of $\varphi,$
namely, for any $\xi\in\rn,$
$$\widehat{\varphi}(\xi):=\int_\rn\varphi(x)e^{-2\pi ix\cdot \xi}\,dx.$$
For any measurable function $h$ on $\rn$ and any $t\in(0,\infty),$ let
$h_t(\cdot):=t^{-n}h(\cdot/t).$
For any $(x,t)\in\mathbb{R}^{n+1}_+,$ let
$$F(x,t):= \sum_{\al\in\zz_+^n,\,|\al|=m}\frac{m!}{\al!}(\p^\al f)\ast
\left(\p^\al \varphi\right)_t(x).$$
For any  $g\in H_X(\rn)$,
any $\psi\in \mathcal{S}(\rn)$ with
$\int_\rn\psi(x)\,dx=0,$ and any
$x\in\rn,$ let
\begin{align*}
S_{\psi}(g)(x):=\left\{\iint_{\Gamma(x)}\left|g\ast\psi_t(y)\right|^2\,\frac{dy\,dt}{t^{n+1}}\right\}^{1/2}.
\end{align*}
Let $\al\in\zz_+^n$ with $|\al|=m.$ By the proof of \cite[Theorem 3.21]{SHYY17}, we conclude
that $S_{\p^\al\varphi}$ is bounded from $H_X(\rn)$ to $X$. Therefore,
\begin{align*}
\left\|S_{\p^\al\varphi}\lf(\p^\al f\r)\right\|_X\lesssim \lf\|\p^\al f\r\|_{H_X(\rn)},
\end{align*}
which further implies that
\begin{align}\label{woniu}
\|F\|_{T_X(\mathbb{R}^{n+1}_+)}\lesssim \|f\|_{\dot{H}_{m,\,X}(\rn)}.
\end{align}
Moreover, from Lemma \ref{thm:3-6-1}, we deduce that there exists a sequence
$\{\lambda_j\}_{j=1}^\infty\subset[0,\fz)$ and a sequence $\{b_j\}_{j=1}^\infty$ of
$(T_X,\,\infty)$-atoms associated, respectively, with the balls $\{B_j\}_{j=1}^\infty$ such that,
for almost every $(x,t)\in{\mathbb R}^{n+1}_+$,
\begin{equation*}
F(x,t)=\sum_{j=1}^\infty \lambda_j b_j(x,t), \quad |F(x,t)|=\sum_{j=1}^\infty \lambda_j|b_j(x,t)|,
\end{equation*}
and
\begin{align*}
\lf\|\lf\{\sum_{j=1}^\fz
\lf(\frac{\lz_j}{\|\mathbf{1}_{B_j}\|_{X}}\r)^s\mathbf{1}_{B_j}\r\}^{1/s}\r\|_X
\lesssim\|F\|_{T_X(\rr^{n+1}_+)}.
\end{align*}
This, combined with \eqref{woniu}, implies that \eqref{norm} holds true.

For any   $j\in\nn$ and $x\in\rn,$
let
$$a_j(x):=\int_0^\infty b_j(\cdot,t)\ast\varphi_t(x)t^{m-1}\,dt.$$
Next, we prove that, for any $j\in\nn,$ $a_j$ is an $(\dot{H}_{m,\,X},\,p)$-atom up to a harmless
constant multiple.

Let $j\in\nn.$
Since $b_j$ is supported in $\widehat{B_j}$  and $\varphi$
is supported in $B(\mathbf{0},1)$, it follows that $\supp\,(a_j)\subset B_j.$
We first show that
\begin{align}\label{yuanzi3}
	\|a_j\|_{L^p(\rn)}\lesssim |B_j|^{1/p+m/n}\|\mathbf{1}_{B_j}\|_X^{-1},
\end{align}
where the implicit positive constant is independent of $j$.
By the Tonelli theorem,  the Fubini
theorem, the H\"older inequality,
and the fact that $\supp \,(b_j)\subset
\widehat{B_j}$,
we find that, for any $h\in L^{p'}(\rn),$
\begin{align*}
\lf|\int_{\rn}a_j(x)h(x)\,dx\r|
&=\left|\int_\rn\int_0^\infty b_j(\cdot,t)\ast\varphi_t(x)t^{m-1}\,dt h(x)\,dx\right|\\
&=\left|\int_0^\infty\int_\rn b_j(\cdot,t)\ast\varphi_t(x)h(x)\,dxt^{m-1}\,dt \right|\\
&=\lf|\int_0^\infty\int_{\rn} \int_\rn \varphi_t(x-y)b_j(y,t)\,dyh(x)\,dxt^{m-1}\,dt\r|\\
&=\lf|\int_0^\infty\int_\rn h\ast\varphi_t(y)b_j(y,t)\,dyt^{m-1}\,dt\r|\\
&\lesssim
r_{B_j}^m\int_{\mathbb{R}^{n+1}_+}\int_{\rn}\mathbf{1}_{B(y,t)}(x)\,dx
\left|h\ast\varphi_t(y)b_j(y,t)\right|\frac{dydt}{t^{n+1}}\\
&\sim r_{B_j}^m\int_\rn\iint_{\Gamma(x)}
\lf|h\ast\varphi_t(y)b_j(y,t)\r|\,\frac{dydt}{t^{n+1}}\,dx\\
&\lesssim r_{B_j}^m\int_\rn S_{\varphi}(h)(x)\mathcal{A}(b_j)(x)\,dx,
\end{align*}
here and thereafter, $\frac{1}{p}+\frac{1}{p'}=1.$
Using this, the H\"older inequality,  and the facts that $S_\varphi$ is bounded on $L^{p'}(\rn)$
(see, for instance, \cite[Theorem 7.8]{FS82}) and that $b_j$ is a $(T_X,\,\infty)$-atom, we conclude that
\begin{align*}
\left|\int_{\rn}a_j(x)h(x)\,dx\right|
&\lesssim r_{B_j}^m\|b_j\|_{T^p(\mathbb{R}^{n+1}_+)}\lf\|S_{\varphi}(h)\r\|_{L^{p'}(\rn)}
\lesssim |B_j|^{1/p+m/n}\|\mathbf{1}_{B_j}\|_X^{-1}\|h\|_{L^{p'}(\rn)},
\end{align*}
which further implies that
\eqref{yuanzi3} holds true.

Then we prove that, for any $\al\in\zz_+^n$ with $|\al|=m$,
\begin{align}\label{ruodaoshu}
	\p^\al a_j=\pi_{\p^\al \varphi}(b_j)
\end{align}
and
\begin{align}\label{xiyang}
\lf\|\p^\al a_j\r\|_{L^p(\rn)}
\lesssim
|B_j|^{1/p}\|\mathbf{1}_{B_j}\|_X^{-1},
\end{align}
where the implicit positive constant is independent of $j$.
Let $\al\in\zz_+^n$ with $|\al|=m$
and $\pi_{\p^\al\varphi}$ be the same as \eqref{pisuanzi} with $\psi$ replaced by
$\p^\al\varphi.$
Let $b_{j,k}(y,t):=b_{j}(y,t)\mathbf{1}_{(1/k,k)}(t)$ for any $(y,t)\in\mathbb{R}^{n+1}_+$ and $k\in\nn.$
It is obvious that $b_{j,k}\to b_{j}$
in $T^p(\mathbb{R}^{n+1}_+)$ as $k\to\infty.$
By this and Lemma \ref{touying} with
$\psi$ replaced by
$\p^\al\varphi$, we obtain that
$\pi_{\p^\al\varphi}(b_{j,k})
\to\pi_{\p^\al \varphi}(b_{j})$
in $L^p(\rn)$ as $k\to\infty.$
This, together with the H\"older inequality, implies that, for any $\psi\in C_{\mathrm{c}}^\infty(\rn)$,
\begin{align*}
\lim_{k\to\infty}\int_\rn\pi_{\p^\al \varphi}(b_{j,k})(x)\psi(x)\,dx=
\int_\rn\pi_{\p^\al \varphi}(b_j)(x)\psi(x)\,dx.
\end{align*}
Using this,
the Tonelli theorem, and the Fubini
theorem, we find that, for any
$\psi\in C_{\mathrm{c}}^\infty(\rn)$,
\begin{align*}
\int_\rn a_j(x)\p^\al \psi(x)\,dx
&=\int_0^\infty\int_\rn b_j(\cdot,t)\ast
\varphi_t(x)\p^\al\psi(x)\,dxt^{m-1}\,dt\\
&=\lim_{k\to\infty}\int_{1/k}^k\int_\rn\int_\rn\varphi_t(x-y)b_j(y,t)\,dy\p^\al\psi(x)\,dxt^{m-1}\,dt\\
&=\lim_{k\to\infty}\int_{1/k}^k\int_\rn\int_\rn\varphi_t(x-y)\p^\al\psi(x)\,dxb_j(y,t)\,dyt^{m-1}\,dt\\
&=(-1)^m\lim_{k\to\infty}\int_{1/k}^k\int_\rn\int_\rn(\p^\al \varphi)_t(x-y)
\psi(x)\,dxb_j(y,t)\,dy\,\frac{dt}{t}\\
&=(-1)^m\lim_{k\to\infty}\int_{1/k}^k\int_\rn\int_\rn(\p^\al \varphi)_t(x-y)b_j(y,t)\,dy
\psi(x)\,dx\,\frac{dt}{t}\\
&=(-1)^m\lim_{k\to\infty}\int_\rn\int_{1/k}^k(\p^\al\varphi)_t\ast b_j(\cdot,t)(x)\,
\frac{dt}{t}\psi(x)\,dx\\
&=(-1)^m\int_\rn\pi_{\p^\al \varphi}(b_j)(x)\psi(x)\,dx.
\end{align*}
This implies that \eqref{ruodaoshu} holds true.
Furthermore, applying Lemma
\ref{touying} with $\psi$ replaced by
$\p^\al\varphi$ and  the fact that
$b_j$ is a $(T_X,\,\infty)$-atom,
we find that
\begin{align*}
\lf\|\p^\al a_j\r\|_{L^p(\rn)}=\lf\|\pi_{\p^\al \varphi}(b_j)\r\|_{L^p(\rn)}\lesssim
\|b_j\|_{T^p(\mathbb{R}^{n+1}_+)}\lesssim |B_j|^{1/p}\|\mathbf{1}_{B_j}\|_X^{-1}.
\end{align*}
This proves \eqref{xiyang}. By \eqref{yuanzi3} and \eqref{xiyang}, we conclude that $a_j$ is an $(\dot{H}_{m,\,X},\,p)$-atom up to a harmless
constant multiple.

To complete the proof of Theorem \ref{HS}, it suffices to show that \eqref{jishu} holds true.
For any $\al\in\zz_+^n$ with $|\al|=m$,
since $S_{\p^\al\varphi}$ is bounded
on $L^2(\rn)$ (see, for instance,
\cite[Theorem 7.8]{FS82}),
it follows that
\begin{align*}
\lf\|S_{\p^\al\varphi}\lf(\p^\al f\r)\r\|_{L^2(\rn)}\lesssim\lf\|\p^\al f\r\|_{L^2(\rn)}.
\end{align*}
Thus, $F\in T^2(\mathbb{R}^{n+1}_+)$. Moreover, by \cite[Lemma 6.9(i)]{SHYY17}, we conclude that
\begin{align}\label{jieyao}
F=\sum_{j=1}^\infty \lambda_j b_j
\end{align}
in $T^2(\mathbb{R}^{n+1}_+).$
Let $\beta\in\zz_+^n$ with $|\beta|=m$.
From \eqref{jieyao}, Lemma
\ref{touying} with $\psi$ replaced by
$\p^\al\varphi$, and \eqref{ruodaoshu},
we deduce that
\begin{align}\label{fuq}
\pi_{\p^\beta \varphi}(F)=
\sum_{j=1}^\infty\ld_j\pi_{\p^\beta \varphi}(b_j)=
\sum_{j=1}^\infty\ld_j\p^\beta a_j
\end{align}
in $L^2(\rn).$
By an argument similar to that used in the proof of \cite[Theorem 1.2]{FJW91},
we find that
\begin{align}\label{moyan}
\lim_{k\to\infty}\sum_{\al\in\zz_+^n,\,|\al|=m}\frac{m!}{\al!}\int_{1/k}^k \p^\beta f
\ast\lf(\p^\al \varphi\r)_t\ast \lf(\p^\al
\varphi\r)_t\,\frac{dt}{t}= \p^\beta f
\end{align}
in $L^2(\rn).$ On the other hand, applying an argument similar to that used in the estimation of
\eqref{yuanzi3}, we obtain
\begin{align*}
\lim_{k\to\infty}\int_{1/k}^kF(\cdot,t)\ast\lf(\p^\beta \varphi\r)_t\,\frac{dt}{t}
=\pi_{\p^\beta \varphi}(F)
\end{align*}
in $L^2(\rn),$
which, together with the fact that
\begin{align*}
\int_{1/k}^kF(\cdot,t)\ast\lf(\p^\beta \varphi\r)_t\,\frac{dt}{t}
&=
\sum_{\al\in\zz_+^n,\,|\al|=m}\frac{m!}{\al!}\int_{1/k}^k \p^\al f\ast\lf(\p^\al \varphi\r)_t
\ast \lf(\p^\beta \varphi\r)_t\,\frac{dt}{t}\\
&=\sum_{\al\in\zz_+^n,\,|\al|=m}\frac{m!}{\al!}\int_{1/k}^k \p^\beta f\ast\lf(\p^\al \varphi\r)_t
\ast \lf(\p^\al \varphi\r)_t\,\frac{dt}{t},
\end{align*}
further implies that
\begin{align*}
\lim_{k\to\infty}\sum_{\al\in\zz_+^n,\,|\al|=m}\frac{m!}{\al!}\int_{1/k}^k \p^\beta f\ast\lf(\p^\al \varphi\r)_t
\ast \lf(\p^\al \varphi\r)_t\,\frac{dt}{t}
=\pi_{\p^\beta \varphi}(F)
\end{align*}
in $L^2(\rn).$
By this and \eqref{moyan},
we conclude that,
for any $\beta\in\zz_+^n$ with $|\beta|=m,$
$\p^\beta f(x)=\pi_{\p^\beta
\varphi}(F)(x)$ for almost every $x\in\rn.$ From this and \eqref{fuq}, it follows that
\eqref{jishu} holds true. This finishes the proof of Theorem \ref{HS}.
\end{proof}

By an argument similar to that used in the proof of \cite[Lemma 2.26]{hmm11},
we have the following lemma; we omit  the details here.

\begin{lemma}\label{2.26}
Let $m\in\nn$, $L$ be a homogeneous divergence form $2m$-order elliptic operator  in
\eqref{high}  satisfying Ellipticity Condition  \ref{ec}, and $r\in[1,2]$. Assume
that the family  $\{e^{-tL}\}_{t\in(0,\infty)}$ of operator satisfies the  $m-L^r(\rn)-L^2(\rn)$
off-diagonal estimate. Then the family $\{tLe^{-tL}\}_{t\in(0,\infty)}$ of operators also satisfies the
$m-L^r(\rn)-L^2(\rn)$ off-diagonal estimate and is bounded from $L^r(\rn)$ to $L^2(\rn)$
with norm bounded by $Ct^{\frac{1}{2m}(\frac{n}{2}-\frac{n}{r})}$, where $C$ is a positive
constant independent of $t$.
\end{lemma}

For any  $h\in L^2(\rn)$ and $x\in\rn$, let
\begin{align*}
S_1(h)(x):=\left[\iint_{\Gamma(x)}\left|t^mL^{1/2}e^{-t^{2m}L}(h)(y)\right|^2\,
\frac{dy\,dt}{t^{n+1}}\right]^{1/2}.
\end{align*}
Then we have the following conclusion.
\begin{lemma}\label{pf}
Let $m\in\nn$ and $L$ be a homogeneous divergence form $2m$-order elliptic operator  in \eqref{high}
satisfying Ellipticity Condition \ref{ec}. Assume that $X$ is a ball quasi-Banach
function space satisfying both Assumptions \ref{vector} and
\ref{vector2} for some $\theta,s\in(0,1]$
and  $q\in[2,p_{+}(L)).$ Then there exists a positive constant $C$ such that,
for any $h\in L^2(\rn)$ with $\|S_1(h)\|_X<\infty,$
\begin{align*}
\|h\|_{H_{X,\,L}(\rn)}\leq C\|S_1(h)\|_X.
\end{align*}
\end{lemma}
\begin{proof}
Let all the symbols be the same as in the present lemma. Let $h\in L^2(\rn)$ with $\|S_1(h)\|_X<\infty,$
$\epsilon\in(n/\theta,\infty)$, and $M\in\nn$ be sufficiently large.  Then, repeating
the proof of \cite[Proposition 6.11]{SHYY17}, we conclude that there exists a sequence $\{\lz_j\}_{j\in\nn}
\subset[0,\fz)$ and a sequence $\{\az_j\}_{j\in\nn}$ of $(X,\,M,\,\epsilon)_L$-molecules
associated, respectively, with the balls $\{B_j\}_{j\in\nn}$ such that $h=\sum_{j=1}^\fz \lz_j\az_j$ in $L^2(\rn)$.
Moreover, there exists a positive constant $C$ such that, for any $h\in L^2(\rn)$ with $\|S_1(h)\|_X<\infty$,
$$\left\|\left\{\sum_{j=1}^\infty
\left(\frac{\lambda_j}{\|\mathbf{1}_{B_j}\|_X}\right)^s\mathbf{1}_{B_j}\right\}^{\frac{1}{s}}\right\|_X\le
C\|S_1(h)\|_X.$$
From this and \cite[Theorem 6.12]{SHYY17}, we deduce that $h\in H_{X,\,L}(\rn)$ and
\begin{align*}
\|h\|_{H_{X,\,L}(\rn)}\lesssim \left\|\left\{\sum_{j=1}^\infty
\left(\frac{\lambda_j}{\|\mathbf{1}_{B_j}\|_X}\right)^s\mathbf{1}_{B_j}\right\}^{\frac{1}{s}}
\right\|_X\lesssim\|S_1(h)\|_X.
\end{align*}
This finishes the proof of Lemma \ref{pf}.
\end{proof}

Now, we show Theorem  \ref{riesz} by using Theorems \ref{HS}, \ref{2.26}, and \ref{pf}.
\begin{proof}[Proof of Theorem \ref{riesz}]
Let all the symbols be the same as in the present theorem. Assume that $h\in L^2(\rn)\cap
H_{X,\,L,\,\mathrm{Riesz}}(\rn)$ and $f:=L^{-1/2}(h).$
By Lemma \ref{pf}, we find that
\begin{align*}
\|h\|_{H_{X,\,L}(\rn)}\lesssim \|S_1(h)\|_X.
\end{align*}
From this, we deduce that, to prove \eqref{shediao}, it suffices to show that
\begin{align}\label{shediao2}
\left\|S_1\lf(L^{1/2}f\r)\right\|_X\lesssim \| f\|_{\dot{H}_{m,\, X}(\rn)}.
\end{align}

By both the boundedness of the Riesz transform $\nabla^mL^{-1/2}$ on $L^2(\rn)$
(see Lemma \ref{b}) and the assumption that $h\in L^2(\rn)\cap H_{X,\,L,\,\mathrm{Riesz}}(\rn)$,
we conclude that $f\in \dot{H}_{m,\,X}(\rn)$ and $\p^\al f\in L^2(\rn)$
for any $\al\in\zz_+^n$ with $|\al|=m.$ From this and Theorem \ref{HS},
it follows that there exists a sequence $\{\lambda_j\}_{j=1}^\infty\subset[0,\fz)$ and
a sequence $\{a_j\}_{j=1}^\infty$ of $(\dot{H}_{m,\,X},\,q)$-atoms
associated, respectively,
with the balls $\{B_j\}_{j=1}^\infty$ such that, for any $\al\in\zz_+^n$ with $|\al|=m,$
\begin{align}\label{jishu10}
\p^\al f=\sum_{j=1}^\infty \ld_j \p^\al a_j,
\end{align}
in $L^2(\rn)$, and
\begin{align}\label{norm10}
\lf\|\lf\{\sum_{j=1}^\fz
\lf(\frac{\lz_j}{\|\mathbf{1}_{B_j}\|_{X}}\r)^s\mathbf{1}_{B_j}\r\}^{1/s}\r\|_X
\lesssim\| f\|_{\dot{H}_{m,\,X}(\rn)}.
\end{align}

For any ball $B\subset \rn$, let $S_0(B):=B$ and $S_i(B):=(2^{i}B)\setminus(2^{i-1}B)$
any $i\in\nn$.
Next, we prove that, for any $(\dot{H}_{m,\,X},\,q)$-atom $a$ associated with the ball $B$ and for
any $i\in\zz_+,$
\begin{align}\label{yida}
\left\|S_1\lf(L^{1/2}a\r)\right\|_{L^q(S_i(B))}\lesssim 2^{-i(1+n/r-n/q)}|B|^{1/q}\|\mathbf{1}_B\|_X^{-1},
\end{align}
where the implicit positive constant is independent of $a$ and $j\in\nn$.

When $i\in\{0,\,1,\,2\},$ by both the fact that $S_1$ is bounded in $L^q(\rn)$ (see, for instance,
\cite[Corollary 6.10 and p.\,67]{A02}) and the assumption that $a$ is an
$(\dot{H}_{m,\,X},\,q)$-atom, we find that
\begin{align}\label{meiguihua}
\left\|S_1\lf(L^{1/2}a\r)\right\|_{L^q(S_i(B))}
\lesssim \lf\|L^{1/2}a\r\|_{L^q(\rn)}
\lesssim \|\nabla^m a\|_{L^q(\rn)}\lesssim |B|^{1/q}\|\mathbf{1_B}\|_X^{-1}.
\end{align}

Let $i\in\nn\cap [3,\infty)$  and $\mathcal{R}(S_i(B)):=\cup_{x\in S_i(B)}\Gamma(x).$ Then,
from the Minkowski integral inequality, we deduce that
\begin{align}\label{aiziyou}
\left\|S_1\lf(L^{1/2}a\r)\right\|_{L^q(S_i(B))}^2
&\lesssim\iint_{\mathcal{R}(S_i(B))}
\left|t^mLe^{-t^{2m}L}(a)(y)\right|^2t^{-n-1+\frac{2n}{q}}\,dy\,dt\\\noz
&\lesssim\int_{\rn\setminus2^{i-2}B}\int_0^{2^ir_B}\left|t^{2m}Le^{-t^{2m}L}(a)(y)
\right|^2t^{-n-1-2m+\frac{2n}{q}}\,dt\,dy\\\noz
&\quad+\int_{\rn\setminus2^{i-2}B}\int_{2^ir_B}^\infty\cdots\,dt\,dy
+\int_{2^{i-2}B}\int_{2^{i-2}r_B}^\infty\cdots\,dt\,dy\\\noz
&=:\mathrm{I}+\mathrm{II}+\mathrm{III}.
\end{align}
By Lemma \ref{2.26}, the H\"older inequality, and the assumption that $a$ is an $(\dot{H}_{m,\,X},\,q)$-atom,
we conclude that
\begin{align}\label{fufei}
\mathrm{III}
&\lesssim \|a\|_{L^r(B)}^2\int_{2^{i-2}r_B}^\infty t^{-\frac{2n}{r}-1-2m+\frac{2n}{q}}\,dt
\lesssim (2^ir_B)^{-\frac{2n}{r}-2m+\frac{2n}{q}}|B|^{\frac{2}{r}-\frac{2}{q}}\|a\|_{L^q(B)}^2\\\noz
&\lesssim 2^{-2i(\frac{n}{r}+m-\frac{n}{q})}|B|^{2/q}\|\mathbf{1}_B\|_X^{-2}.
\end{align}
Moreover, using an argument similar to that used in the estimation of \eqref{fufei}, we find that
\begin{align*}
\mathrm{II}\lesssim 2^{-2i(\frac{n}{r}+m-\frac{n}{q})}|B|^{2/q}\|\mathbf{1}_B\|_X^{-2}
\end{align*}
and
\begin{align*}
\mathrm{I}\lesssim \|a\|_{L^r(B)}^2\int_0^{2^jr_B}t^{-\frac{2n}{r}-1-2m+\frac{2n}{q}}
e^{-\frac{(2^jr_B)^2}{ct^2}}\,dt\lesssim 2^{-2i(\frac{n}{r}+m-\frac{n}{q})}|B|^{2/q}\|\mathbf{1}_B\|_X^{-2}.
\end{align*}
From this, \eqref{fufei},  and \eqref{aiziyou}, it follows that \eqref{yida} holds trues.

Similarly to \eqref{meiguihua}, we have
\begin{align*}
\left\|S_1\lf(L^{1/2}f\r)\right\|_{L^2(\rn)}
\lesssim \lf\|L^{1/2}f\r\|_{L^2(\rn)}\lesssim \|\nabla^m f\|_{L^2(\rn)},
\end{align*}
which, together with \eqref{jishu10}, further implies that, for almost every $x\in\rn,$
\begin{align*}
S_1\lf(L^{1/2}f\r)(x)\leq \sum_{j=1}^\infty \ld_jS_1\lf(L^{1/2}a_j\r)(x).
\end{align*}
By this, \eqref{yida}, Lemma \ref{thm:2-2-3}, and \eqref{norm10}, we find that
\begin{align*}
\left\|S_1\lf(L^{1/2}f\r)\right\|_X
\leq
\left\|\sum_{j=1}^\infty \ld_jS_1\lf(L^{1/2}a_j\r)\right\|_X\lesssim\lf\|\lf\{\sum_{j=1}^\fz
\lf(\frac{\lz_j}{\|\mathbf{1}_{B_j}\|_{X}}\r)^s\mathbf{1}_{B_j}\r\}^{1/s}\r\|_X
\lesssim\| f\|_{\dot{H}_{m,\,X}(\rn)}.
\end{align*}
This finishes the proof of \eqref{shediao2}, and hence of Theorem \ref{riesz}.
\end{proof}

\section{Applications}\label{appl}
In this section, we apply Theorems \ref{th2}, \ref{r}, and \ref{riesz}, respectively, to
weighted Hardy spaces associated with $L$ (see Subsection \ref{wls} below),
variable Hardy spaces associated with $L$ (see Subsection \ref{vls} below),
mixed-norm Hardy spaces associated with $L$ (see Subsection \ref{mls} below),
Orlicz--Hardy spaces associated with $L$ (see Subsection \ref{os} below),
Orlicz-slice Hardy spaces associated with $L$ (see Subsection \ref{oss} below),
and Morrey--Hardy spaces associated with $L$ (see Subsection \ref{hms} below).
These applications explicitly indicate the generality and the flexibility
of the main results of this article and more applications
to new function spaces are obviously possible.

\subsection{Weighted Hardy Spaces}\label{wls}

In this subsection, we apply Theorems \ref{th2}, \ref{r}, and \ref{riesz}
to the weighted Hardy space associated with $L$. We begin with recalling the definition of the weighted
Lebesgue space.

\begin{definition}\label{twl}
Let $p\in(0,\infty)$ and $\omega\in A_{\infty}(\rn).$ Then the \emph{weighted Lebesgue space}
$L^p_{\omega}(\rn)$ is defined to be the set of all the measurable functions $f$ on $\rn$ such that
$$\|f\|_{L^p_{\omega}(\mathbb{R}^n)}:=\left	[\int_{\mathbb{R}^n}|f(x)|^p\omega(x)\,dx
\right]^{\frac{1}{p}}<\infty.$$
\end{definition}

We point out that the space $L^p_\omega(\rn)$ with $p\in(0,\infty)$ and $\omega\in A_\infty(\rn)$
is a ball quasi-Banach function space, but it may not be a (quasi-)Banach function space
(see, for instance, \cite[Section 7.1]{SHYY17}).

When $X:=L^p_\omega(\rn)$, the Hardy space $H_{X,\,L}(\rn)$ is just
the \emph{weighted Hardy space associated with $L$}; in this case, we denote $H_{X,\,L}(\rn)$
simply by $H^p_{L,\,\omega}(\rn)$.
Moreover, for any given $s\in(0,1),$ $p\in(s,\infty)$, and $\omega\in A_{p/s}(\rn),$ let
\begin{align}\label{gouqi}
\epsilon_{(p,\,s,\,\omega)}:=\frac{\log([\omega]_{A_{p/s}(\rn)}^{\frac{1}{p/s-1}})
-\log([\omega]_{A_{p/s}(\rn)}^{\frac{1}{p/s-1}}-2^{-p})}{2(n+1)\log2}.
\end{align}
Then, applying Theorem \ref{th2} to the weighted Hardy
space $H^p_{L,\,\omega}(\rn)$, we have the the following conclusion.

\begin{theorem}\label{wls1}
Let $m\in\nn$ and $L$ be a homogeneous divergence form $2m$-order elliptic operator in \eqref{high}
satisfying  Strong Ellipticity Condition \ref{sec}.
Assume that $s\in(0,1],$ $p\in(s,\infty),$ and $\omega\in A_{p/s}(\rn)$ satisfy that
$$p<\frac{p_+(L)\epsilon_{(p,\,s,\,\omega)}}{\epsilon_{(p,\,s,\,\omega)}+1},$$
where $\epsilon_{(p,\,s,\,\omega)}$ is the same as in \eqref{gouqi}.
Then the conclusion of Theorem \ref{th2} holds true with $H_{X,\,L}(\rn)$
replaced by $H^{p}_{L,\,\omega}(\rn).$
\end{theorem}
\begin{proof}
Let all the symbols be the same as in the present theorem. By Theorem \ref{th2}, to show the present theorem,
it suffices to prove that $X:=L^{p}_\omega(\rn)$ satisfies both Assumptions \ref{vector}
and \ref{vector2} for some $\theta\in(0,1],\,s\in(\theta,1],$ and $q\in(p_{-}(L),p_{+}(L)).$

Let $s$ be the same as in the present theorem, $\theta\in(0,s),$  and
$$q\in(\max\{p_-(L),\,p(\epsilon_{(p,\,s,\,\omega)}+1)/\epsilon_{(p,\,s,\,\omega)}\},p_+(L)).$$
From the Fefferman--Stein vector-valued maximal inequality on the weighted Lebesgue space
(see, for instance, \cite[Theorem 3.1(b)]{AJ80}), we deduce that $X:=L^p_\omega(\rn)$
satisfies Assumption \ref{vector} for such a $\theta$ and an $s.$
Moreover, it is easy to show that
\begin{align}\label{sjj}
\left[\left(X^{\frac{1}{s}}\right)'\right]^{\frac{1}{(\frac{q}{s})'}}=L^{(p/s)'/(q/s)'}_{\omega^{1-(p/s)'}}(\rn).
\end{align}
Furthermore, since $\omega\in A_{p/s}(\rn)$, it follows that $\omega^{1-(p/s)'}\in A_{(p/s)'}(\rn).$
By this and the proofs of both \cite[Theorem 7.2.2]{g14} and \cite[Corollary 7.6(1)]{D01}, we conclude
that $\omega^{1-(p/s)'}\in A_{h}(\rn)$, where
$$h:=\frac{(p/s)'+\epsilon_{(p,\,s,\,\omega)}}{1+\epsilon_{(p,\,s,\,\omega)}},$$
which, together with the assumption that $h<(p/s)'/(q/s)'$, further implies that
$\omega^{1-(p/s)'}\in A_{h}(\rn)\subset A_{(p/s)'/(q/s)'}(\rn).$ From this, \eqref{sjj},
and the boundedness of the Hardy--Littlewood maximal operator $\cm$ on the weighted Lebesgue space
(see, for instance, \cite[Theorem 7.1.9]{g14}), we deduce that $X:=L^{p}_\omega(\rn)$ satisfies
Assumption \ref{vector2} for such an $s$ and a $q$.
This finishes the proof of Theorem \ref{wls1}.
\end{proof}

Moreover, by both Theorems \ref{r} and \ref{riesz}, we have the following results; since their proofs are
similar to that of Theorem \ref{wls1}, we omit the details here.

\begin{theorem}\label{wsl2}
Let $m\in\nn$ and $L$ be a homogeneous divergence form $2m$-order elliptic operator  in \eqref{high}
satisfying Ellipticity Condition \ref{ec}.
Assume that $s\in(n/(n+m),1],$ $p\in(s,\infty),$ and $\omega\in A_{p/s}(\rn)$ satisfy that
$$p<\frac{\min\{p_+(L),\,q_+(L)\}\epsilon_{(p,\,s,\,\omega)}}{\epsilon_{(p,\,s,\,\omega)}+1},$$
where $\epsilon_{(p,\,s,\,\omega)}$ is the same as in \eqref{gouqi}.
Then the conclusion of Theorem \ref{r} holds true with $H_{X,\,L}(\rn)$
replaced by $H^{p}_{L,\,\omega}(\rn).$
\end{theorem}

\begin{theorem}\label{wsl3}
Let $m\in\nn$, $L$ be a homogeneous divergence form $2m$-order elliptic operator  in \eqref{high}
satisfying Ellipticity Condition \ref{ec}, and the family $\{e^{-tL}\}_{t\in(0,\infty)}$
of operators satisfy  the  $m-L^r(\rn)-L^2(\rn)$ off-diagonal estimate for some $r\in(1,2].$
Assume that $s\in(nr/(n+mr),1],$ $p\in(s,\infty),$ and $\omega\in A_{p/s}(\rn)$ satisfy
$$p<\frac{p_+(L)\epsilon_{(p,\,s,\,\omega)}}{\epsilon_{(p,\,s,\,\omega)}+1},$$
where $\epsilon_{(p,\,s,\,\omega)}$ is the same as in \eqref{gouqi}.
Then the conclusion of Theorem \ref{riesz} holds true with $H_{X,\,L}(\rn)$
replaced by $H^{p}_{L,\,\omega}(\rn).$
\end{theorem}
\begin{remark}
To the best of our knowledge, the conclusions of Theorems \ref{wls1}, \ref{wsl2}, and
\ref{wsl3} are totally new.
\end{remark}

\subsection{Variable Hardy Spaces}\label{vls}

In this subsection, we apply Theorems \ref{th2}, \ref{r}, and \ref{riesz}
to the variable Hardy space associated with $L$. We first recall the definition of the variable
Lebesgue space.

Let $r:\ \rn\to(0,\infty)$ be a measurable function,
$$
\widetilde{r}_-:=\underset{x\in\rn}{\essinf}\,r(x),\ \ \text{and}\ \ \
\widetilde r_+:=\underset{x\in\rn}{\esssup}\,r(x).
$$
A function $r:\ \rn\to(0,\infty)$ is said to be \emph{globally
log-H\"older continuous} if there exists an $r_{\infty}\in\rr$ and a positive constant $C$
such that, for any $x,y\in\rn$,
$$
|r(x)-r(y)|\le \frac{C}{\log(e+1/|x-y|)}\ \ \ \text{and}\ \ \
|r(x)-r_\infty|\le \frac{C}{\log(e+|x|)}.
$$
The \emph{variable Lebesgue space $L^{r(\cdot)}(\rn)$} associated with the function
$r:\ \rn\to(0,\infty)$ is defined to be the set of all the measurable functions $f$ on $\rn$ with
the finite \emph{quasi-norm}
$$
\|f\|_{L^{r(\cdot)}(\rn)}:=\inf\lf\{\lambda
\in(0,\infty):\ \int_\rn\lf[\frac{|f(x)|}{\lambda}\r]^{r(x)}\,dx\le1\r\}.
$$
Then it is known that $L^{r(\cdot)}(\rn)$ is a ball quasi-Banach function space
(see, for instance, \cite[Section 7.8]{SHYY17}). In particular, when $1<\widetilde r_-\le \widetilde
r_+<\infty$, $(L^{r(\cdot)}(\rn), \|\cdot\|_{L^{r(\cdot)}(\rn)})$ is a Banach function space in the
terminology of Bennett and Sharpley \cite{BS88} and hence also a ball Banach function space
(see, for instance, \cite[p.\,94]{SHYY17}). More results on variable Lebesgue spaces can be found
in \cite{N1,N2,NS12,CUF,CUW,DHR} and the references therein.

When $X:=L^{r(\cdot)}(\rn)$, the Hardy space $H_{X,\,L}(\rn)$ is just
the \emph{variable Hardy space associated with $L$}; in this case, we denote $H_{X,\,L}(\rn)$
simply by $H^{r(\cdot)}_{L}(\rn)$.
Then, applying Theorem \ref{th2} to
the variable Hardy space $H^{r(\cdot)}_{L}(\rn)$, we have the following conclusion.

\begin{theorem}\label{vls1}
Let $m\in\nn$ and $L$ be a homogeneous divergence form $2m$-order elliptic operator
in \eqref{high} satisfying  Strong Ellipticity Condition \ref{sec}.
Assume that $r:\ \rn\to(0,\infty)$ is globally log-H\"older continuous and $0< \wt r_{-}\leq
\wt r_{+}<p_{+}(L)$. Then the conclusion of Theorem \ref{th2} holds true with $H_{X,\,L}(\rn)$ replaced
by $H^{r(\cdot)}_{L}(\rn).$
\end{theorem}

\begin{proof}
Let all the symbols be the same as in the present theorem. By Theorem \ref{th2},
to prove the present theorem, it suffices to show that $X:=L^{r(\cdot)}(\rn)$
satisfies both Assumptions \ref{vector} and \ref{vector2}
for some $\theta\in(0,1],\,s\in(\theta,1],$ and $q\in(p_{-}(L),p_{+}(L)).$

Let $\theta\in(0,\wt r_{-})$,
$s\in\lf(\theta,\min\lf\{1,\,\wt r_{-}\r\}\r)$, and
$q\in\lf(\max\lf\{1,\,\wt r_+,\,p_{-}(L)\r\},p_+(L)\r).$
Then, from  the Fefferman--Stein vector-valued  maximal inequality on  variable Lebesgue spaces
(see, for instance, \cite[Lemma 2.4]{NS12}), we deduce that $X:=L^{r(\cdot)}(\rn)$ satisfies
Assumptions \ref{vector} for such a $\theta$ and an $s$. Moreover, by the dual result on variable Lebesgue
spaces (see, for instance, \cite[Theorem 2.80]{CUF}), we conclude that
\begin{align*}
\left[\left(X^{\frac{1}{s}}\right)'\right]^{\frac{1}{(\frac{q}{s})'}}=L^{(r(\cdot)/s)'/(q/s)'}(\rn),
\end{align*}
which, together with the assumption that $q\in(\max\{1,\,\wt r_+,\,p_{-}(L)\},p_+(L))$ and the
boundedness of the Hardy--Littlewood maximal operator $\cm$ on variable Lebesgue spaces (see,
for instance, \cite[Theorem 3.16]{CUF}), further implies that $X:=L^{r(\cdot)}(\rn)$ satisfies
Assumption \ref{vector2} for such an $s$ and a $q$. This finishes the proof of Theorem \ref{vls1}.
\end{proof}

By both Theorems \ref{r} and \ref{riesz}, we have the following conclusions; since their proofs are similar
to that of Theorem \ref{vls1}, we omit the details here.

\begin{theorem}\label{vsl2}
Let $m\in\nn$, $L$ be a homogeneous divergence form $2m$-order elliptic operator  in \eqref{high}
satisfying Ellipticity Condition \ref{ec}, and $r:\ \rn\to(0,\infty)$ be
globally log-H\"older continuous. Assume that $\frac{n}{n+m}< \wt r_{-}\leq \wt r_{+}<
\min\{p_{+}(L),q_{+}(L)\}$. Then the conclusion of Theorem \ref{r} holds true with
$H_{X,\,L}(\rn)$ replaced by $H^{r(\cdot)}_{L}(\rn).$
\end{theorem}

\begin{theorem}\label{vsl3}
Let $m\in\nn$, $L$ be a homogeneous divergence form $2m$-order elliptic operator  in \eqref{high}
satisfying Ellipticity Condition \ref{ec}, and the family $\{e^{-tL}\}_{t\in(0,\infty)}$
of operators satisfy the  $m-L^r(\rn)-L^2(\rn)$ off-diagonal estimate with some $r\in(1,2].$ Let $r:\
\rn\to(0,\infty)$ be globally log-H\"older continuous. Assume that $\frac{nr}{n+mr}<\wt r_{-}\leq
\wt r_{+}<p_{+}(L)$. Then the conclusion of Theorem \ref{riesz} holds true with
$H_{X,\,L}(\rn)$ replaced by $H^{r(\cdot)}_{L}(\rn).$
\end{theorem}

\begin{remark}
When $m:=1$ and $\wt r_{+}\in(0,1]$, Theorem \ref{vls1} is just \cite[Theorem 5.3]{YZZ18}.
Meanwhile, when $m:=1$ and $\frac{n}{n+1}<\wt r_{-}\leq \wt r_{+}\leq1$,
Theorem \ref{vsl2} is just \cite[Theorem 5.17]{YZZ18}. Moreover, to the best of our knowledge,
Theorem \ref{vsl3} is totally new even when $m:=1$.
\end{remark}

\subsection{Mixed-norm Hardy Spaces}\label{mls}

In this subsection, we apply Theorems \ref{th2}, \ref{r}, and \ref{riesz}
to the mixed-norm Hardy space associated with $L$.
We begin with recalling the definition of the mixed-norm Lebesgue space.

For a given vector $\vec{r}:=(r_1,\ldots,r_n)\in(0,\infty]^n$, the \emph{mixed-norm Lebesgue
space $L^{\vec{r}}(\rn)$} is defined to be the set of all the measurable functions $f$ on
$\rn$ with the finite \emph{quasi-norm}
$$
\|f\|_{L^{\vec{r}}(\rn)}:=\lf\{\int_{\rr}
\cdots\lf[\int_{\rr}|f(x_1,\ldots,
x_n)|^{r_1}\,dx_1\r]^{\frac{r_2}{r_1}}\cdots\,dx_n\r\}^{\frac{1}{r_n}},
$$
where the usual modifications are made when $r_i:=\infty$ for some $i\in\{1,\ldots,n\}$.
Here and in the remainder of this subsection, let
$r_-:=\min\{r_1, \ldots , r_n\}$ and
$r_+:=\max\{r_1, \ldots , r_n\}$.

It is worth pointing out that the space $L^{\vec{r}}(\rn)$ with $\vec{r}\in(0,\infty)^n$ is a ball
quasi-Banach function space; but $L^{\vec{r}}(\rn)$ with $\vec{r}\in[1,\infty]^n$ may not be a Banach
function space (see, for instance, \cite[Remark 7.21]{ZWYY20}).
The study of mixed-norm Lebesgue spaces can be traced back to H\"ormander \cite{H1} and Benedek
and Panzone \cite{BP}. More results on mixed-norm Lebesgue spaces and other mixed-norm
function spaces can be found in \cite{l70,CGN19,CGN17,cgn17bs,tn,noss20,HLYY} and
the references therein.

In particular, when $X:=L^{\vec{r}}(\rn)$, the Hardy space $H_{X,\,L}(\rn)$ is just
the \emph{mixed-norm Hardy space associated with $L$}; in this case, we denote $H_{X,\,L}(\rn)$
simply by $H^{\vec{r}}_{L}(\rn)$. Then, applying Theorem \ref{th2}
to the mixed-norm Hardy space $H^{\vec{r}}_{L}(\rn)$, we have the following conclusion.

\begin{theorem}\label{mls1}
Let $m\in\nn$ and $L$ be a homogeneous divergence form $2m$-order elliptic operator  in \eqref{high}
satisfying  Strong Ellipticity Condition \ref{sec}. Assume that $\vec{r}:=(r_1,\ldots,r_n)
\in(0,\infty)^n$ with $r_+\in(0,p_+(L)).$ Then the conclusion of Theorem \ref{th2} holds true
with $H_{X,\,L}(\rn)$ replaced by $H^{\vec{r}}_L(\rn).$
\end{theorem}

\begin{proof}
Let all the symbols be same as in the present theorem. By Theorem \ref{th2}, to show the present theorem,
it suffices to prove that $X:=L^{\vec{r}}(\rn)$ satisfies both Assumptions \ref{vector} and
\ref{vector2} for some $\theta\in(0,1],\,s\in(\theta,1],$ and $q\in(p_{-}(L),p_{+}(L)).$

Let $\theta\in(0, r_{-})$,
$s\in(\theta,\min\{1,\,r_{-}\})$,
and
$q\in(\max\{1,\,r_+,\,p_{-}(L)\},p_+(L)).$
Then, by the Fefferman--Stein vector-valued maximal inequality on mixed-norm Lebesgue spaces
(see, for instance, \cite[Lemma 3.7]{HLYY}), we find that $X:=L^{\vec{r}}(\rn)$ satisfies
Assumptions \ref{vector} for such a $\theta$ and an $s$. Furthermore, from the dual result of mixed-norm
Lebesgue spaces (see, for instance, \cite[Theorems 1 and 2]{BP}), it is easy to follow that,
when $X:=L^{\vec{r}}(\rn)$,
\begin{align}\label{ta}
\left[\left(X^{\frac{1}{s}}\right)'\right]^{\frac{1}{(\frac{q}{s})'}}=L^{(\vec{r}/s)'/(q/s)'}(\rn),
\end{align}
where $(\vec{r}/s)':=((r_1/s)',\ldots,(r_n/s)')$ and $\frac{1}{(r_i/s)'}+\frac{1}{r_i/s}=1$
for any $i\in\{1,\ldots,n\}.$ Moreover, it is well known that the Hardy--Littlewood maximal
operator $\cm$ is bounded on the  mixed-norm Lebesgue space $L^{\vec{r}}(\rn)$
when $\vec{r}\in(1,\infty)^n$ (see, for instance, \cite[Lemma 3.5]{HLYY}). From this, the assumption that
$q\in(\max\{1,\, r_+,\,p_{-}(L)\},p_+(L)),$ and \eqref{ta}, it follows that $X:=L^{\vec{r}}(\rn)$
satisfies Assumption \ref{vector2} for such an $s$ and a $q$.
This finishes the proof of Theorem \ref{mls1}.
\end{proof}

By both Theorems \ref{r} and \ref{riesz}, we have the following results; since their  proofs are similar
to that of Theorem \ref{mls1}, we omit the details here.

\begin{theorem}\label{mls2}
Let $m\in\nn$ and $L$ be a homogeneous divergence form $2m$-order elliptic operator
in \eqref{high} satisfying Ellipticity Condition \ref{ec}.
Assume that $\vec{r}:=(r_1,\ldots,r_n)\in(0,\infty)^n$ satisfies   $\frac{n}{n+m}<
r_{-}\leq  r_{+}<\min\{p_{+}(L),q_{+}(L)\}$. Then the conclusion of Theorem \ref{r}
holds true with $H_{X,\,L}(\rn)$ replaced by $H^{\vec{r}}_L(\rn).$
\end{theorem}

\begin{theorem}\label{mls3}
Let $m\in\nn$, $L$ be a homogeneous divergence form $2m$-order elliptic operator  in \eqref{high}
satisfying Ellipticity Condition \ref{ec}, and the family
$\{e^{-tL}\}_{t\in(0,\infty)}$ of operators satisfy the  $m-L^r(\rn)-L^2(\rn)$ off-diagonal estimate
for some $r\in(1,2].$ Assume that $\vec{r}:=(r_1,\ldots,r_n)\in(0,\infty)^n$ satisfies $\frac{nr}{n+mr}<
r_{-}\leq  r_{+}<p_{+}(L)$. Then the conclusion of Theorem \ref{riesz} holds true with
$H_{X,\,L}(\rn)$ replaced by $H^{\vec{r}}_L(\rn).$
\end{theorem}

\begin{remark}
To the best of our knowledge, Theorems \ref{mls1}, \ref{mls2}, and \ref{mls3}  are totally new.
\end{remark}

\subsection{Orlicz--Hardy Spaces}\label{os}

In this subsection, we apply Theorems \ref{th2}, \ref{r}, and \ref{riesz}
to the Orlicz--Hardy space associated with $L$.
We begin with recalling the definition of the Orlicz function.

A non-decreasing function $\Phi:\ [0,\infty)\ \to\ [0,\infty)$ is called an \emph{Orlicz function}
if $\Phi(0)= 0$, $\Phi(t)>0$ for any $t\in(0,\infty)$, and $\lim_{t\to\infty}\Phi(t)=\infty$.
Moreover, an Orlicz function $\Phi$ is said to be of \emph{lower} [resp., \emph{upper}] \emph{type}
$r$ for some $r\in\rr$ if there exists a positive constant $C_{(r)}$ such that, for any
$t\in[0,\infty)$ and $s\in(0,1)$ [resp., $s\in[1,\infty)$],
$$\Phi(st)\le C_{(r)} s^r
\Phi(t).$$

In the remainder of this subsection, we \emph{always} assume that $\Phi:\ [0,\infty)\ \to\ [0,\infty)$
is an Orlicz function with positive lower type $r_{\Phi}^-$ and positive upper type $r_{\Phi}^+$.
The \emph{Orlicz norm} $\|f\|_{L^\Phi(\rn)}$ of a measurable function $f$ on $\rn$ is then
defined by setting
$$\|f\|_{L^\Phi(\rn)}:=\inf\lf\{\lambda\in(0,\infty):\ \int_{\rn}\Phi\lf(\frac{|f(x)|}
{\lambda}\r)\,dx\le1\r\}.$$
Then the \emph{Orlicz space $L^\Phi(\rn)$} is defined to be the set of all the measurable functions
$f$ on $\rn$ with finite norm $\|f\|_{L^\Phi(\rn)}$.

It is easy to prove that the Orlicz space $L^\Phi(\rn)$ is a ball quasi-Banach function space
(see, for instance, \cite[Section 7.6]{SHYY17}). In particular, if $1\leq r_{\Phi}^-\leq r_{\Phi}^+<\infty$,
then $L^\Phi(\rn)$ is a Banach function space (see \cite[p.\,67, Theorem 10]{rr91}), and hence is also
a ball Banach function space. For more results on Orlicz spaces, we refer the reader to
\cite{km91,RR02,ZWYY20} and the references therein.

In particular, when $X:=L^{\Phi}(\rn)$, the Hardy space $H_{X,\,L}(\rn)$ is just
the \emph{Orlicz--Hardy space associated with $L$}; in this case, we denote $H_{X,\,L}(\rn)$
simply by $H_{\Phi,\,L}(\rn)$. Then, applying Theorem \ref{th2}
to the Orlicz--Hardy space $H_{\Phi,\,L}(\rn)$, we have the following conclusion.

\begin{theorem}\label{os1}
Let $m\in\nn$ and $L$ be a homogeneous divergence form $2m$-order elliptic operator  in \eqref{high}
satisfying  Strong Ellipticity Condition \ref{sec}. Assume that $\Phi$ is an Orlicz function
with positive lower type $r_{\Phi}^-$ and positive upper type $r_{\Phi}^+\in(0,p_{+}(L))$.
Then the conclusion of Theorem \ref{th2} holds true with $H_{X,\,L}(\rn)$
replaced by $H_{\Phi,\,L}(\rn).$
\end{theorem}
\begin{proof}
Let all the symbols be the same as in the present theorem. By Theorem \ref{th2}, to prove the present theorem,
it suffices to show that $X:=L^\Phi(\rn)$ satisfies both Assumptions \ref{vector} and \ref{vector2}
for some $\theta\in(0,1],\,s\in(\theta,1],$ and $q\in(p_{-}(L),p_{+}(L)).$

Let $\theta\in(0, r_{\Phi}^-),\,s\in(\theta,\min\{1,\,r_{\Phi}^-\}),$ and $q\in(\max\{1,\,
r_{\Phi}^+,\,p_{-}(L)\},p_+(L)).$ Then, from \cite[Theorem 7.14(i)]{SHYY17}, we deduce that
$X:=L^{\Phi}(\rn)$ satisfies Assumption \ref{vector} for such a $\theta$ and an $s$.
Furthermore, by the dual theorem on Orlicz spaces (see, for instance, \cite[Theorem 13]{RR02}),
we find that, when $X:=L^{\Phi}(\rn)$,
\begin{align}\label{sa}
\left[\left(X^{\frac{1}{s}}\right)'\right]^{\frac{1}{(\frac{q}{s})'}}=L^{\Psi}(\rn),
\end{align}
where, for any $t\in[0,\fz)$,
$$\Psi(t):=\sup_{h\in(0,\infty)}\lf[t^{1/(q/s)'}h-\Phi\lf(h^{1/s}\r)\r].$$
Then, from \cite[Proposition 7.8]{SHYY17} and the assumption that $q\in(\max\{1,\,
r_{\Phi}^+,\,p_{-}(L)\},p_+(L))$, it follows that $\Psi$ is an Orlicz function with positive
lower type $r_{\Psi}^-:=(r_{\Phi}^+/s)'/(q/s)'\in(1,\infty)$, which, combined with both \eqref{sa}
and the boundedness of the Hardy--Littlewood maximal operator $\cm$ on Orlicz spaces
(see, for instance, \cite[Theorem 7.12]{SHYY17}), further implies that
$X:=L^{\Phi}(\rn)$ satisfies Assumption \ref{vector2} for such an $s$ and a $q$.
This finishes the proof of Theorem \ref{os1}.
\end{proof}

Moreover, by both Theorems \ref{r} and \ref{riesz}, we have the following conclusions; since their  proofs are
similar to that of Theorem \ref{os1}, we omit the details here.

\begin{theorem}\label{os2}
Let $m\in\nn$ and $L$ be a homogeneous divergence form $2m$-order elliptic operator  in
\eqref{high} satisfying Ellipticity Condition  \ref{ec}.
Assume that $\Phi$ is an Orlicz function with positive lower type $r_{\Phi}^-$ and
positive upper type $r_{\Phi}^+$. Assume further that $\frac{n}{n+m}<r_{\Phi}^-\leq r_{\Phi}^+
<\min\{p_+(L),q_+(L)\}$. Then the conclusion of Theorem \ref{r} holds true with $H_{X,\,L}(\rn)$
replaced by $H_{\Phi,\,L}(\rn).$
\end{theorem}

\begin{theorem}\label{os3}
Let $m\in\nn$, $L$ be a homogeneous divergence form $2m$-order elliptic operator  in
\eqref{high} satisfying Ellipticity Condition  \ref{ec}, and the family
$\{e^{-tL}\}_{t\in(0,\infty)}$ of operators satisfy the  $m-L^r(\rn)-L^2(\rn)$ off-diagonal estimate
for some $r\in(1,2].$ Assume that $\Phi$ is an Orlicz function with positive lower type $r_{\Phi}^-$ and
positive upper type $r_{\Phi}^+$. Assume further that $\frac{nr}{n+mr}<r_{\Phi}^-\leq r_{\Phi}^+<p_+(L).$
Then the conclusion of Theorem \ref{riesz} holds true with $H_{X,\,L}(\rn)$
replaced by $H_{\Phi,\,L}(\rn).$
\end{theorem}

\begin{remark}
When $m:=1$ and $r_\Phi^+\in(0,1]$, Theorem \ref{os1} was obtained in \cite[Theorem 5.2]{JY10}.
Meanwhile, when $m:=1$ and $\frac{n}{n+1}<r_{\Phi}^-\leq r_{\Phi}^+\leq1$,
Theorem \ref{oss2} is just \cite[Theorem 7.4]{JY10}. Furthermore, to the best of our knowledge,
Theorem \ref{os3} is totally new even when $m:=1$.
\end{remark}

\subsection{Orlicz-slice Hardy Spaces}\label{oss}
In this subsection, we apply Theorems \ref{th2}, \ref{r}, and \ref{riesz}
to the Orlicz-slice Hardy space associated with $L$.
We first recall the definitions of Orlicz-slice spaces and then describe briefly
some related facts. Throughout this subsection, we \emph{always} assume that
$\Phi: [0,\infty)\to [0,\infty)$ is an Orlicz function with positive lower type $r_{\Phi}^-$
and positive upper type $r_{\Phi}^+$. For any given $t, r\in(0,\infty)$, the \emph{Orlicz-slice space}
$(E_\Phi^r)_t(\rn)$ is defined to be the set of all the measurable functions $f$ on $\rn$ with the finite
\emph{quasi-norm}
$$
\|f\|_{(E_\Phi^r)_t(\rn)} :=\lf\{\int_{\rn}\lf[\frac{\|f\mathbf{1}_{B(x,t)}\|_{L^\Phi(\rn)}}
{\|\mathbf{1}_{B(x,t)}\|_{L^\Phi(\rn)}}\r]^r\,dx\r\}^{\frac{1}{r}}.
$$

Orlicz-slice spaces were introduced in \cite{ZYYW} as a generalization of slice spaces
in \cite{AM2014,APA} and Wiener amalgam spaces in \cite{h75,knt,h19}.
Meanwhile, by \cite[Lemma 2.28]{ZYYW} and \cite[Remark 7.41(i)]{ZWYY20},
we find that the Orlicz-slice space $(E_\Phi^r)_t(\rn)$ is a ball Banach function space,
but in general is not a Banach function space.

In particular, when $X:=(E_\Phi^r)_t(\rn)$, the Hardy space $H_{X,\,L}(\rn)$ is just
the \emph{Orlicz-slice Hardy space associated with $L$}; in this case, we denote $H_{X,\,L}(\rn)$
simply by $(HE_{\Phi,\,L}^r)_t(\rn)$. Then, applying Theorem \ref{th2}
to the Orlicz-slice Hardy space $(HE_{\Phi,\,L}^r)_t(\rn)$, we have the following conclusion.

\begin{theorem}\label{oss1}
Let $m\in\nn$ and $L$ be a homogeneous divergence form $2m$-order elliptic operator
in \eqref{high} satisfying  Strong Ellipticity Condition \ref{sec}.
Assume that $t\in(0,\infty),\,r\in(0,p_+(L)),$ and $\Phi$ is an Orlicz function with
positive lower type $r_{\Phi}^-$ and positive upper type $r_{\Phi}^+\in(0,p_+(L))$.
Then the conclusion of Theorem \ref{th2} holds true with $H_{X,\,L}(\rn)$
replaced by $(HE_{\Phi,\,L}^r)_t(\rn)$.
\end{theorem}

\begin{proof}
Let all the symbols be the same as in the present theorem. Assume further that $\theta\in(0,1),$
$s\in(\theta,\min\{1,\,r_\Phi^-,\,r\}),$ and $q\in(\max\{r,\,r_\Phi^+,\,p_{-}(L)\},\,p_+(L)).$
It was proved in \cite[Lemmas 4.3 and 4.4]{ZYYW} that $X:=(E_\Phi^r)_t(\rn)$ satisfies
both Assumptions \ref{vector} and \ref{vector2} for such a $\theta,$ an $s,$ and a $q.$ From this and Theorem
\ref{th2}, it follows that the conclusion of Theorem \ref{oss1} holds true.
\end{proof}

Moreover, by both Theorems \ref{r} and \ref{riesz}, we have the following results; since their  proofs are
similar to that of Theorem \ref{oss1}, we omit the details here.

\begin{theorem}\label{oss2}
Let $m\in\nn$ and $L$ be a homogeneous divergence form $2m$-order elliptic operator  in
\eqref{high} satisfying   Ellipticity Condition \ref{ec}.
Assume that $t\in(0,\infty)$,
$$r\in\lf(\frac{n}{n+m},\min\lf\{p_+(L),\,q_+(L)\r\}\r),$$
and $\Phi$ is an Orlicz function with positive lower type $r_{\Phi}^-\in(\frac{n}{n+m},p_+(L))$
and positive upper type $r_{\Phi}^+\in(0,\min\{p_+(L),\,q_+(L)\})$.
Then the conclusion of Theorem \ref{r} holds true with $H_{X,\,L}(\rn)$
replaced by $(HE_{\Phi,\,L}^r)_t(\rn)$.
\end{theorem}
\begin{theorem}\label{oss3}
Let $m\in\nn$, $L$ be a homogeneous divergence form $2m$-order elliptic operator  in
\eqref{high} satisfying   Ellipticity Condition \ref{ec}, and the family
$\{e^{-tL}\}_{t\in(0,\infty)}$ of operators satisfy the $m-L^r(\rn)-L^2(\rn)$ off-diagonal estimate
for some $r\in(1,2].$ Assume that $t\in(0,\infty),\,r\in(\frac{n}{n+m},p_+(L)),$ and
$\Phi$ is an Orlicz function with positive lower type $r_{\Phi}^-\in(\frac{nr}{n+mr},p_+(L))$
and positive upper type $r_{\Phi}^+\in(0,p_+(L))$.
Then the conclusion of Theorem \ref{riesz} holds true with $H_{X,\,L}(\rn)$
replaced by $(HE_{\Phi,\,L}^r)_t(\rn)$.
\end{theorem}
\begin{remark}
To the best of our knowledge, Theorems \ref{oss1}, \ref{oss2}, and \ref{oss3} are
totally new even when $m:=1.$
\end{remark}
\subsection{Morrey--Hardy Spaces}\label{hms}
In this subsection, we apply Theorems \ref{th2}, \ref{r}, and \ref{riesz}
to the Morrey--Hardy space associated with $L$.
We begin with recalling the definition of the Morrey space.

Let $0<r\le p\le\fz$. Recall that the \emph{Morrey space} ${\mathcal M}^p_r(\rn)$ is defined
to be the set of all the $f\in L^r_{{\rm loc}}(\rn)$ such that
\begin{equation*}
\|f\|_{{\mathcal M}^p_r(\rn)}:=\sup_{B\subset\rn}
|B|^{\frac1p-\frac1r}\left[\int_{B}|f(y)|^r\,dy\right]^{\frac1r}<\fz,
\end{equation*}
where the supremum is taken over all balls $B\subset\rn$.

The space ${\mathcal M}^p_r(\rn)$ was introduced by Morrey \cite{M38}. It is known that
${\mathcal M}^p_r(\rn)$ with $1\le r < p <\fz$ is not a Banach function space,
but is a ball Banach function space (see, for instance, \cite[Section 7.4]{SHYY17}).
Moreover, by the definition of ${\mathcal M}^p_r(\rn)$, it is easy to find that
${\mathcal M}^p_r(\rn)$ with $0<r\le p\le\fz$ is a ball quasi-Banach space.

In particular, when $X:={\mathcal M}^p_r(\rn)$, the Hardy space $H_{X,\,L}(\rn)$ is just
the \emph{Morrey--Hardy space associated with $L$}; in this case, we denote $H_{X,\,L}(\rn)$
simply by $H{\mathcal M}^p_{r,\,L}(\rn)$. Then, applying Theorem \ref{th2}
to the Morrey--Hardy space $H{\mathcal M}^p_{r,\,L}(\rn)$, we have the following conclusion.

\begin{theorem}\label{hms1}
Let $m\in\nn$ and $L$ be a homogeneous divergence form $2m$-order elliptic operator
in \eqref{high} satisfying  Strong Ellipticity Condition \ref{sec}.
Assume that $p\in(0,p_+(L))$ and $r\in(0,p]$. Then the conclusion of Theorem \ref{th2} holds
true with $H_{X,\,L}(\rn)$ replaced by $H{\mathcal M}^p_{r,\,L}(\rn)$.
\end{theorem}

\begin{proof}
Let all the symbols be the same as in the present theorem. Assume further that $\theta\in(0,1),$
$s\in(\theta,\min\{1,\,r\}),$ and $q\in(\max\{p,\,p_{-}(L)\},\,p_+(L)).$
Then it is known that $X:={\mathcal M}^p_r(\rn)$ satisfies both Assumptions \ref{vector} and
\ref{vector2} for such a $\theta,$ an $s,$ and a $q$ (see, for instance, \cite[Remarks 2.4(e) and 2.7(e)]{WYY20}).
By this and Theorem \ref{th2}, we find that the conclusion of Theorem \ref{hms1} holds true.
\end{proof}

Moreover, by both Theorems \ref{r} and \ref{riesz}, we have the following results; since their  proofs are
similar to that of Theorem \ref{hms1}, we omit the details here.

\begin{theorem}\label{hms2}
Let $m\in\nn$ and $L$ be a homogeneous divergence form $2m$-order elliptic operator  in
\eqref{high} satisfying  Ellipticity Condition \ref{ec}.
Assume that $p\in(\frac{n}{n+m},\min\{p_+(L),\,q_+(L)\})$ and $r\in(\frac{n}{n+m},p]$.
Then the conclusion of Theorem \ref{r} holds true with $H_{X,\,L}(\rn)$
replaced by $H{\mathcal M}^p_{r,\,L}(\rn)$.
\end{theorem}
\begin{theorem}\label{hms3}
Let $m\in\nn$, $L$ be a homogeneous divergence form $2m$-order elliptic operator  in
\eqref{high} satisfying   Ellipticity Condition \ref{ec}, and the family
$\{e^{-tL}\}_{t\in(0,\infty)}$ of operators satisfy the $m-L^s(\rn)-L^2(\rn)$ off-diagonal estimate
for some $s\in(1,2].$ Assume that $p\in(\frac{ns}{n+ms},p_+(L))$ and $r\in(\frac{ns}{n+ms},p]$.
Then the conclusion of Theorem \ref{riesz} holds true with $H_{X,\,L}(\rn)$
replaced by $H{\mathcal M}^p_{r,\,L}(\rn)$.
\end{theorem}

\begin{remark}
To the best of our knowledge, Theorems \ref{hms1}, \ref{hms2}, and \ref{hms3} are totally new
even when $m:=1$.
\end{remark}

\bigskip

\noindent Xiaosheng Lin, Dachun Yang (Corresponding author) and Wen Yuan

\medskip

\noindent Laboratory of Mathematics and Complex Systems (Ministry of Education of China),
School of Mathematical Sciences, Beijing Normal University, Beijing 100875,
The People's Republic of China

\smallskip

\smallskip

\noindent {\it E-mails}: \texttt{xiaoslin@mail.bnu.edu.cn} (X. Lin)

\noindent\phantom{{\it E-mails:} }\texttt{dcyang@bnu.edu.cn} (D. Yang)

\noindent\phantom{{\it E-mails:} }\texttt{wenyuan@bnu.edu.cn} (W. Yuan)

\bigskip

\noindent Sibei Yang

\medskip

\noindent School of Mathematics and Statistics, Gansu Key Laboratory of Applied Mathematics
and Complex Systems, Lanzhou University, Lanzhou 730000, The People's Republic of China

\smallskip

\noindent{\it E-mail:} \texttt{yangsb@lzu.edu.cn}

\end{document}